\documentclass[11pt]{amsart}

\newcommand{\pathtotrunk}{./}

\usepackage{microtype}
\usepackage{enumerate}
\usepackage{tikz}
\usetikzlibrary{calc}
\usetikzlibrary{decorations.pathreplacing}
\pgfdeclarelayer{background}
\pgfdeclarelayer{foreground}
\pgfsetlayers{background,main,foreground}

\tikzstyle{shaded}=[fill=red!10!blue!20!gray!50!white]
\tikzstyle{shaded line}=[double=red!10!blue!20!gray!50!white, double distance=2mm, draw=black]
\tikzstyle{unshaded}=[fill=white]
\tikzstyle{unshaded line}=[double=white, double distance=2mm, draw=black]
\tikzstyle{Tbox}=[circle, draw, thick, fill=white, opaque,]
\tikzstyle{empty box}=[circle, draw, thick, fill=white, opaque, inner sep=2mm]
\tikzstyle{background rectangle}= [fill=red!10!blue!20!gray!40!white,rounded corners=2mm] 
\tikzstyle{on}=[very thick, red!50!blue!50!black]
\tikzstyle{off}=[gray]
\tikzstyle{traces}=[scale=.2, inner sep=1mm,baseline]
\tikzstyle{quadratic}=[scale=.23, inner sep=.5mm, baseline]
\tikzstyle{annular}=[scale=.6, inner sep=1mm, baseline]
\tikzstyle{make triple edge size}= [scale=.4, inner sep=1mm,baseline] 
\tikzstyle{icosahedron network}=[scale=.35, inner sep=1mm, baseline]
\tikzstyle{ATLsix}=[scale=.25, baseline]
\tikzstyle{TL12}=[scale=.15,baseline=-0.5ex]
\tikzstyle{TLEG}=[scale=.5,baseline]
\tikzstyle{PAdefn}=[scale=.7,baseline]
\tikzstyle{STrain}=[baseline=0,scale=2]%
\newcommand{\upsidedown}[1]{\begin{scope}[y=-1cm] #1 \end{scope}}

\newcommand{\drawS}[3]{%
	\filldraw[fill=white,thick] (#1,#2) ellipse (3mm and 3mm);
	\node at (#1,#2) {\Large $S$};
	\path(#1,#2) ++(#3:0.37) node {$\star$};
}

\newcommand{\RainbowOne}{
	\fill[shaded] (-0.8,0) -- (-0.8,0.6) arc (180:0:0.8) -- (0.8,0) -- (0.2,0) -- (0.2,1) -- (-0.2,1) -- (-0.2,0);
	\draw (-0.8,0) -- (-0.8,0.6) arc (180:0:0.8) -- (0.8,0);
	\draw (-0.2,0) -- (-0.2,1);
	\draw (0.2,0) -- (0.2,1);
	\node at (0.45,0.5) {\footnotesize$2n-1$};
	\drawS{0}{1}{-90}
}

\newcommand{\RainbowTwo}{
	\draw (0,0) -- (0,1);
	\node at (0.15,0.5) {\footnotesize$2n$};
	\drawS{0}{1}{180}
	\filldraw[shaded] (-1,0) -- (-1,0.5) arc (180:0:1) -- (1,0)  -- (0.9,0) -- (0.9,0.5) arc (0:180:0.9) -- (-0.9,0);
}

\newcommand{\JWPlusTwo}{%
	\filldraw[fill=white,thick] (-1,-0.2) rectangle (1,0.2);
	\node at (0,0) {\Large$\JW{2n+2}$};
}

\newcommand{\JWPlusFour}{%
	\filldraw[fill=white,thick] (-1.2,-0.2) rectangle (1.2,0.2);
	\node at (0,0) {\Large$\JW{2n+4}$};
}

\newcommand{\STrainOneOne}{%
	\foreach \x in {-1,0} {
		\node[anchor=south] at (\x+0.5,1) {\footnotesize$n-1$};
		\draw (\x,1) -- (\x + 1, 1);
	}
	\foreach \x in {-1,0,1} {
		\drawS{\x}{1}{90}
	}
}

\newcommand{\STrainStrings}[2]{%
	\fill[shaded] (-0.5,0) rectangle (0.5,1);
	\draw (-0.5,1) -- (-0.5,0);
	\node[anchor=west] at (-0.5,0.5) {\footnotesize#1};
	\draw (0.5,1) -- (0.5,0);
	\node[anchor=west] at (0.5,0.5) {\footnotesize#2};
}

\newcommand{\STrainOne}{%
	\node[anchor=south] at (0,1) {\footnotesize$n-1$};
	\draw (-0.5,1) -- (0.5, 1);
	\foreach \x in {-0.5,0.5} {
		\drawS{\x}{1}{90}
	}
}

\newcommand{\STrainThreeStrings}[3]{%
	\fill[shaded] (-1,1) rectangle (1,0);
	\foreach \x in {-1,0,1} {
		\draw (\x,0) -- (\x,1);
	}
        \node [anchor=west] at (-1,0.5) {\footnotesize#1};
        \node [anchor=west] at (0,0.5) {\footnotesize#2};
	\node [anchor=west] at (1,0.5) {\footnotesize#3};
}
\usepackage[english]{babel}
\usepackage{amsthm,amssymb}
\usepackage{latexsym}
\usepackage{amsmath} 
\usepackage{graphics}
\usepackage{amsfonts}
\usepackage[latin1]{inputenc}
\usepackage{verbatim}
\usepackage{array}
\usepackage{xcolor}
\usepackage{enumitem}

\usepackage[pdftex,plainpages=false,hypertexnames=false,pdfpagelabels]{hyperref}

\def\Hom{{\operatorname{Hom}}}

\def\mod{{\hbox{mod}}}

\def\tr{\hbox{\rm tr}}

\def\End{\hbox{\rm End}}

\def\iso{\cong}

\newcommand{\tensor}{\otimes}

\newcommand{\eset}{\emptyset}

\newcommand{\II}{\mathrm{II}}
\newcommand{\III}{\mathrm{III}}

\newcommand{\Integer}{\mathbb Z}

\newcommand{\Real}{\mathbb R}
\newcommand{\Complex}{\mathbb C}

\def\semicolon{;}
\def\applytolist#1{
    \expandafter\def\csname multi#1\endcsname##1{
        \def\multiack{##1}\ifx\multiack\semicolon
            \def\next{\relax}
        \else
            \csname #1\endcsname{##1}
            \def\next{\csname multi#1\endcsname}
        \fi
        \next}
    \csname multi#1\endcsname}

\def\calc#1{\expandafter\def\csname c#1\endcsname{{\mathcal #1}}}
\applytolist{calc}QWERTYUIOPLKJHGFDSAZXCVBNM;

\def\bfc#1{\expandafter\def\csname b#1\endcsname{{\mathbf #1}}}
\applytolist{bfc}QWERTYUIOPLKJHGFDSAZXCVBNM;

\def\bbc#1{\expandafter\def\csname bb#1\endcsname{{\mathbb #1}}}
\applytolist{bbc}QWERTYUIOPLKJHGFDSAZXCVBNM;

\def\hpic #1 #2 {\mbox{$\begin{array}[c]{l} \epsfig{file=#1,height=#2}
\end{array}$}}
\def\vpic #1 #2 {\mbox{$\begin{array}[c]{l} \epsfig{file=#1,width=#2}
\end{array}$}}

\newtheorem{theorem}{Theorem}[section]
\newtheorem{proposition}[theorem]{Proposition}

\newtheorem{definition}[theorem]{Definition}

\newtheorem{fact}[theorem]{Fact}
\newtheorem{conj}[theorem]{Conjecture}
\newtheorem{question}[theorem]{Question}

\newcommand{\arxiv}[1]{\href{http://arxiv.org/abs/#1}{\tt arXiv:\nolinkurl{#1}}}
\newcommand{\doi}[1]{\href{http://dx.doi.org/#1}{{\tt DOI:#1}}}

\newcommand{\googlebooks}[1]{(preview at \href{http://books.google.com/books?id=#1}{google books})}

\newcommand{\pathtodiagrams}{\pathtotrunk}
\newcommand{\mathfig}[2]{{\hspace{-3pt}\begin{array}{c}%
  \raisebox{-2.5pt}{\includegraphics[width=#1\textwidth]{\pathtodiagrams diagrams/#2}}%
\end{array}\hspace{-3pt}}}

\makeatletter
\newcommand{\hashdef}[2]{\@namedef{#1}{#2}}
\newcommand{\hashlookup}[1]{\@nameuse{#1}}
\makeatother

\newcommand{\pathtographs}{diagrams/graphs/}

\hashdef{bwd1duals1}{FF1481245D7F1881}%
\hashdef{bwd1v1duals1v1}{7BFEC24E607C77CA}%
\hashdef{bwd1v1p1p1p1duals1v2x1x4x3}{19385ECD51785D9D}%
\hashdef{bwd1v1p1p1v1x1x0duals1v1x2x3}{53D091157B8A128E}%
\hashdef{bwd1v1p1v1x0p0x1duals1v1x2}{FC413903EB6EB09E}%
\hashdef{bwd1v1p1v1x0p0x1v1x0p1x0p0x1v1x0x0p0x0x1v1x0p1x0p1x0p0x1p0x1v0x1x0x1x0v1duals1v1x2v1x3x2v1x2x4x3x5v1}{E062CBA9D4D0BB85}%
\hashdef{bwd1v1p1v1x0p0x1v1x0p1x0p1x0p0x1v0x1x0x1v1duals1v1x2v1x2x4x3v1}{344D516A5767A6CA}%
\hashdef{bwd1v1p1v1x0p1x0duals1v1x2}{58AF9CED6D2C5F1A}%
\hashdef{bwd1v1p1v1x0p1x0v1x0p1x0v1x1duals1v1x2v1x2}{8F007BF50C3E09AC}%
\hashdef{bwd1v1p1v1x0p1x0v1x0v1p1v1x0p1x0v1x1duals1v1x2v1v1x2}{21FA8EFC6770814D}%
\hashdef{bwd1v1p1v1x0p1x1duals1v1x2}{FE060C2815A2CF0E}%
\hashdef{bwd1v1p1v1x1v1duals1v1x2v1}{434A46414835DC7C}%
\hashdef{bwd1v1v1p1p1v1x0x0p0x1x0duals1v1v2x1}{63FDB4A04F409318}%
\hashdef{bwd1v1v1p1p1v1x0x0p0x1x0p0x0x1duals1v1v1x2x3}{28770811E6CBBC84}%
\hashdef{bwd1v1v1p1p1v1x0x0p0x1x0p0x0x1duals1v1v2x1x3}{7BF64A4A5AFB92CF}%
\hashdef{bwd1v1v1p1v1x0p0x1p0x1v0x1x0p1x0x1duals1v1v2x1x3}{185DF41920CC8F1E}%
\hashdef{bwd1v1v1v1p1p1p1v0x0x0x1p0x0x0x1duals1v1v1x2x3x4}{DB36D3E4E78F84AD}%
\hashdef{bwd1v1v1v1p1p1p1v0x0x1x0p0x0x0x1v1x0p0x1duals1v1v1x2x3x4v2x1}{A18F2F6E20066862}%
\hashdef{bwd1v1v1v1p1p1p1v0x0x1x0p0x0x1x0p0x0x0x1p0x0x0x1duals1v1v1x2x3x4}{FA83C91D5F9E1293}%
\hashdef{bwd1v1v1v1p1p1p1v1x0x0x0p0x1x0x0p0x0x1x0p0x0x0x1v1x0x0x0p0x1x0x0p0x0x1x0p0x0x0x1duals1v1v1x2x3x4v2x1x4x3}{E73A09E81FAA171B}%
\hashdef{bwd1v1v1v1p1p1v0x0x1p0x0x1duals1v1v1x2x3}{CD73E7EC42BC0A68}%
\hashdef{bwd1v1v1v1p1p1v0x1x0p0x0x1v1x0p0x1duals1v1v1x2x3v2x1}{8FAB4C9171733C9F}%
\hashdef{bwd1v1v1v1p1p1v0x1x0p0x0x1v1x0p0x1duals1v1v1x3x2v2x1}{2F4978FCEE7AF334}%
\hashdef{bwd1v1v1v1p1p1v1x0x0duals1v1v1x2x3}{1EA7D76D1EC7DC2E}%
\hashdef{bwd1v1v1v1p1p1v1x0x0duals1v1v1x3x2}{859E46C1C5FE8578}%
\hashdef{bwd1v1v1v1p1p1v1x0x0v1duals1v1v1x2x3v1}{C4B6B3927CB75C0E}%
\hashdef{bwd1v1v1v1p1v0x1p0x1v0x1v1duals1v1v1x2v1}{A7E28CB724C88259}%
\hashdef{bwd1v1v1v1p1v0x1p0x1v1x0p1x0p0x1v0x1x0v1duals1v1v1x2v1x2x3v1}{BAC8DDF76679121F}%
\hashdef{bwd1v1v1v1p1v1x0p0x1v1x0p0x1duals1v1v1x2v2x1}{25781EDEA61320E6}%
\hashdef{bwd1v1v1v1p1v1x0p0x1v1x0p0x1p0x1v1x0x0v1duals1v1v1x2v2x1x3v1}{E240CBB942B88F1C}%
\hashdef{bwd1v1v1v1p1v1x0p0x1v1x0p1x0p0x1p0x1v0x0x0x1p1x0x0x0v1x0p0x1v0x1p1x0duals1v1v1x2v4x2x3x1v1x2}{B20285934C344592}%
\hashdef{bwd1v1v1v1p1v1x0p0x1v1x0p1x0p0x1p0x1v0x1x0x0p0x0x0x1p1x0x0x0v1x0x0p0x1x0p0x1x0p0x0x1v0x0x0x1p1x0x0x0p0x1x0x0duals1v1v1x2v4x2x3x1v3x2x1x4}{497CBA2FCE9E1E85}%
\hashdef{bwd1v1v1v1p1v1x0p0x1v1x0p1x0p0x1v0x1x0p0x0x1duals1v1v1x2v3x2x1}{C0F8080B70410906}%
\hashdef{bwd1v1v1v1p1v1x0p0x1v1x0p1x1v1x0v1duals1v1v1x2v1x2v1}{48383DE8F4973A95}%
\hashdef{bwd1v1v1v1p1v1x0p0x1v1x1duals1v1v1x2v1}{F32D9ADB2D06E9C0}%
\hashdef{bwd1v1v1v1p1v1x0p1x0duals1v1v1x2}{E26689FBED470E90}%
\hashdef{bwd1v1v1v1p1v1x0p1x0v1x0p0x1duals1v1v1x2v1x2}{4B2F4D160292120F}%
\hashdef{bwd1v1v1v1p1v1x0p1x0v1x0p0x1v1x0p0x1v1x0p0x1v1x0p0x1duals1v1v1x2v1x2v2x1}{5C5F2EF2AFA52D81}%
\hashdef{bwd1v1v1v1p1v1x0p1x0v1x0p0x1v1x0p1x0p0x1v0x1x0p0x0x1v1x0p0x1p0x1duals1v1v1x2v1x2v2x1}{4EED68A56E6F0FF3}%
\hashdef{bwd1v1v1v1p1v1x0p1x0v1x0v1p1duals1v1v1x2v1}{9C7CA4862FDC5FF7}%
\hashdef{bwd1v1v1v1v1v1p1p1v1x0x0p0x1x0v1x0v1v1duals1v1v1v1x2x3v1v1}{5AB2F40E2972944E}%
\hashdef{bwd1v1v1v1v1v1p1v0x1p0x1v0x1v1duals1v1v1v1x2v1}{96B12B0FC207748A}%
\hashdef{bwd1v1v1v1v1v1p1v1x0p0x1v1x0p0x1p0x1v1x0x0v1duals1v1v1v1x2v2x1x3v1}{1D8CDC2F9BD90A57}%
\hashdef{bwd1v1v1v1v1v1v1v1p1v1x0p0x1v1x0p0x1duals1v1v1v1v1x2v2x1}{C53E32050C862F37}%
\hashdef{bwd1v1v1v1v1v1v1v1p1v1x0p1x0duals1v1v1v1v1x2}{CA1952BED6C8EB18}%

\newcommand{\bigraph}[1]{{\hspace{-3pt}\begin{array}{c}%
  \raisebox{-2.5pt}{\includegraphics[height=9mm]{\pathtographs \hashlookup{#1}}}% 
\end{array}\hspace{-3pt}}}

\definecolor{blue-unshaded}{rgb}{0.75,0.75,1}
\definecolor{blue-shaded}{rgb}{0.6,0.6,1}
\definecolor{red-unshaded}{rgb}{1,0.75,0.75}
\definecolor{red-shaded}{rgb}{1,0.6,0.6}

\hypersetup{ 
    linkcolor={dark-red},
    citecolor={dark-blue}, urlcolor={medium-blue}
}

\begin{document}

\title{The classification of subfactors of index at most 5.}
\author{Vaughan F. R. Jones}
%\address{V.J.: Department of Mathematics, Vanderbilt University,  1326 Stevenson Center, Nashville, TN 37240, USA}
%\email{vfr@vanderbilt.edu}
\thanks{V.J. is supported by the NSF under Grant No. DMS-0301173}
\author{Kim Morrison}
%\email{scott@tqft.net}
\thanks{S.M. is supported by the Australian Research Council under the Discovery Early Career Researcher Award No. DE120100232}
\author{Noah Snyder}
\thanks{N.S. was supported by a NSF Postdoctoral Fellowship at Columbia University.}
\thanks{All authors were supported by
DARPA grants HR0011-11-1-0001 and HR0011-12-1-0009.
}
%\subjclass[2000]{46L37 (46L54)}
%\keywords{Subfactors, Planar Algebras}

\begin{abstract}
A subfactor is an inclusion $N \subset M$ of von Neumann algebras with trivial centers.  The simplest example comes from the fixed points of a group action $M^G \subset M$, and subfactors can be thought of as fixed points of more general group-like algebraic structures.  These algebraic structures are closely related to tensor categories and have played important roles in knot theory, quantum groups, statistical mechanics, and topological quantum field theory.  There's a measure of size of a subfactor, called the index.  Remarkably the values of the index below 4 are quantized, which suggests that it may be possible to classify subfactors of small index.  Subfactors of index at most 4 were classified in the '80s and early '90s.  The possible index values above 4 are not quantized, but once you exclude a certain family it turns out that again the possibilities are quantized.  Recently the classification of subfactors has been extended up to index 5, and (outside of the infinite families) there are only 10 subfactors of index between 4 and 5.  We give a summary of the key ideas in this classification and discuss what is known about these special small subfactors.
%This is the submitted version of \arxiv{1304.6141}.
\end{abstract}

\maketitle
\vspace{-0.5cm}
\tableofcontents

\hypersetup{ 
    linkcolor={dark-red},
    citecolor={dark-blue}, urlcolor={medium-blue}
}

\section{Introduction}
   If you haven't heard of subfactors you may be wondering what part of mathematics
they belong to. Actually that is not an easy question to answer! There is no doubt that
the subject has its roots in functional analysis, 
in particular the theory of von Neumann algebras which von Neumann introduced as the
mathematical structure underlying quantum mechanics. But since the earliest days of the
subject there have been numerous interactions between subfactors and other branches of
mathematics and physics, beginning with the discovery of a polynomial  invariant of knots, and 
most recently a new connection with random matrices. On the way many possible homes
for subfactors have been visited including three manifold topology, statistical mechanics
in two dimensions, conformal field theory, Hecke algebras, quantum groups, 2-categories, compact
groups, finite groups, and discrete groups. One feature has emerged from all approaches---subfactors are group-like objects representing symmetries generalizing ordinary group
actions.

  To know what a subfactor is one obviously first has to know what a factor is. The precise
definition will be given below, but for now just think of a factor $M$ as a 
simple algebra with identity $1$. A subfactor $N\subset M$ is then a simple unital subalgebra.
If we took this definition literally it would include the case where $M$ and $N$ are fields.
Although inaccurate, this is a useful first approximation. In 
particular, the degree of the field extension $[M:N]$ is the dimension of $M$ as a vector
space over $N$ and indeed the original impetus for subfactor theory was to see to what
extent this notion, which we will rather call  the ``index'' of $N$ in $M$, can be extended
beyond the context of fields. There are a host of ways to extend this index but they all
give the same answer for a very special class of factors, namely the $\II_1$ factors of Murray and von
Neumann. A $\II_1$ factor is a simple complex *-algebra $M$ which possesses a trace
 $tr: M\rightarrow \mathbb C$, and may be realised as an infinite dimensional algebra of operators on Hilbert space.
The trace is linear and satisfies $tr(ab)=tr(ba)$ for all $a,b\in M$. One may normalise it so that
$tr(1)=1$. 

 The trace is the key
to the definition of index as it allows one to associate a \emph{dimension} to vector spaces
on which $M$ acts. This is familiar from K-theory where a finitely generated projective
$M$-module $V$ is the same thing as an idempotent in some matrix algebra over $M$ so
one may take its dimension $ \dim_MV$ 
as the sum of the traces of the diagonal elements of the matrix. Now if $N\subset M$ is
a subfactor, $M$ becomes an $N$-module by left multiplication.  Finiteness of the index
means that $M$ is a finitely generated  projective module and $[M:N]$ is by definition
the trace of the corresponding idempotent.  That is,
$$[M:N] = \dim_N(M).$$

There are two closely related ways to get subfactors from finite groups. The $\II_1$ factors tend
to have many (quite explicit) outer automorphisms. For instance, if $M$ is a $\II_1$ factor then so is $M \otimes M$, and  the ``flip'' automorphism
of $M\otimes M$ is outer. If $G$ is a finite group of outer automorphisms one may form
the ``crossed product'' $M\rtimes G$ consisting of formal linear combinations 
 $ \mathop \sum_{g\in G} a_g u_g$ where the $a_g$ are in $M$ and the algebra
structure is determined  in the obvious way by $u_gu_h=u_{gh}$ and $u_gxu_g^{-1}=g(x)$ for $x\in M$.
That $M\rtimes G$ is also a $\II_1$ factor is guaranteed by outerness of the action.  As you might expect, it is not difficult to see that $[M\rtimes G:M]=|G|$.   
By a duality which is a major part of the subfactor story, one obtains that the fixed point 
algebra $M^G$ is a $\II_1$ factor and that $[M:M^G]=|G|$.  There is a version of the Galois correspondence which says that any automorphism of $M$ fixing $M^G$ pointwise is inner (i.e. conjugation by an element of $M$), and thus  one may recover the group from the knowledge of the subfactor  $M^G\subset M$ as those inner automorphisms fixing $M^G$.   In
this way the theory of subfactors contains the theory of finite groups, though we are often most interested in understanding more complicated subfactors which don't arise from group theory.

Even better there is
an extraordinary $\II_1$ factor, called the \emph{hyperfinite} $\II_1$ factor $R$ on which each finite group
acts  in exactly one way by outer automorphisms.
\begin{theorem}[\cite{MR587749,MR0448101}]\label{vjthesis}
Any two actions of the finite group $G$ on $R$, for which all non-trivial elements of $G$ 
act by outer automorphisms, are conjugate by an automorphism of $R$.
\end{theorem}
It is easy to construct an outer action of any finite group (indeed any subgroup of the unitary group of Hilbert space!)
by outer automorphisms.
 So subfactors are an even better vessel for
groups than fields where the same abstract group can arise from many different Galois 
extensions. A further advantage is that, by \cite{MR0454659}, any subfactor of the hyperfinite $\II_1$ factor is
isomorphic to it (or finite dimensional) so, unlike in ordinary Galois theory where many different field extensions share the same
Galois group, all one sees in a hyperfinite subfactor $M^G \subset M$ is the group $G$!

Using groups as above  one obtains subfactors of index $n$ for all $n\in \mathbb N$ by varying $G$.  In fact, by looking at $M \rtimes H \subset M \rtimes G$, you can build subfactors of index $n = [G:H]$ for every transitive group action $G$ acting on $G/H$.    Now the property
of $\II_1$ factors that most intrigued their discoverers Murray and von Neumann is that the
trace takes \emph{all positive real values} on idempotents in matrix algebras! Thus in principle
the index could be any real number $\geq 1$. But the subject of subfactors got under way with the
result that the index is not always an integer, nor is it any real $\geq 1$, but rather it is
``quantized''.
\begin{theorem}[\cite{MR0696688}] \label{indexlessthanfour} If $[M:N]<4$ then there is a natural number $n\geq 3$ with $$[M:N]=4\cos^2\pi/n.$$
Moreover all these values occur as indices of subfactors, as does any real number $\geq 4$.
\end{theorem}

The numbers $4\cos^2\pi/n$ begin

\begin{center}
\newcommand\T{\rule{0pt}{2.6ex}}       % Top strut
\newcommand\B{\rule[-1.2ex]{0pt}{0pt}} % Bottom strut
\newcommand\TT{\rule{0pt}{3.6ex}}       % Top strut
\newcommand\BB{\rule[-1.8ex]{0pt}{0pt}} % Bottom strut
   \begin{tabular}{l | c|l  }
   $  4\cos^2\pi/3$& 1&1 \T\B \\ \hline
   $4\cos^2\pi/4$  & 2 & 2  \T\B \\ \hline
$4\cos^2\pi/5 $&$ \frac{3+\sqrt{5}}{2} $ &  2.61803  \T\B \\ \hline
   $4\cos^2\pi/6$ & 3 & 3\T\B \\ \hline
 $4\cos^2\pi/7$ &$\frac{5}{3}+\frac{1}{3} \left(\frac{7}{2} \left(1+3 \sqrt{3} i\right)\right)^{1/3}+\frac{7}{3} \left(\frac{7}{2} \left(1+3 \sqrt{3} i\right)\right)^{-1/3}$&3.24698\TT\BB \\ \hline
$4\cos^2\pi/8$ & $2+\sqrt 2$ & 3.41421 \T\B 
   \end{tabular}
\end{center}
(and of course these index values tend to 4 as $n\rightarrow \infty$).
% $$1,\hspace {4pt} 2,\hspace {4pt}\frac{3+\sqrt 5}{2} (\approx 2.618),\hspace {4pt}3,\hspace {4pt}...$$
  So the first subfactor of non-integer index 
has index equal to the square of the golden ratio.

A subfactor of non-integer index cannot come from a finite group but one might hope 
that it can be constructed from some ``group-like'' object.  One source of examples of such ``group-like" objects are Drinfel'd-Jimbo quantum groups at roots of unity \cite{MR797001,MR869575,MR0934283, MR890482, MR1091619, MR1470857, MR901157}.  The place to hunt down all such ``group-like" 
objects is by looking at subfactors of the hyperfinite $\II_1$ factor, for by  Theorem \ref{vjthesis} one would expect
the  structure of the factor itself to be invisible so that the subfactor is given by some
data describing ``pure symmetry''. This data is to be teased out of the subfactor itself. 
This teasing process is the same in spirit as Weyl's approach to group representations \cite{MR0000255}.  Namely, we extract the irreducible representations from the tensor powers of a basic one. The group,
if necessary, could then be constructed by some Tannaka-like duality from the
resulting tensor category \cite{0020.00904, MR1173027, MR0033809, MR1010160}. With 20-20 hindsight, this is exactly what happened in the study of subfactors.
The description of this group-like object will be the topic of the next section. For now we will simply refer to it as the ``standard invariant'' of the subfactor.

Theorem \ref{indexlessthanfour} answers the question ``What are the possible index values
of subfactors?''. But for index at most four we can go a lot further and actually classify
all the subfactors themselves.   In fact, there's an ADE classification of similar flavor to the McKay correspondence for subgroups of $SU(2)$ \cite{MR604577} due to Ocneanu \cite{MR996454} and others.  This was a significant achievement and we will describe it in more detail in Section \ref{sec:index-leq-4}.

The main focus of this paper is on subfactors of the hyperfinite $\II_1$ of index between 4 and 5.  Below index $4$ a deep theorem of Popa's \cite{MR1055708} showed that the standard invariant is a complete invariant of subfactors of the hyperfinite.  However, above index $4$ there are subfactors whose standard invariant gives no information about the subfactor and where Popa's theorem does not apply.  There are many  such subfactors but no progress has been
made in understanding them. These ``bad" subfactors satisfy a techical condition called non-amenability \cite{MR1278111}. 
We have shown that the standard invariant is trivial for any non-amenable subfactor whose index
is between 4 and 5, and  we have  a complete classification
of the standard invariants of amenable subfactors in this index range.  There are precisely five index
values, each one affording two subfactors. We will give a more precise statement in Theorem \ref{main-theorem}.
\begin{center}
\newcommand\T{\rule{0pt}{2.6ex}}       % Top strut
\newcommand\B{\rule[-1.2ex]{0pt}{0pt}} % Bottom strut
\newcommand\TT{\rule{0pt}{3.2ex}}       % Top strut
\newcommand\BB{\rule[-1.8ex]{0pt}{0pt}} % Bottom strut
   
   \begin{tabular}{ l | c | c  }
   subfactor  & index & approximate index \\ \hline
 `Haagerup' &  $   \frac{5+\sqrt{13}}{2} $& 4.30278  \T\B \\ \hline
 `extended Haagerup' &   $ {\frac{8}{3}}+{\frac1 3}\sqrt[3]{\scalebox{0.85} {${\frac{13} 2}(-5-3i\sqrt 3)$}}+{\frac 1 3}\sqrt[3]{\scalebox{0.85}{${\frac{13} 2}(-5+3i\sqrt 3)$}}$  & 4.37720  \TT\BB \\ \hline
 `Asaeda-Haagerup' & $ \frac{5+\sqrt{17}}{2} $ &  4.56155  \T\B \\ \hline
  `3311' & $3+\sqrt 3$ & 4.73205\T\B \\ \hline
  `2221' & $\frac{5+\sqrt{21}}{2}$ & 4.79129\T\B  
   \end{tabular}
\end{center}

At index 5 (again ignoring the ``bad" subfactors) we also have a classification of subfactors with nontrivial standard invariant and we know there is an $\epsilon>0.004$ such that there are no indices of such subfactors in the range $(5,5+\epsilon)$.

This work has been accomplished over a long period of time and has involved many different
people. The completion of the project for indices between 4 and 5 has entailed extensive 
computer calculations. Some of these calculations are rather innocent---those going into
the construction of examples are never any worse than the (exact) calculation of the traces
of the fourth powers of some moderately sized matrices. To exclude all the values between
4 and 5 (except those listed above) involved a systematic enumeration of certain graphs whose norms
are between 2 and $\sqrt 5$.  This calculation requires nothing more than computing graph norms and checking simple combinatorial conditions, and can be checked locally by hand.   But the sheer number of graphs that comes up would make checking the whole calculation very slow.  We have checked this calculation several times on different computers and in different languages, which gives us confidence in the computer calculation.

It would of course be desirable to have  computer-free 
versions of the proofs.   A naive estimate suggests that the classification up to index 5 requires $100$ times the computational effort of Haagerup's classification up to index $3+\sqrt{3}$, so doing the calculation by hand using our techniques would require a Herculean effort.  A proof without a computer thus requires some major new insight. Such insight could come from conformal
field theory---one of the most fascinating questions on subfactors is whether all subfactors are
in some way obtainable from conformal field theory. Although the subfactors
of index $3+\sqrt 3$ and  $\frac{1}{2}(5+\sqrt{21})$ are known to arise from CFT, no such construction is known for the others (although it is argued in \cite{MR2837122} that the
subfactors of index $\frac{1}{2}(5+\sqrt{13})$ ``should'' come from CFT).

\paragraph{\textbf{Can we extend the classification beyond index 5?} }
You might wonder whether we could classify all finite index subfactors.  This is way too ambitious. By  Theorem \ref{vjthesis} this would at the very least entail a complete classification of finite groups. But in fact things are far worse. The intuitive idea that a finite index subfactor corresponds to a finite object is quite wrong. One can construct a finite index subfactor from any finitely
generated group in such a way that the subfactor remembers the Cayley graph of the group. One is thus led
immediately into undecidability questions. And possibly even worse, by \cite{MR2503171} one should 
be able to construct
families whose Borel structure is not countably separated!  And all these ``wild'' behaviours should occur 
for subfactors of index 6. 

Since $6 = 3 \cdot 2$ is a product of allowed index values, a subfactor $N \subset M$ of index $6$ can have an intermediate subfactor $N \subset P \subset M$ \cite{MR1262294, MR1437496}.  Although such a subfactor can be thought of as built by combining two simpler subfactors, it turns out that there is often a bewildering variety of ways to compose them.  In particular, Bisch and Haagerup \cite{MR1386923} show that there's an index $6$ subfactor of the form $R^{\mathbb{Z}/2} \subset R \rtimes \mathbb{Z}/3$ for any of the profusion of quotient groups of the free product $\mathbb{Z}/2 * \mathbb{Z}/3 \cong \mathrm{PSL}_2(\mathbb{Z})$.  We will discuss these and other  ``wildness" results (e.g. \cite{MR2314611}) in \S \ref{Index6}.

\paragraph{\textbf{Should we extend the classification beyond index 5?}}
By the discussion so far the situation is rather intriguing. Up to index 5 we have a nice classification of
subfactors, but at index 6 all hell breaks loose. So somewhere in between 5 and 6 the onset of wildness 
must occur. Indeed we have $5<3+\sqrt 5<6$ and $3+\sqrt 5= \frac{3+\sqrt 5}{2} \times 2$ which
is a product on index values, so it is not implausible that wildness of some sort should begin at
index $3+\sqrt 5$. Of course it could begin earlier though we doubt this. One way to prove that
there is no wildness less than index $3+\sqrt 5$ is to push our classification methods all the way from
5 to $3+\sqrt 5$. 
A number of people have begun working on this and, although there are not yet definitive results, it appears there
are only two subfactors with index in this range! Then to finish the ``onset of wildness'' project we would
have to find large families of subfactors of index $3+\sqrt 5$. There are already signs of an explosion
of subfactors at this index but the situation is not yet clear.

\paragraph{\textbf{Acknowledgements}}
We have many people to thank and acknowledge. Emily Peters in particular has been working with us since the inception of this project, and provided the inspiration for the second and third authors to begin thinking about classifying planar algebras.  We'd also like to thank coauthors Stephen Bigelow, Frank Calegari, Masaki Izumi, David Penneys, Emily Peters, and James Tener.  We would also like to acknowledge the deep impact that the work of Uffe Haagerup and Marta Asaeda had on the development of this project; in particular Haagerup's approach to classifying subfactors of small index provided the backbone of our work.  

In addition, we would also like to thank Dietmar Bisch and Yasuyuki Kawahigashi for their support and interest in this project.  Furthermore, this work would not be possible without deep foundational work by Alain Connes, Adrian Ocneanu, and Sorin Popa.  We are also grateful for the hospitality of several places which hosted us while we were working on this project.  In particular, we'd like to thank Dietmar Bisch, the Shanks family, and Vanderbilt for hosting a conference where this work began.  We're grateful to Kyoto University, the University of Tokyo, Microsoft Station Q,  Canada/USA Mathcamp and the American Institute of Mathematics for hosting.  The second and third authors would also like to thank the Jones family for allowing us to use their home in Bodega Bay for a number of `planar algebra programming camps'.
% no need to mention grants here; done in a footnote on the title page.

\section{An introduction to subfactors}

\subsection{Subfactors and the standard invariant}

\subsubsection{Technical definition of a subfactor and its index.}
We quickly recall some standard definitions, which can be found in any standard text \cite{MR548728, MR0094722}.

A finite von Neumann algebra  $M$ may be defined as a unital Banach *-algebra possessing a linear
trace $\tr:M\rightarrow \mathbb C$ with the properties
\begin{enumerate}[label=(\roman*)]
\item $||a^*a||=||a||^2$,
\item $\tr(ab)=\tr(ba)$,
\item $\tr(1)=1$,
\item $\tr(a^*a)>0$ for all $a \neq 0$, and
\item  the unit ball (for the Banach space structure) of $M$ is complete for the metric defined by $\langle x, y\rangle = \tr(y^* x)$.
\end{enumerate}

A finite von Neumann algebra $M$ is called a \emph{$\II_1$ factor} if it is infinite dimensional
and its centre is one dimensional. The trace $tr$ is then the unique linear functional with 
properties (\romannumeral 2) and (\romannumeral 3).

\begin{definition} A \emph{subfactor} $N$ of a $\II_1$ factor $M$ is a sub *-algebra 
containing the identity of $N$ which is a $\II_1$ factor with the inherited structure.
\end{definition}

\begin{definition} \cite{MR0696688} A subfactor $N\subset M$ will be said to have \emph{finite index} 
if $M$ is a finitely generated projective left $N$-module  under left multiplication, and then
the index $[M:N]$ is the trace of an idempotent in a matrix algebra defining $M$
as a left $N$-module.
\end{definition}

There are a few important technical definitions which we will need below.  A subfactor is called \emph{irreducible} if $M$ is irreducible as an $N$-$M$ bimodule.  A subfactor is called extremal if the normalized traces on $M$ and the commutant $N'$ agree on $M \cap N'$.  Often we focus on the extremal case because it simplifies many technical issues, and because extremality is a relatively mild assumption as all irreducible subfactors are extremal and all finite depth subfactors (defined below) are extremal.

\subsubsection {Examples.}
The simplest way to get a $\II_1$ factor is to take a discrete  group $\Gamma$  and take the commutant of
the left regular representation on $\ell^2(\Gamma)$. This von Neumann algebra is commonly
denoted $L(\Gamma)$ and a trace on it is given by $$\tr(x)=\langle x\xi,\xi\rangle$$ 
where $\xi$ is the characteristic function of the identity in $\Gamma$. The von Neumann algebra $L(\Gamma)$
is a $\II_1$ factor exactly when  all of the non-identity conjugacy
classes of $\Gamma$  are infinite (an `ICC group'). This happens quite often, for instance the free groups $F_n$ for $n\geq 2$.

Note that $L(\Gamma)$ is not a complete invariant of $\Gamma$.  For example, if $\Gamma$ can be realized as a union of finite groups (for example, the group $S_\infty$ of permutations which leave all but finitely many letters fixed), then $L(\Gamma)$ is always the hyperfinite $\II_1$ factor $R$. Famously, it is unknown whether $L(F_n) \iso L(F_m)$ for $n \neq m$, both at least two.

It is easy to construct certain subfactors of $L(\Gamma)$.  If $\Gamma_0\subset \Gamma$ is a subgroup with infinite conjugacy classes one has 
(immediately from the definition):$$[L(\Gamma):L(\Gamma_0)]=[\Gamma:\Gamma_0].$$

There are many other constructions of $\II_1$ factors each of which comes with its supply
of subfactors, but in this paper we want to focus on the abstract symmetry defined
by a subfactor so we will de-emphasize the factors themselves.

\subsubsection{The principal graphs.}
 An algebra is a bimodule over any of its subalgebras, and bimodules over an algebra
have a tensor product structure. So it is natural, given a subfactor $N\subset M$, to
consider the tensor powers $M_k=M\otimes_NM\otimes_N\cdots\otimes_NM$ (with
$k$ copies of $M$ in the tensor product). If $[M:N]$ is
infinite the purely algebraic tensor product may not be appropriate \cite{MR1303779} but if $[M:N]<\infty$ the
purely algebraic tensor product works just as well as Hilbert space based versions \cite{MR2501843}.
Each $M_k$ is infinite dimensional but there are many ways to extract finite dimensional
information from them. Following Weyl \cite{MR0000255} we should decompose the $M_k$ into 
irreducible $N-N$ bimodules or equivalently consider $\End_{N-N}(M_k)$. But we can do
more---since $M_k$ is also an $M-M$ bimodule we can decompose into irreducible
$M-M$, $M-N$, and  $N-M$ bimodules as well. 
\begin{definition} The \emph{principal graph} of $N\subset M$ is the pointed bipartite
graph whose vertices are the (isomorphism classes of) irreducible $N-N$ and $N-M$ bimodules contained 
in all the $M_k$, with  $\dim (\Hom_{N-N}(V,W))$ edges between an  $N-M$ bimodule $V$ and an $N-N$ bimodule $W$.
The distinguished vertex of $\Gamma$ is the $N-N$ bimodule $N$ itself, and we will use a $\star$ to
denote it on the graph.

\end{definition}

The distance from $\star$ to a vertex is called its depth.  Note that the vertices at even depths are $N$-$N$ bimodules while the vertices at odd depths are $N$-$M$ bimodules.

Similarly the ``dual principal graph'' is obtained by restricting irreducible $M-M$ bimodules to $M-N$ bimodules.

We typically indicate slightly more information when giving a pair of principal graphs.  Namely any bimodule has a dual, or contragredient, bimodule.  The dual of an $A$-$B$ bimodule (where $A$ and $B$ are each one of $M$ and $N$) is a $B$-$A$ bimodule, so duals of even vertices are even vertices on the same graph, while duals of odd vertices are odd vertices on the other graph.  We record this duality data by using red tags to indicate duality on even vertices, and by having odd vertices at each depth of each graph at the same relative height as its dual on the other graph.

We say a principal graph is \emph{$n$-supertransitive} if up to distance $n$ from $\star$ the graph is just a linear chain.

This process is a little abstract but can be quite easy in practice and leads immediately
to many interesting questions. The easiest case to understand is the principal graph
for $N\subset N\rtimes G$. It is obvious that, as an $N-N$ bimodule, $N\rtimes G$ is the
direct sum, over $G$, of bimodules which are the trivial $N-N$ bimodules with the right action
twisted by the corresponding group element. Since this set of bimodules is closed under
tensor product we see that the principal graph is as follows (illustrated for the symmetric group
$S_3$), with as many $N-N$ bimodules as the order of the group:
$$
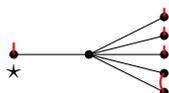
\begin{tikzpicture}
\draw[fill] (0,0) node[below] {$\star$} circle (0.05);
\draw (0.,0.) -- (1.,0.);
\draw[fill] (1.,0.) circle (0.05);
\draw (1.,0.) -- (2.,-0.5);
\draw (1.,0.) -- (2.,-0.25);
\draw (1.,0.) -- (2.,0.);
\draw (1.,0.) -- (2.,0.25);
\draw (1.,0.) -- (2.,0.5);
\draw[fill] (2.,-0.5) circle (0.05);
\draw[fill] (2.,-0.25) circle (0.05);
\draw[fill] (2.,0.) circle (0.05);
\draw[fill] (2.,0.25) circle (0.05);
\draw[fill] (2.,0.5) circle (0.05);
\draw[red, thick] (0,0) -- +(0,0.16) ;
\draw[red, thick] (2.,0) -- +(0,0.12) ;
\draw[red, thick] (2.,0.25) -- +(0,0.12) ;
\draw[red, thick] (2.,.5) -- +(0,0.12) ;
\draw[red, thick] (2.,-0.5) to[out=135,in=-135] (2.,-0.25);
\end{tikzpicture}$$
(the red marks indicate bimodules which are dual to each other, corresponding to inverse elements in the group).

With a little more thought one obtains the dual principal graph which also has a single $M-N$ bimodule,
but the $M-M$-bimodules are indexed by the \emph{irreducible representations} of $G$, with as many
edges as dimension of the representation. Thus for $S_3$ we obtain:
$$
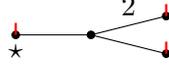
\begin{tikzpicture}
\draw[fill] (0,0) node[below] {$\star$} circle (0.05);
\draw (0.,0.) -- (1.,0.);
\draw[fill] (1.,0.) circle (0.05);
\draw (1.,0.) -- (2.,-0.25);
\draw (1.,0.) -- node[above] {$2$} (2.,0.25);
\draw[fill] (2.,-0.25) circle (0.05);
\draw[fill] (2.,0.25) circle (0.05);
\draw[red, thick] (0,0) -- +(0,0.16) ;
\draw[red, thick] (2.,-0.25) -- +(0,0.16) ;
\draw[red, thick] (2.,0.25) -- +(0,0.16) ;
\end{tikzpicture}$$

We see that the principal graph and dual principal graph may be different. It is an accident of small index
that all subfactors of index $\leq 4$ have the same principal and dual principal graphs.

The principal and dual principal graphs can be defined slightly differently as recording the fusion rules for tensoring by the basic bimodules $_{N}M_M$ and $_{M}M_N$.  For example, given an $N$-$M$ bimodule $V$ and and $N$-$N$ bimodule $W$, the number of edges from $V$ to $W$ is $\dim (\Hom_{N-N}(V \otimes_M M,W))$.    By Frobenius reciprocity, the multiplicity is also $\dim (\Hom_{N-M}(V,W \otimes_N M))$.  Thus you can read off the rules for tensoring on one side with the basic bimodules from the principal graph and the dual principal graph.  To read off the rules for tensoring on the other side, just use the dual data together with the fact that $(V \otimes X)^* \cong X^* \otimes V^*$.

\begin{definition} A subfactor is of \emph{finite depth} if the principal graph is finite.
\end{definition}

Unlike general finite index subfactors, finite depth subfactors are genuinely finite objects which can be thought of as generalizing finite groups.

Here are some basic facts about the principal graphs referring to a subfactor $N\subset M$ with
principal graph $\Gamma$ and dual principal graph $\Gamma'$.

\begin{definition}
The graph norm $||\Gamma||$ is the operator norm of the adjacency matrix for $\Gamma$.  When $\Gamma$ is finite, this is the largest eigenvalue of the adjacency matrix.
\end{definition}

\begin{fact}
\label{fact:index-and-graph-norm}
$[M:N]\geq ||\Gamma||^2$ with equality when $\Gamma$ is finite.
\end{fact}
\begin{fact} The radiuses (from $\star$) of $\Gamma$ and $\Gamma'$ differ by at most one.
\end{fact}

\begin{fact}Write $\stackrel {\rightarrow}{v}$ for the function on the vertices of $\Gamma$ 
whose value $\stackrel{\rightarrow}{v}_V$ on a bimodule is $\displaystyle dim_NV$ for
$N-N$ bimodules and $\displaystyle \sqrt{\dim_N(V)\dim_M(V)}$ for $N-M$ bimodules.
Then $\stackrel {\rightarrow}{v}$ is an eigenvector for the adjacency matrix of $\Gamma$
with eigenvalue $\sqrt{[M:N]}$. When $\Gamma$ is finite this is the unique positive eigenvector.
\end{fact}

\begin{fact} The dimensions $\dim H_0(\otimes^k_NM)=\dim H^0(\otimes^k_NM)$ of the invariant and coinvariant spaces are both given by the number of walks on
$\Gamma$ starting and ending at $\star$ of length $2k$, or the number of walks on
$\Gamma'$ starting and ending at $\star$ of length $2k$.
\end{fact}

\subsection{The standard invariant and reconstruction.}

The principal graphs capture the combinatorics of tensoring with the basic bimodules, but the full structure of these bimodules contains more than just the combinatorics.  We have a collection of bimodules (of four flavors: $N$-$N$, $N$-$M$, $M$-$N$, and $M$-$M$), maps between bimodules, and ways of taking duals and tensor products of bimodules and maps.  All of this information is called the standard invariant of $N \subset M$. 
But what kind of algebraic structure is the standard invariant?  If we were only looking at one flavor of bimodule (say $N$-$N$), this collection would have the structure of a tensor category.  Here we have a slightly different setup where there are two tensor products $\otimes_N$ and $\otimes_M$, but it is not difficult to modify the definition of a tensor category to allow this slightly more general situation.  

Modern higher categorical language gives one way of summarizing the structure on the standard invariant. Just as a monoid is the same thing as a category with only one object (the elements becoming $1$-morphisms), a tensor category can be thought of as a $2$-category with one object (the objects becoming $1$-morphisms, and the morphisms becoming $2$-morphisms).  From this point of view, the standard invariant becomes a $C^*$-2-category with two objects (corresponding to $M$ and $N$) together with a choice of generating $1$-morphism (corresponding to ${}_N M_M$).  The $N$-$N$ bimodules and the $M$-$M$ bimodules are called the even parts of the standard invariant, and the $N$-$M$ bimodules are called the odd part.

In fact the first axiomatizations of the standard invariant (by Ocneanu \cite{MR996454} and Popa \cite{MR1055708,MR1334479}) were combinatorial, based on the principal graphs, and
did not use categorical language at all. We will explain Ocneanu's cell calculus below, and Popa's  $\lambda$-lattices when we
have defined planar algebras which allow a diagrammatic descripition of $\lambda$-lattices. But both approaches turn around
a \emph{pair of towers of finite dimensional C$^*$-algebras}:
$$\begin{array} {lllll}
B_0 \subset &  B_1 \subset & B_2 \subset &\cdots \subset &B_n \subset \cdots \\
\cup            &\cup              &\cup             &                        &\cup                  \\
A_0 \subset & A_1 \subset & A_2 \subset &\cdots \subset &A_n \subset \cdots
\end{array}$$
together with a coherent trace on all these algebras. This data allows us to complete the unions of the $A_n$ and $B_n$
to obtain von Neumann algebras $A$ and $B$ respectively, with $A\subset B$.
\begin{theorem} {\cite{MR1055708}} If $M$ is the hyperfinite II$_1$ factor and the principal graph of the 
subfactor $N\subset M$ is finite, suppose $A\subset B$ is constructed from the standard invariant as above.
Then the inclusion $N\subset M$ is isomorphic to the inclusion $A\subset B$.
\end{theorem}

(Depending on your conventions for this tower, it may be an anti-isomorphism rather than an isomorphism.)
Popa went on in \cite{MR1278111} to prove even more powerful classification results when $\Gamma$ is infinite but they require subtle
notions of amenability of the standard invariant so we refrain from stating his theorem. If $\Gamma$ is infinite, neither $A$
nor $B$ need be factors. Note that although all subfactors with $[M:N] \leq 4$ are amenable, when $[M:N]>4$ and the principal graph is the one-sided infinite Dynkin diagram $A_\infty$ then the subfactor is \emph{not} amenable and Popa's results are not available.

Because of these problems with $A\subset B$ it is not clear that there is a hyperfinite subfactor having a given
standard invariant. It is however true \cite{MR1334479}  that the map from subfactors to standard invariants is surjective, and in fact
by \cite{MR2051399} there is, to every standard invariant, a subfactor of $L(F_\infty)$ which realises that standard
invariant. (If the principal graph is finite there is no issue---there is a hyperfinite subfactor realising the standard invariant.)

Outside the hyperfinite world one would not expect classification results.
Radulescu \cite{MR1426841} has shown that there are at least two outer involutory 
automorphisms of free group factors and thus two index two subfactors.  In the opposite direction, Vaes \cite{MR2471930} constructed factors with no irreducible subfactors at all!  Since there are many constructions producing a new subfactor from an old one using its standard invariant, it is not reasonable to expect factors with only one nontrivial subfactor.  Nonetheless, Falgui\`eres and Raum \cite{MR3028581} have proved a precise result saying that for any fixed finite depth subfactor planar algebra, there's a factor whose only finite index subfactors are those subfactors guaranteed by the planar algebra itself.  These constructions use Popa's theory of deformation rigidity \cite{MR2215135, MR2231961, MR2231962, MR2334200}.

\subsection{Other perspectives on the standard invariant}

\subsubsection{Cell calculus}
\label{sec:cell-calculus}

In this section we explain Ocneanu's approach to the standard invariant.  The key notion here is a ``connection" which is a certain collection of numbers attached to certain ``cells" and which play the role of 6j symbols.  We will define each of these notions below, and sketch how to get a connection from a subfactor.

First we introduce a generalization of the principal graph.  The various bimodules appearing on the principal graphs can be arranged on a square of graphs as below with edges connecting the vertices according to the principal graphs. We call this the 4-partite principal graph.
$$\begin{tikzpicture}
\node (NN) [fill=red-unshaded] {${}_N \mod_N$};
\node (NM) [fill=red-shaded] at ($(NN)+(3,0)$) {${}_N \mod_M$};
\node (MN) [fill=blue-shaded] at ($(NN)+(0,3)$) {${}_M \mod_N$};
\node (MM) [fill=blue-unshaded] at ($(NN)+(3,3)$) {${}_M \mod_M$};
\draw (NN) -- node[below] {$\Gamma$} (NM);
\draw (NN) -- node[left] {$\Gamma$} (MN);
\draw (NM) -- node[right] {$\Gamma'$} (MM);
\draw (MN) -- node[above] {$\Gamma'$} (MM);
\end{tikzpicture}$$

See Figure \ref{fig:haagerup-4} for an example.
\begin{figure}[ht]
$$
\left(\Gamma=\begin{tikzpicture}[baseline=-2]
\draw[fill] (0,0) node[below] {$\star$}circle (0.05);
\draw (0.,0.) -- (1.,0.);
\draw[fill] (1.,0.) circle (0.05);
\draw (1.,0.) -- (2.,0.);
\draw[fill] (2.,0.) circle (0.05);
\draw (2.,0.) -- (3.,0.);
\draw[fill] (3.,0.) circle (0.05);
\draw (3.,0.) -- (4.,-0.5);
\draw (3.,0.) -- (4.,0.5);
\draw[fill] (4.,-0.5) circle (0.05);
\draw[fill] (4.,0.5) circle (0.05);
\draw (4.,-0.5) -- (5.,-0.5);
\draw (4.,0.5) -- (5.,0.5);
\draw[fill] (5.,-0.5) circle (0.05);
\draw[fill] (5.,0.5) circle (0.05);
\draw (5.,-0.5) -- (6.,-0.5);
\draw (5.,0.5) -- (6.,0.5);
\draw[fill] (6.,-0.5) circle (0.05);
\draw[fill] (6.,0.5) circle (0.05);
\draw[red, thick] (0.,0.) -- +(0,0.333333) ;
\draw[red, thick] (2.,0.) -- +(0,0.333333) ;
\draw[red, thick] (4.,-0.5) -- +(0,0.333333) ;
\draw[red, thick] (4.,0.5) -- +(0,0.333333) ;
\draw[red, thick] (6.,-0.5) to[out=135,in=-135] (6.,0.5);
\end{tikzpicture}, \Gamma'=\begin{tikzpicture}[baseline=-2]
\draw[fill] (0,0) node[below] {$\star$}circle (0.05);
\draw (0.,0.) -- (1.,0.);
\draw[fill] (1.,0.) circle (0.05);
\draw (1.,0.) -- (2.,0.);
\draw[fill] (2.,0.) circle (0.05);
\draw (2.,0.) -- (3.,0.);
\draw[fill] (3.,0.) circle (0.05);
\draw (3.,0.) -- (4.,-0.5);
\draw (3.,0.) -- (4.,0.5);
\draw[fill] (4.,-0.5) circle (0.05);
\draw[fill] (4.,0.5) circle (0.05);
\draw (4.,-0.5) -- (5.,-0.5);
\draw (4.,-0.5) -- (5.,0.5);
\draw[fill] (5.,-0.5) circle (0.05);
\draw[fill] (5.,0.5) circle (0.05);
\draw[red, thick] (0.,0.) -- +(0,0.333333) ;
\draw[red, thick] (2.,0.) -- +(0,0.333333) ;
\draw[red, thick] (4.,-0.5) -- +(0,0.333333) ;
\draw[red, thick] (4.,0.5) -- +(0,0.333333) ;
\end{tikzpicture}\right)
$$
$$
\begin{tikzpicture}[baseline,x=8mm,y=8mm]
	\node at (-5,1.5) {$\Gamma\,\Big\{$};
	\node at (-5,-.5) {$\Gamma'\,\Big\{$};
	\node  [fill=red-unshaded] at (-3.5,2) {$\sb{N}{\sf{Mod}}_N$};
	\filldraw (-2,2) node[below] {$\star$}circle(1mm) ; 
	\filldraw (0,2) circle (1mm);	
	\filldraw (2,2) circle (1mm); 
	\filldraw (3,2) circle (1mm);
	\filldraw (6,2) circle (1mm);
	\filldraw (7,2) circle (1mm); 
	\node  [fill=red-shaded]  at (-3.5,1) {$\sb{N}{\sf{Mod}}_M$};
	\filldraw (-1,1) circle (1mm);  
		\draw (-1,1)--(-2,2);  
		\draw (-1,1)--(-2,0);
		\draw (-1,1)--(0,2);  
		\draw (-1,1)--(0,0); 
	\filldraw (1,1) circle (1mm);  
		\draw (1,1)--(0,2);  
		\draw (1,1)--(0,0);
		\draw (1,1)--(2,2);  
		\draw (1,1)--(2,0);
		\draw (1,1)--(3,2);  
		\draw (1,1)--(3,0);  		
	\filldraw (4,1) circle (1mm); 
		\draw (4,1)--(2,2);  
		\draw (4,1)--(2,0);
		\draw (4,1)--(6,2);  
	\filldraw (5,1) circle (1mm); 
		\draw (5,1)--(3,2);  
		\draw (5,1)--(2,0);
		\draw (5,1)--(7,2);	
	\node  [fill=blue-unshaded] at (-3.5,0) {$\sb{M}{\sf{Mod}}_M$};  
	\filldraw (-2,0) node[below] {$\star$}circle (1mm); 
	\filldraw (0,0) circle (1mm);
	\filldraw (2,0) circle (1mm);
	\filldraw (3,0) circle (1mm); 
	\node  [fill=blue-shaded] at (-3.5,-1) {$\sb{M}{\sf{Mod}}_N$};
	\filldraw (-1,-1) circle (1mm);  
		\draw (-1,-1)--(-2,-2);  
		\draw (-1,-1)--(-2,0);
		\draw (-1,-1)--(0,-2);  
		\draw (-1,-1)--(0,0); 
	\filldraw (1,-1) circle (1mm);  
		\draw (1,-1)--(0,-2);  
		\draw (1,-1)--(0,0);
		\draw (1,-1)--(2,-2);  
		\draw (1,-1)--(2,0);
		\draw (1,-1)--(3,-2);  
		\draw (1,-1)--(3,0);  		
	\filldraw (4,-1) circle (1mm); 
		\draw (4,-1)--(2,-2);  
		\draw (4,-1)--(2,0);
		\draw (4,-1)--(7,-2);  
	\filldraw (5,-1) circle (1mm);
		\draw (5,-1)--(3,-2);  
		\draw (5,-1)--(2,0);
		\draw (5,-1)--(6,-2);  
	\node  [fill=red-unshaded] at (-3.5,-2) {$\sb{N}{\sf{Mod}}_N$};	
	\filldraw (-2,-2) node[below] {$\star$}circle (1mm); 
	\filldraw (0,-2) circle (1mm);
	\filldraw (2,-2) circle (1mm); 
	\filldraw (3,-2) circle (1mm);
	\filldraw (6,-2) circle (1mm);
	\filldraw (7,-2) circle (1mm); 
\end{tikzpicture}$$
\caption{The principal graphs for the Haagerup subfactor described in \S \ref{Haagerup}, as well as the 4-partite principal graph. 
In the 4-partite graph the rows of vertices correspond to the $N-N$ bimodules, the $N-M$ bimodules, the $M-M$ bimodules, the $M-N$ bimodules, and finally the $N-N$ bimodules again. Vertices  are ordered lexicographically by depth and height in $\Gamma$ or $\Gamma'$.}
\label{fig:haagerup-4}
\end{figure}

A ``cell" is a directed based loop formed by four edges, one taken from each of the four graphs.  A connection is an assignment of numbers to each cell satisfying certain properties. 
To compute the connection from a subfactor, first make a choice of basis for each basic hom space $\Hom(V \otimes X, W)$ (where $X$ is ${}_N M_M$ or ${}_M M_N$, whichever is appropriate) and each space $\Hom(X \otimes V, W)$.  Using these choices we can think of each edge in the $4$-partite graph as giving an explicit bimodule map.  Thus each cell gives two different maps $\Hom_{M-M}(X \otimes V \otimes X, U)$ coming from the two ways of going halfway around the cell.  Since $U$ is simple, these two maps differ by a scalar, which assigns a number to each cell.  This number can be thought of as an associator or $6j$-symbol for tensoring with the basic bimodules.

These numbers attached to cell  satisfy many properties, involving also the duality maps among the irreducible bimodules.  This can be axiomatized by saying that the connection is biunitary and flat.  We postpone the definitions of these two notions until the discussion of graph planar algebras in \S \ref{sec:GPA}, but for now we note that biunitarity is easy to check, while flatness is much more mysterious and difficult to compute.

\subsubsection{Planar Algebras}  

Recall that the standard invariant is roughly a bunch of hom spaces together with the operations of composition and tensor product.  However, because of Frobenius reciprocity, you can identify any hom space with some invariant space  $H_0(\otimes^n_NM)$.  From this point of view, composition and tensor product no longer look special, and there are many other equally valid compositions.  Planar algebras are a coordinate free approach to the standard invariant which emphasize this multitude of operations.

The notion of a (colored) operad makes rigorous this idea of a multitude of operations \cite{MR0420610}: each element of an operad gives an operation and operations can be composed using the operad structure.  For example, there's an operad of words where the composition is given by substitution.  An algebra over an operad is a vector space on which all of these operations make sense, so an algebra over the operad of words is an associative algebra.

A  planar algebra is an algebra over the planar operad. Elements of the planar operad are ``planar tangles'' whose definition we
avoid by giving an example:
$$\mathfig{0.35}{tangle}$$

The smooth curves joining the boundary points of the discs are called strings. A tangle may be glued into an internal disc of another tangle provided the starred boundary  points, and the number of strings meeting the
glueing disc match. (This forces us to consider strings that form closed loops not touching any discs.)
This is the operad structure. Thus a planar algebra is an $\mathbb N$-graded vector space 
whose grading corresponds to the number of boundary points of a disc, and every tangle
determines a multilinear map from the product of the vector spaces associated with the
input discs to the vector space of the output disc. These  maps must only depend on the tangle up
to isotopy and are compatible in the
obvious way with the gluing.  See \cite{math.QA/9909027, JonesPANotes} for further details.

One example of a planar algebra is the planar algebra of  linear combinations of $3$-dimensional tangles (that is links with boundaries).  Here the vector space attached to a disc with $n$ boundary points has as a basis the set of all isotopy classes (rel boundary) of tangles in the disc cross $I$ with $n$ boundary points lying on the equator.  This is an algebra for the planar operad because we can plug $3$-dimensional tangles into the holes in a planar tangle to get a new $3$-dimensional tangle.

The key point here is that each planar tangle gives an operation.  For example, composition and tensor product are given by the operations below.%
%
%\beginpgfgraphicnamed{diagrams/tikz/#1-external}%
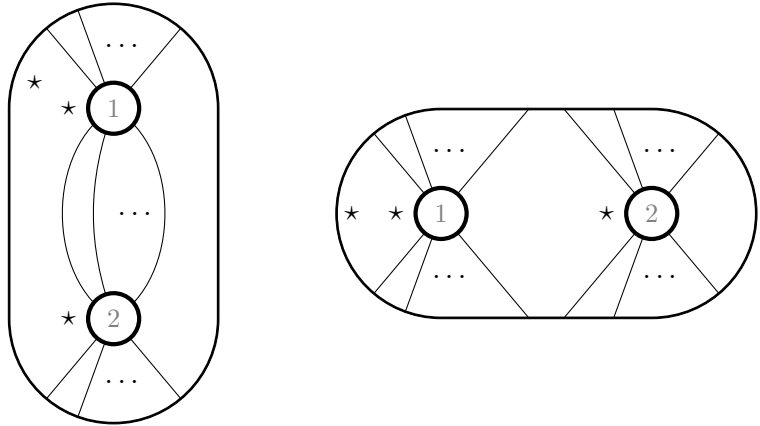
\tikzstyle{shaded}=[fill=red!10!blue!20!gray!50!white]
\tikzstyle{shaded line}=[double=red!10!blue!20!gray!50!white, double distance=2mm, draw=black]
\tikzstyle{unshaded}=[fill=white]
\tikzstyle{unshaded line}=[double=white, double distance=2mm, draw=black]
\tikzstyle{Tbox}=[circle, draw, thick, fill=white, opaque,]
\tikzstyle{empty box}=[circle, draw, thick, fill=white, opaque, inner sep=2mm]
\tikzstyle{background rectangle}= [fill=red!10!blue!20!gray!40!white,rounded corners=2mm] 
\tikzstyle{on}=[very thick, red!50!blue!50!black]
\tikzstyle{off}=[gray]
%
%These are for resizing a family of drawings at the same time
\tikzstyle{traces}=[scale=.2, inner sep=1mm,baseline]
\tikzstyle{quadratic}=[scale=.23, inner sep=.5mm, baseline]
\tikzstyle{annular}=[scale=.6, inner sep=1mm, baseline]
\tikzstyle{make triple edge size}= [scale=.4, inner sep=1mm,baseline] 
\tikzstyle{icosahedron network}=[scale=.35, inner sep=1mm, baseline]
\tikzstyle{ATLsix}=[scale=.25, baseline]
\tikzstyle{TL12}=[scale=.15,baseline=-0.5ex]
% I added the "=-0.5ex" -SB
\tikzstyle{TLEG}=[scale=.5,baseline]
\tikzstyle{PAdefn}=[scale=.7,baseline]
\tikzstyle{STrain}=[baseline=0,scale=2]%
%\endpgfgraphicnamed
%
$$%
%\beginpgfgraphicnamed{diagrams/tikz/#1-external}%
 \begin{tikzpicture}[PAdefn]
	\clip [draw] (2,2) arc (0:180:2cm) -- (-2,-2) arc (-180:0:2cm) -- (2,2);
	
	%first draw the lines
	\draw (0,2) .. controls ++(-150:1.5cm) and ++(150:1.5cm) .. (0,-2) .. controls ++(110:1.5cm) and ++(-110:1.5cm) .. (0,2);

	\draw (0,2) .. controls ++(-30:1.5cm) and ++(30:1.5cm) .. (0,-2);
	
	\draw (0,2) -- +(110:3cm) -- +(130:3cm) -- (0,2);
	\draw (0,2) -- ++(50:3cm);
		
	\draw (0,-2) -- ++(-50:3cm);
	\draw (0,-2) -- +(-110:3cm) -- +(-130:3cm) -- (0,-2);
	
	%decorate with dots, labels and stars
	\node at (.4,0) {$\ldots$};
	\node at (.2,3.2) {$\dots$};
	\node at (.2,-3.2) {$\dots$};

	\node at (0,2) [Tbox,ultra thick, inner sep=1.4mm] (A) {\small{\textcolor{gray}{1}}};
	\node at (0,-2) [Tbox,ultra thick, inner sep=1.4mm] (B) {\small{\textcolor{gray}{2}}};
	\node at (A.180) [left] {$\star$};
	\node at (B.180) [left] {$\star$};
	\node at (-1.5,2.5)  {$\star$};	
	
	%redraw boundary
	\draw[ultra thick] (2,2) arc (0:180:2cm) -- (-2,-2) arc (-180:0:2cm) -- (2,2);
\end{tikzpicture}%
%\endpgfgraphicnamed
 \qquad \qquad %
%\beginpgfgraphicnamed{diagrams/tikz/#1-external}%
 \begin{tikzpicture}[PAdefn]
	\clip [draw] (2,2) arc (90:-90:2cm) -- (-2,-2) arc (-90:-270:2cm) -- (2,2);
	
	%first draw the lines	
	\draw (2,0) -- +(110:3cm) -- +(130:3cm) -- (2,0);
	\draw (2,0) -- ++(50:3cm);
	\draw (2,0) -- ++(-50:3cm);
	\draw (2,0) -- +(-110:3cm) -- +(-130:3cm) -- (2,0);
		
	\draw (-2,0) -- +(110:3cm) -- +(130:3cm) -- (-2,0);
	\draw (-2,0) -- ++(50:3cm);
	\draw (-2,0) -- ++(-50:3cm);
	\draw (-2,0) -- +(-110:3cm) -- +(-130:3cm) -- (-2,0);
	
	%decorate with dots, labels and stars
	\node at (-1.8,1.2) {$\dots$};
	\node at (-1.8,-1.2) {$\dots$};
	\node at (2.2,1.2) {$\dots$};
	\node at (2.2,-1.2) {$\dots$};

	\node at (2,0) [Tbox,ultra thick, inner sep=1.4mm] (A) {\small{\textcolor{gray}{2}}};
	\node at (-2,0) [Tbox,ultra thick, inner sep=1.4mm] (B) {\small{\textcolor{gray}{1}}};
	\node at (A.180) [left] {$\star$};
	\node at (B.180) [left] {$\star$};
	\node at (-3.7,0)  {$\star$};	
	
	%redraw boundary
	\draw[ultra thick] (2,2) arc (90:-90:2cm) -- (-2,-2) arc (-90:-270:2cm) -- (2,2);
\end{tikzpicture}%
%\endpgfgraphicnamed
$$

A \emph{subfactor planar algebra} is a special kind of planar algebra. We ask first of all
that the regions of the tangle are alternately unshaded and shaded (corresponding to $N$ and $M$) which forces all the
discs to have an even number of boundary points. If we call $P_{n,\pm}$ the vector space 
corresponding to $2n$ boundary points, with the starred region unshaded for $\pm=+$ and shaded for $\pm=-$, then we insist that:
\begin{enumerate}[label=(\roman*)]
\item $\dim P_{n,\pm}<\infty$,
\item there is a conjugate-linear involution $*$ on each $P_{n,\pm}$ compatible with
orientation reversing diffeomorphisms of the tangles,
\item \label{subfactor-PA-evaluable} $\dim P_{0,\pm}=1$, and
\item\label{subfactor-PA-positive}  the sesquilinear form defined by 
$$
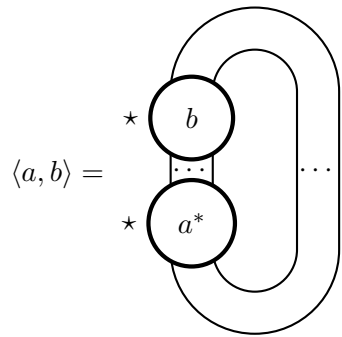
\langle a, b \rangle = 
\begin{tikzpicture}[baseline=16,x=0.7cm,y=0.7cm]
\draw[thick] (-0.4,-0.5) -- (-0.4,2.5) arc (180:0:1.6) -- (2.8,-0.5) arc (0:-180:1.6);
\draw[thick] (0.4,-0.5) -- (0.4,2.5) arc (180:0:0.8) -- (2,-0.5) arc (0:-180:0.8);
\node at (0,1) {$\cdots$};
\node at (2.4,1) {$\cdots$};
\node[draw, circle,  ultra thick, fill=white, inner sep=7pt] (a) at (0,0) {$a^*$};
\node[draw, circle,  ultra thick, fill=white, inner sep=8pt] (b) at (0,2) {$b$};
\node[left] at (a.180) {$\star$};
\node[left] at (b.180) {$\star$};
\end{tikzpicture}
$$
 is positive definite.
\end{enumerate}

Condition \ref{subfactor-PA-evaluable} above, along with the canonical `empty diagram' element in $P_{0,\pm}$, allows us to identify $P_{0,\pm} \iso \Complex$. The sesquilinear form defined in condition \ref{subfactor-PA-positive} is thus, like all closed diagrams, valued in $\Complex$.

Oftentimes we also want to require the following condition, which corresponds to extremality of the subfactor.

\begin{enumerate}[resume*]
 \item \label{subfactor-spherical} A planar algebra is spherical if for any element $x$ of $P_{1,+}$, the following two closed diagrams are equal as elements of $P_{0,\pm} \iso \Complex$.
 \begin{equation*}
 \begin{tikzpicture}[baseline=1]
\node[draw,fill=white,circle, ultra thick] (a) at (1,0) {$x$};
\node[left=-2pt] at (a.180) {$\star$};
\begin{pgfonlayer}{background}
    \fill[white] (0.25,-0.8) rectangle (1.75,0.8);
    \filldraw[thick,shaded] (a.north) arc (180:0:0.3) -- +(0,-0.57) arc (360:180:0.3);
\end{pgfonlayer}
 \end{tikzpicture}
=
 \begin{tikzpicture}[baseline=1]
\node[draw,fill=white,circle, ultra thick] (a) at (1,0) {$x$};
\node[left=-2pt] at (a.180) {$\star$};
\begin{pgfonlayer}{background}
    \fill[shaded] (0.25,-0.8) rectangle (1.75,0.8);
    \filldraw[thick,fill=white] (a.north) arc (0:180:0.3) -- +(0,-0.57) arc (180:360:0.3);
\end{pgfonlayer}
 \end{tikzpicture}
 \end{equation*}
\end{enumerate}

\begin{theorem} The spaces $H_0(\otimes^n_NM)$ admit a canonical subfactor planar algebra structure.
\end{theorem}
\begin{proof}
We quickly sketch the proof, which follows the diagram calculus of \cite{MR776784, MR1113284, MR1091619}.  For a full proof, see \cite{math.QA/9909027}.  All planar tangles are generated by the tensor product tangles, compositions tangles, and single strands.  The actions of these special tangles can be interpreted in the usual way (tensoring, composing, and identity/evaluation/coevaluation maps), and the tensor category axioms guarantee that this action is well-defined.
\end{proof}

It turns out that the index of a subfactor is just the value of two nested circles, so it makes sense to talk about the index of a subfactor planar algebra.  In the extremal case, the planar algebra is spherical, so the index is just the square of the value of a single circle.  This scalar is called the loop value and is typically written $\delta$.

So we can finally make rigorous what kind of gadget the standard invariant is: it is a subfactor planar algebra.  Our goal is to classify subfactor planar algebras of index below $5$.  By Popa's reconstruction results, so long as these planar algebras are amenable, this will give us a classification of subfactors of the hyperfinite $\II_1$ of index below $5$.  It turns out that there's only one family of non-amenable planar algebras, so we end up with a nearly complete classification of subfactors of the hyperfinite $\II_1$.

Planar algebras afford the possibility of a ``generators and relations'' approach. Indeed
a special example of planar algebra is the theory of van Kampen diagrams \cite{MR1506963}. The idea is to
give abstract elements called ``$n$-boxes'' which are to be elements of $P_n$ for some
planar algebra and consider linear identities between planar tangles labelled
by  these $n$-boxes. One tries to find such identities which cause the quotient of
the free planar algebra on the $n$-boxes to collapse to a planar algebra with finite
dimensional $P_{n,\pm}$'s. This is the same in spirit as Conway's skein theory and indeed
one finds examples of planar algebras immediately from knot theory, by treating
the crossing as generators and the Reidemeister moves, and, say, the Conway 
skein relation \cite{MR1501429} for the Alexander polynomial \cite{MR0258014}, as planar identities.

\subsubsection{Algebra objects and module categories}
 \label{sec:module-categories}
There is another algebraic approach to the standard invariant coming out of the work of Longo, M\"uger, and Ostrik \cite{MR1257245, MR1444286, MR1966524, MR1976459}, which has been developed extensively in the work of Etingof, Gelaki, Nikshych, and Ostrik \cite{MR2119143, MR2097289, MR2183279,EGNO}, and is based on two closely related concepts: algebra objects and module categories.

An algebra object in a tensor category $\cC$ is an object $A$ endowed with the structure of an algebra, namely a unit morphism $1 \rightarrow A$ and a multiplication morphism $A \otimes A \rightarrow A$ satisfying the unit and associativity axioms.  A subfactor $N \subset M$ yields an algebra object $M$ in a tensor category of $N$-$N$ bimodules.  In the subfactor literature, $C^*$-algebra objects  in $C^*$ tensor categories are called $Q$-systems.  

Just as algebras have modules, algebra objects have module objects, which are objects $V$ together with an action $A \otimes V \rightarrow V$ satisfying associativity.  These modules can be tensored over algebras in the usual way.  Thus, an algebra object in a tensor category also produces a standard invariant by considering the $1$-$1$ bimodules (i.e objects in $\cC$), $1$-$A$ bimodules (i.e. right $A$-modules), $A$-$1$ bimodules (i.e. left $A$-modules), and $A$-$A$ bimodules in $\cC$.

The category of right $A$-modules has an additional structure: you can tensor on the left with objects in $\cC$ to get new right $A$-modules.  This makes the category of \emph{right} $A$-modules into a \emph{left} $\cC$-module category, that is a category $\cM$ together with a bifunctor $\cC \times \cM \rightarrow \cM$ endowed with associativity morphisms satisfying a coherence relation.  When $\cC$ is nice enough, any nice module category comes from an algebra object in this way using Ostrik's internal hom construction \cite{MR1976459}, namely if $X$ is any object in $\cM$ then $\underline{\operatorname{Hom}}(X, X)$ is an algebra object in $\cC$ whose category of modules recovers $\cM$.

Thus one can also think of the standard invariant as being a Morita equivalence of tensor categories $_{\cC}\cM_\cD$ where $\cC$ and $\cD$ are the even parts of the subfactor and $\cM$ is the odd part, together with a choice of simple object in $\cM$.  This point of view offers great flexibility in changing the choice of favorite object in $\cM$ and thereby makes certain constructions in subfactor theory easier to understand.

This theory has been most extensively worked out when $\cC$ is a fusion category, that is it is finite, semisimple, and the unit object is simple.  These correspond to finite depth finite index subfactors (finiteness is finite depth, semisimplicity follows from unitarity, and the unit object is simple corresponds to the von Neuman algebras being factors), but with the $C^*$ condition relaxed.

\subsection{A toolkit for planar algebras}
We now explain three of of the most important tools for analyzing planar algebras. The Temperley-Lieb planar algebra $TL(\delta)$ naturally maps to any planar algebra with loop value $\delta$. Every planar algebra has the structure of a module over the category of annular tangles, and we can decompose the planar algebra into irreducible modules. Finally, every subfactor planar algebra $P$ embeds into a certain combinatorially defined graph planar algebra $\cG(\Gamma(P))$, which only depends on the principal graph pair $\Gamma(P)$ of $P$ (and, in the infinite depth case, a choice of positive eigenvector for the adjacency matrix).

\subsubsection{Temperley-Lieb diagrams.}\label{sec:TL}
A planar tangle with no internal discs gives a map of the form $\mathbb{C} \to P_{n,\pm}$, which we can interpret (via the image of $1$) as giving an element in any planar algebra. 

Many interesting planar algebras have the property that a closed loop counts for a multiplicative
factor which we call $\delta$. (And indeed, if $P_{0,\pm}$ are 1-dimensional, the closed loops must be scalar multiples of the empty diagram.)
In this case we can restrict ourselves to connected tangles with no closed loops. These are called Temperley-Lieb tangles (though they should be called Kauffman diagrams after \cite{MR899057}).  If there are $2n$ boundary points there are $\frac{1}{n+1}\binom{2n}{n}$ Temperley-Lieb tangles.

The Temperley-Lieb tangles, with a specified loop value of $\delta$, themselves form a planar algebra $TL(\delta)$, and given any planar algebra $P$ with loop value $\delta$ there is a canonical map of planar algebras $TL(\delta) \to P$.  Note, however, that this map need not be injective; indeed below we will see that that for subfactor planar algebras it is injective exactly if $\delta \geq 2$. 

We can now sketch the proof of Theorem \ref{indexlessthanfour}.
If the planar algebra is a subfactor one, the index of the subfactor is  $\delta^2$ and
the sesquilinear form on the linear span $TL(\delta)_{n,\pm}$ of the Temperley-Lieb tangles
has to be positive semi-definite. As we observed above the sesquilinear form comes from a
trace on the *-algebra $TL_n$ whose multiplication is defined above. This algebra is best analysed
in terms of the Jones-Wenzl idempotent $f^{(n)}$ which is the unique (for $n>2$) idempotent whose scalar multiples
form a two-sided ideal. By induction the sesquilinear form is positive definite on the ideal complementary
to the Jones-Wenzl idempotent, so to check positive definiteness you need only check that trace of the Jones-Wenzl idempotent is positive. 
This was first calculated in \cite{MR0696688}. In \cite{MR873400} a
useful inductive formula was found for the Jones-Wenzl idempotents as linear combinations of Temperley-Lieb diagrams.

%%% \allchars
% A macro from http://www.tug.org/TUGboat/Articles/tb28-1/tb88glister.pdf for processing strings one character at time. 
% Calls `\doachar` on each character.
\makeatletter
\catcode`\^^G=12
\newcommand*{\allchars}[1]{%
\def\stuff{#1}\ifx\stuff\@empty\else
\@llchars#1^^G\fi}
\def\@llchars#1#2^^G{%
\def\letter{#1}\def\others{#2}%
\ifx\letter\@empty\let\next\@gobble
\else
\doachar{#1}%
\ifx\others\@empty \let\next\@gobble
\else \let\next\@llchars \fi
\fi
\next#2^^G}
\catcode`\^^G=15
\makeatother
%%%  finished with \allchars

\newcounter{xtop} %counter
\newcounter{xbottom} %counter

\newcommand{\doachar}[1]{
\ifx#1i
  \draw (\thextop,1) .. controls ($(\thextop,1)+(0,-0.5)$) and ($(\thexbottom,0)+(0,0.5)$) ..  (\thexbottom,0);
  \stepcounter{xtop};
  \stepcounter{xbottom};
\else\ifx#1u
  \draw (\thextop,1) arc (-180:0: 0.5 and 0.3);
  \stepcounter{xtop};
  \stepcounter{xtop};
\else\ifx#1n
  \draw (\thexbottom,0) arc (180:0: 0.5 and 0.3);
  \stepcounter{xbottom};
  \stepcounter{xbottom};
\fi\fi\fi
}

\newcommand{\TL}[1]{
\setcounter{xtop}{1}
\setcounter{xbottom}{1}
\begin{tikzpicture}[baseline=10,x=0.5cm]
% finally read the string
\allchars{#1}
\path[use as bounding box] (0.8,1.2) rectangle ($(\thexbottom, -0.2)+(-0.8,0)$);
\end{tikzpicture}
}

As examples, the first three Jones-Wenzl idempotents are given by:%
\begin{align*}
f^{(1)} & = \TL{i} &
f^{(2)} & = \TL{ii} - \frac{1}{[2]} \TL{un} 
\end{align*}%
\begin{align*}
f^{(3)} & = \TL{iii} - \frac{[2]}{[3]} \left( \TL{iun} + \TL{uni} \right) + \frac{1}{[3]} \left( \TL{uin} + \TL{niu} \right)
\end{align*}
where as usual $[n] = \frac{q^n - q^{-n}}{q - q^{-1}}$ with $\delta = q+q^{-1}$.

For index $\geq 4$ all the Jones-Wenzl idempotents have positive trace so the sesquilinear form is positive definite on Temperley-Lieb diagrams when $\delta \geq 2$. But when $\delta  = 2 \cos \pi/m$ it has a kernel, and this kernel is generated (as an ideal in the planar algebra) by none other than the $(m-1)$-th Jones-Wenzl idempotent. When $\delta < 2$ but not equal to $2 \cos\pi/m$, the form is nondegenerate but not positive definite.

The Temperley-Lieb planar algebra plays a key role in knot theory.  The Kauffman bracket can be thought of as a map of planar algebras from the planar algebra of tangles to Temperley-Lieb: $$%
%\beginpgfgraphicnamed{diagrams/tikz/#1-external}%
\begin{tikzpicture}[baseline]
\node (x) at (0,0){};
    \draw (x.45)-- (.5,.5);
    \draw (x.135) -- (-.5,.5);
    \draw (x.315) -- (.5,-.5);
    \draw (x.45) -- (-.5,-.5);
\end{tikzpicture}%
%\endpgfgraphicnamed
 \mapsto is %
%\beginpgfgraphicnamed{diagrams/tikz/#1-external}%
\begin{tikzpicture}[baseline]
    \draw (1.5,.5) .. controls (2,0) .. (1.5,-.5);
    \draw (2.5,.5) .. controls (2,0) .. (2.5,-.5);
\end{tikzpicture}%
%\endpgfgraphicnamed
 - i s^{-1} %
%\beginpgfgraphicnamed{diagrams/tikz/#1-external}%
\begin{tikzpicture}[baseline]
    \draw (3.5,.5) .. controls (4,0) .. (4.5,.5);
    \draw (3.5,-.5) .. controls (4,0) .. (4.5,-.5);
\end{tikzpicture}%
%\endpgfgraphicnamed
,$$
where $\delta = s^2 + s^{-2}$.
The value of this map on links recovers the Kauffman bracket of the link \cite{MR899057}. The Jones-Wenzl idempotents are key to the extension of the Jones polynomial \cite{MR0766964} in
$S^3$ to invariants of links in arbitrary 3-manifolds and hence $2+1$ dimensional topological quantum field theory (as in \cite{MR1191373,MR1362791} giving a diagrammatic version of Witten-Reshetikhin-Turaev theory \cite{MR990772, MR1091619}).

\subsubsection{Annular tangles.}
\newcommand{\Aff}{\operatorname{{\mathcal A}ff}}
\newcommand{\Ann}{\operatorname{{\mathcal A}nn}}
The Temperley-Lieb planar algebra captures all the information we can get out of planar tangles with no internal discs.  The next most complicated diagrams are those with \emph{one} internal disc and these are called annular tangles.  Annular tangles form a
category, using the glueing operation of the planar operad. The objects of the category are 
$\mathbb N$, being half the number of points where strings meet the boundary. Morphisms
are the tangles themselves, up to isotopy. 
 The special interest for us of this annular category is that by definition it will act on any planar
algebra so that a first coarse classification of a planar algebra will be its structure as a module over
the annular category.  This action is well understood by work of Graham and Lehrer \cite{MR1659204} adapted to the $C^*$ setting in \cite{MR1929335}.

It is natural to think of an element $R$ of a planar algebra as being a ``weight vector'' whose weight is
half the number of strings emanating from a disc containing $R$. Thus annular tangles can be thought of as
raising and lowering operators according to whether they increase or decrease the weight of their inputs.
With this in mind it is natural to look in a planar algebra for lowest weight vectors, i.e. ones annihilated 
by any lowering operator. The vector space of lowest weight vectors of a given weight is clearly invariant
under annular Temperley-Lieb tangles which preserve weight and these tangles could hardly be simpler, consisting
of powers of the rotation tangle. Decomposing lowest weight vectors according to this cyclic group action
we see that the lowest weight vectors may also be assumed to be eigenvectors for the rotation.

Thus we look in planar algebras for elements $R$ having the following three properties:
{
\tikzstyle{shaded}=[fill=red!10!blue!20!gray!40!white]
\tikzstyle{empty box}=[circle, draw, thick, fill=white, opaque, inner sep=2mm]
\tikzstyle{annular}=[scale=.6, inner sep=1mm, baseline]
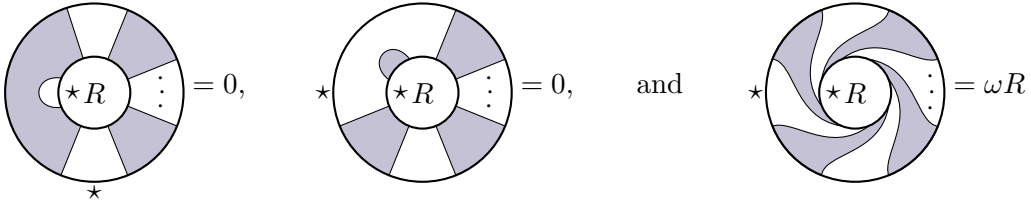
\begin{equation*}
\begin{tikzpicture}[annular]
	\node at (-90:2.2cm) {$\star$};
	\clip (0,0) circle (2cm);
	\draw[shaded] (0,0)--(-68:4cm) arc (-68:-22:4cm) --(0,0); 
	\draw[shaded] (0,0)--(-112:4cm) arc (-112:-252:4cm) --(0,0); 
	\draw[shaded, fill=white] (0,0) .. controls ++(215:2cm) and ++(145:2cm) .. (0,0);
	\draw[shaded] (0,0)--(22:4cm) arc (22:68:4cm) --(0,0); 
	\draw[ultra thick] (0,0) circle (2cm);
	\node at (0,0)  [empty box] (T) {$R$};
	\node at (T.180) [right] {$\star$};
	\node at (0:1.5cm) {$\cdot$};
	\node at (12:1.5cm) {$\cdot$};
	\node at (-12:1.5cm) {$\cdot$};
\end{tikzpicture}
 = 0, \qquad 
\begin{tikzpicture}[annular]
	\node at (180:2.2cm) {$\star$};
	\clip (0,0) circle (2cm);
	\draw[shaded] (0,0) .. controls ++(170:2cm) and ++(100:2cm) .. (0,0);
	\draw[shaded] (0,0)--(-158:4cm) arc (-158:-112:4cm) --(0,0); 
	\draw[shaded] (0,0)--(-68:4cm) arc (-68:-22:4cm) --(0,0); 
	\draw[shaded] (0,0)--(22:4cm) arc (22:68:4cm) --(0,0); 
	\draw[ultra thick] (0,0) circle (2cm);
	\node at (0,0)  [empty box] (T) {$R$};
	\node at (T.180) [right] {$\star$};
	\node at (0:1.5cm) {$\cdot$};
	\node at (12:1.5cm) {$\cdot$};
	\node at (-12:1.5cm) {$\cdot$};
\end{tikzpicture}
 = 0, \qquad \text{and} \qquad
\begin{tikzpicture}[annular]
	\node at (-180:2.2cm)  {$\star$};
	\clip (0,0) circle (2cm);
	\foreach \rho in {0,90,180,270} {
		\draw[shaded,rotate=\rho] (0,0)--(158:0.8cm) .. controls (112:1.5cm) and (62:1.5cm) .. (68:2cm) arc (68:22:2cm) .. controls (22:1.5cm) and (68:1.5cm) .. (112:0.8cm) -- cycle;
	}
	\draw[ultra thick] (0,0) circle (2cm);
	\node at (0,0)  [empty box] (T) {$R$};
	\node at (T.180) [right] {$\star$};
	\node at (0:1.7cm) {$\cdot$};
	\node at (12:1.7cm) {$\cdot$};
	\node at (-12:1.7cm) {$\cdot$};
\end{tikzpicture}
 = \omega R
\end{equation*}
}
where $\omega$ is some root of unity.  (Combined, these imply that all caps are zero.)

Each such $R$ generates an irreducible annular submodule of the planar algebra linearly spanned by the ``annular consequences''
of $R$ which are diagrams where you add (possibly nested) cups around the outside of the generator, as in:
$$
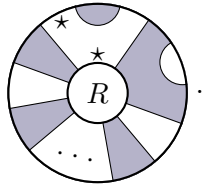
\begin{tikzpicture}[baseline = 0cm]
	\clip (0,0) circle (1.2cm);
	\draw[unshaded] (0,0) circle (1.2cm);
	\draw[shaded] (90:1.2) circle (.3cm);
	\draw[shaded] (0,0)--(130:1.3cm) arc (130:160:1.3cm) --(0,0); 
	\draw[shaded] (0,0)--(190:1.3cm) arc (190:220:1.3cm) --(0,0); 
	\draw[shaded] (0,0)--(55:1.3cm) arc (55:-20:1.3cm) --(0,0); 
	\draw[shaded] (0,0)--(-50:1.3cm) arc (-50:-80:1.3cm) --(0,0); 
	\draw[thick, unshaded] (0,0) circle (.4cm);
	\draw[unshaded] (20:1.2) circle (.3cm);
	\node at (90:.52) {$\star$};
	\node at (117:1.05) {$\star$};
	\node at (-.5,-.8) {$\cdot$};
	\node at (-.3,-.85) {$\cdot$};
	\node at (-.1,-.9) {$\cdot$};
	\node at (0,0) {$R$};
	\draw[ultra thick] (0,0) circle (1.2cm);
\end{tikzpicture}\;.
$$

It is easy to see that, if $R_1$ and $R_2$ are orthogonal lowest weight vectors then the entire annular submodules built on
$R_1$ and $R_2$ are orthogonal so we get a nice orthogonal decomposition of a planar algebra as a direct sum of annular
submodules, unique up to a choice of basis of rotational eigenvectors in the lowest weight spaces.

In fact, we have a complete classification of the possible irreducible representations along these lines.

\begin{theorem}{\cite{MR1659204,MR1929335}}
 Suppose $\delta >2$.
The irreducible unitary representations of annular Temperley-Lieb
up to unitary equivalence split into two classes depending on whether the lowest weight $n$ is 0 or at least 1. 
When $n \geq 1$, the irreps are characterized by $\omega$, a unit complex number with $\omega^n=1$, giving the eigenvalue of the rotation operator on the one-dimensional lowest weight space.

\end{theorem}

For $n=0$ there is also a characterization of the irreps, but it is slightly more complicated.  For a subfactor planar algebra, there's exactly one $n=0$ annular subrepresentation, which is the Temperley-Lieb subalgebra.  This is the only $n=0$ irrep which we'll ever meet.

This analysis would be of limited use if the various annular consequences of a lowest weight vector had complicated linear
relations among themselves. Fortunately in index $>4$ the annular consequences are all linearly independent as was proved
in \cite{MR1929335,MR1659204}. Thus there is an easy combinatorial argument showing that in an irrep generated by a weight $n>0$ lowest weight vector the dimension of the $n+k$-box space is
$\binom{2n+2k}{k}$. (Orient each cup so the lowest weight vector appears to its right; now we need to choose $k$ out of the $2n+2k$ boundary points to be the beginnings of cups.)

This annular decomposition of a planar algebra reveals a totally different structure from what one would obtain
by thinking in terms of algebras, $\operatorname{hom}$-spaces and the like. The dimensions of the lowest weight spaces and
the multiplicities of the rotational eigenvectors are invariants of a planar algebra largely invisible in that picture.
In fact the dimensions of the lowest weight spaces can be determined directly from the principal graph, using the formulas above for the dimensions of the irreps. First, $\dim P_{n,\pm} = \omega_{n}$, the number of loops of length $2n$ on the principal graph $\Gamma$ (both principal graphs give the same counts) based at the starred vertex. Then $a_n$, the multiplicity of annular modules with lowest weight $n$, is given by
$$a_n = \sum_{r=0}^n (-1)^{r-n} \frac{2n}{n+r} \binom{n+r}{n-r} \omega_r.$$
We call this sequence of dimensions the ``annular multiplicities''  of the planar algebra.

The planar algebras of subfactors of small index are often planarly generated by a single (lowest) lowest weight vector
that is thus a rotational eigenvalue. This has been a cornerstone of our analysis, and the interaction between the
annular picture and the algebra one has been rich in consequence.  It is worth noting that rotation also plays a key role in the study of higher Frobenius-Schur indicators of pivotal categories (which are traces of rotation operators) in the work of Ng and Schauenberg \cite{MR2213320, MR2381536, MR2313527}.

Next in complexity after the annular tangles are the quadratic tangles, with two input discs. These do not form a category, and it is much harder to extract useful information from them. Nevertheless there is a particularly important quadratic tangle $S \circ T$ which connects two $n$-boxes by $n-1$ strings, producing an $n+1$-box:
$$
S \circ T =
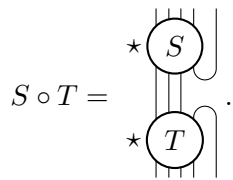
\begin{tikzpicture}[baseline=6]
\draw (-0.25, 1.5) -- (-0.25, -0.75);
\draw (-0.08, 1.5) -- (-0.08, -0.75);
\draw (0.08, 1.5) -- (0.08, -0.75);
\draw (0.25, 1.5) -- (0.25, 0.7) arc (180:360:0.15) -- +(0,0.8);
\draw (0.25, -0.75) -- (0.25, 0.05) arc (180:0:0.15) -- +(0,-0.8);
\node[draw, thick, circle, fill=white] (S) at (0,1) {$S$};
\node[left=-2pt] at (S.180) {$\star$};
\node[draw, thick, circle,  fill=white] (T) at (0,-0.25) {$T$};
\node[left=-2pt] at (T.180) {$\star$};
\end{tikzpicture}\;.
$$

Moreover it is possible, through a careful analysis of the dual basis for annular consequences, to compute explicit formulas for the projection of $S \circ T$ to the span of annular consequences of $n$-boxes. 
In some cases, by dimension counting arguments based on annular multiplicities, it is possible to show that $S \circ T$ must actually lie in annular consequences of $n$-boxes, in which case these formulas produce extremely strong constraints on the planar algebra. Further, one can compute the inner products between $S \circ T$  and some
of its rotations which tightens the constraints and allows larger annular multiplicities to be considered. This approach has been developed in \cite{math/1007.1158} and \cite{1208.3637} and plays a key role in several of the constructions of small index subfactors.

\subsubsection{Graph planar algebras}
\label{sec:GPA}
In this section we describe a plethora of planar algebras introduced in \cite{MR1865703} which are constructed from graphs and are close relatives of the path algebras from \cite{MR996454, MR1642584}.  These are not subfactor planar algebras, because the dimension of $P_0$ is the number of vertices of the graph instead of $1$, but nonetheless they are very useful because any planar subalgebra with a $1$-dimensional $0$-box space will be a subfactor planar algebra.  Furthermore, graph planar algebras give a way to describe connections, biunitarity, and flatness in planar algebraic language.

Given a bipartite graph $\Gamma$ with marked point $\star$ (e.g. one of the two principal graphs of a subfactor), we can define the `graph planar algebra' $\cG(\Gamma)$. The spaces $\cG(\Gamma)_{n,\pm}$ are given by functionals on based loops of length $2n$, with the base point either at an even depth or odd depth vertex depending on whether the sign $\pm$ is $+$ or $-$.

Alternatively, given a 4-partite graph ($\Gamma$, $\Gamma'$) as described in \S \ref{sec:cell-calculus}, we can construct another `two-sided' graph planar algebra $\cG(\Gamma, \Gamma')$. This is a more general type of planar algebra, with four shadings instead of two, and with spaces indexed by loops on a square. The space $\cG(\Gamma,\Gamma')_\kappa$ is the vector space of functionals on based loops on the 4-partite graph that descend to the specified loop $\kappa$. The graph planar algebra $\cG(\Gamma)$ described in the first paragraph above can be recovered as a subalgebra by only considering loops on the square that remain on a single edge.

We won't describe the action of planar tangles in detail here, but only say that it depends on the Frobenius-Perron eigenvector of the graph. Details can be found in, e.g.,  \cite{MR1865703} and \cite[\S 1]{1205.2742}.

\begin{theorem}[\cite{MR2812459, gpa}]
There is a canonical embedding $\cP \to \cG(\Gamma(\cP))$ for any subfactor planar algebra $\cP$.
\end{theorem}

We can interpret a \emph{connection} as an element $K \in \cG(\Gamma,\Gamma')_{\square}$. The condition that a connection is biunitary becomes the two planar equations
\tikzstyle{conn}=[circle, draw, thick, fill=white, opaque, inner sep = .7mm]
\tikzstyle{gauge}=[circle, draw, thick, fill=white, opaque, inner sep = .7mm]
\tikzstyle{invgauge}=[circle, draw, thick, fill=black!70!white, inner sep = .7mm]
$$
\begin{tikzpicture}[scale=.7, baseline=0]
	\clip (-1,-1.8) rectangle (1,1.8);

	\fill[fill=red-unshaded] (-1,-2)--(-.3,-2)--(-.3,2)--(-1,2);
	\fill[fill=red-shaded] (-.3,-2)--(-.3,-1)--(.3,-1)--(.3,-2);
	\fill[fill=red-shaded] (-.3,2)--(-.3,1)--(.3,1)--(.3,2);
	\fill[blue-unshaded] (1,-2)--(.3,-2)--(.3,3)--(1,2);
	\fill[blue-shaded] (-.3,-1)--(-.3,1)--(.3,1)--(.3,-1);
	
	\draw (-.3,-2)--(-.3,2);
	\draw (.3,-2)--(.3,2);
	
	\node[conn, minimum size=8mm] at (0,1) {$K$};
	\node[conn,  minimum size=8mm] at (0,-1) {$K^*$};
	
\end{tikzpicture}
=
\begin{tikzpicture}[scale=.7, baseline=0]
	\clip (-1,-1.8) rectangle (1,1.8);

	\fill[fill=red-unshaded] (-1,-2)--(-.3,-2)--(-.3,2)--(-1,2);
	\fill[fill=red-shaded] (-.3,-2)--(-.3,2)--(.3,2)--(.3,-2);
	\fill[blue-unshaded] (1,-2)--(.3,-2)--(.3,3)--(1,2);
	
	\draw (-.3,-2)--(-.3,2);
	\draw (.3,-2)--(.3,2);
\end{tikzpicture}
\quad
\text{and}
\quad
\begin{tikzpicture}[scale=.7,baseline=0]
	\clip (-1.9,-1) rectangle (1.9,1);

	\fill[red-shaded] (-2,1) rectangle (2,0);
	\fill[blue-shaded] (-2,-1) rectangle (2,0);

	\filldraw[fill=blue-unshaded] (-1,.3) .. controls (0,.6) .. (1,.3)--(1,-.3) .. controls (0,-.6) .. (-1,-.3) -- (-1,.3);
	\filldraw[fill=red-unshaded] (-2,.6) -- (-1,.3) -- (-1,-.3) -- (-2,-.6);
	\filldraw[fill=red-unshaded] (2,.6) -- (1,.3) -- (1,-.3) -- (2,-.6);
	
	\node[conn, minimum size=8mm] at (-1,0) {$K$};
	\node[conn, minimum size=8mm] at (1,0) {$K^*$};

\end{tikzpicture}
=
\begin{tikzpicture}[scale=.7, baseline=0]
	\clip (-1.9,-1) rectangle (1.9,1);
	
	\fill[red-shaded] (-2,1) rectangle (2,0);
	\fill[blue-shaded] (-2,-1) rectangle (2,0);

	\filldraw[fill=red-unshaded] (-2,-.6) .. controls (0,0) .. (2,-.6)--(2,.6) .. controls (0,0) .. (-2,.6);
\end{tikzpicture} \; .
$$

Any biunitary connection has a corresponding shaded planar algebra of \emph{flat elements}. An element $x \in \cG(\Gamma)$ is flat if there is a corresponding $y \in \cG(\Gamma')$ such that 
\newcommand{\ybelow}{\begin{tikzpicture}[scale=.8, baseline]
	\clip (-.9,-2.9) rectangle (3.9,1.9);
	\filldraw[fill=red-unshaded] (0,2) --(0,0).. controls (0,-.8) .. (1,-1) -- (1,2)--(0,2);
	\filldraw[fill=red-shaded] (1,2)--(1,-1) .. controls (2,-1.2) and (1,-1.8) .. (2,-2) -- (2,2)--(1,2);
	\filldraw[fill=red-unshaded] (4,2)--(2,2)--(2,-2) .. controls (3,-2.2) .. (3,-3)--(4,-3)--(4,2);
	\filldraw[fill=blue-shaded] (3,-3)--(3,-3) .. controls (3,-2.2) .. (2,-2) -- (2,-3)--(3,-3);
	\filldraw[fill=blue-unshaded] (2,-2)--(2,-3)--(1,-3)--(1,-1) .. controls (2,-1.2) and (1,-1.8) .. (2,-2);
	\filldraw[fill=blue-shaded] (0,2)-- (0,0).. controls (0,-.8) .. (1,-1) -- (1,-3) -- (-1,-3)--(-1,2)--(0,2);
	\node[conn, minimum size=.7cm, shape=circle, fill=white, draw]  at (1,-1) {\tiny$K$};
	\node[conn, minimum size=.7cm, shape=circle, fill=white, draw]  at (2,-2) {\tiny$K^*$};
	\node[conn, minimum size=1.2cm, shape=rectangle, rounded corners = 4mm, fill=white, draw]  at (1.5,.5) {$y$};
\end{tikzpicture}
}
\newcommand{\xabove}{\begin{tikzpicture}[scale=.8, baseline]
	\clip (-.9,-2.9) rectangle (3.9,1.9);
	\filldraw[fill=red-unshaded] (0,2) .. controls (0,1.2) .. (1,1) -- (1,2)--(0,2);
	\filldraw[fill=red-shaded] (1,2)--(1,1) .. controls (2,.8) and (1,.2) .. (2,0) -- (2,2)--(1,2);
	\filldraw[fill=red-unshaded] (4,2)--(2,2)--(2,0) .. controls (3,-.2) .. (3,-1) -- (3,-3)--(4,-3)--(4,2);
	\filldraw[fill=blue-shaded] (3,-3)--(3,-1) .. controls (3,-.2) .. (2,0) -- (2,-3)--(3,-3);
	\filldraw[fill=blue-unshaded] (2,0)--(2,-3)--(1,-3)--(1,1) .. controls (2,.8) and (1,.2) .. (2,0);
	\filldraw[fill=blue-shaded] (0,2) .. controls (0,1.2) .. (1,1) -- (1,-3) -- (-1,-3)--(-1,2)--(0,2);
	\node[conn, minimum size=.7cm, shape=circle, fill=white, draw]  at (1,1) {\tiny$K$};
	\node[conn, minimum size=.7cm, shape=circle, fill=white, draw]  at (2,0) {\tiny$K^*$};
	\node[conn, minimum size=1.2cm, shape=rectangle, rounded corners = 4mm, fill=white, draw]  at (1.5,-1.5) {$x$};
\end{tikzpicture}
}
\begin{equation*}
\xabove = \ybelow.
\end{equation*}

The 0-box space of the flat subalgebra consists of those elements on which the double ring operator acts by $\delta^2$. Since the Frobenius-Perron eigenvalue of the graph is always multiplicity free, we see the 0-box space is 1-dimensional, and hence the flat subalgebra is always evaluable.

One says that the biunitary connection itself is flat exactly if the principal graph pair of the flat subalgebra is the same as the graph pair $(\Gamma, \Gamma')$ that we started with \cite{MR996454,MR1308617}.

\section{Classification in index \texorpdfstring{$\leq 4$}{<4}}
\label{sec:index-leq-4}

\subsection{Subfactors of index less than 4}
\label{sec:index-lt-4}
The principal graph of a subfactor of index less than $4$ is a graph of index less than $2$.  It is well known that the only bipartite graphs of norms less than $2$ are the ADE Dynkin diagrams.  Furthermore, it is not difficult to see by dimension considerations that the starred vertex must be the vertex furthest from any branch vertices (otherwise one of the bimodules would have dimension less than $1$).  Thus in order to classify subfactor planar algebras of index strictly less than $4$ we only need to classify all subfactor planar algebras whose principal graph is a fixed Dynkin diagram.  Since these graphs are all finite, this also gives a complete classification of subfactors of the hyperfinite $\II_1$ via Popa's reconstruction theorem.

\begin{theorem}
The number of subfactor planar algebras realizing each of the $ADE$ Dynkin diagrams is given by the following table:

\begin{center}\begin{tabular}{c|c|c|c|c|c|c}
Principal graph & $A_n$ & $D_{2n+1}$  & $D_{2n}$ & $E_6$ & $E_7$ & $E_8$ \\ \hline
Realizations & $1$ & $0$ & $1$ & $2$ & $0$ & $2$
\end{tabular} \end{center}
\end{theorem}
\begin{proof}
See \cite{MR996454} for an outline, and \cite{MR999799, MR1193933, MR1145672, MR1313457, MR1308617} for more details. 
\end{proof}

Thus, below index 4, there are two infinite families of subfactors, $A_n$ and $D_{2n}$, and four sporadic examples. For all of these subfactors, both principal graphs are the same, and all even bimodules are self-dual except for the two vertices at the end of $D_{4n}$ which are dual to each other. The two versions of the sporadic cases, $E_6$ and $E_8$, are complex conjugate to each other (depending on a choice of a third or a fifth root of unity). 

It is a curious feature of the classification that there are no $D_m$ subfactors with $m$ odd, nor any $E_7$ subfactor.
It is easy to see that some of the $D_m$ cannot exist; the two vertices at the end of $D_{2^k+1}$ have dimensions which are not algebraic integers (dimensions of bimodules in finite depth subfactors must be algebraic integers because they are eigenvalues of integer fusion matrices). In the general case, Izumi showed \cite{MR1145672} that there are no consistent fusion rules for any of the $D_m$ with $m$ odd, or for $E_7$.  (We can easily rule out $E_7$, by noting that there would need to be a bimodule with dimension less 2 but not of the form $2 \cos(\pi/n)$.)

The $A_n$ are the easiest planar algebras to understand; they're just the image of Temperley-Lieb (the kernel of the map consists essentially of the $n$-th Jones-Wenzl idempotent).  The $D_{2n}$ can be quickly constructed from Temperley-Lieb using an orbifold construction \cite{MR1308617}.  Namely there is a non-trivial bimodule of dimension $1$, and you can take a ``quotient'' which identifies this bimodule with the trivial. The exceptional cases are more difficult to describe; the first constructions were in \cite{MR1193933} for $E_6$ and \cite{MR1313457} for $E_8$.  Another construction is via conformal inclusions: $E_6$ can be built starting from $SU(2)_{10} \subset \operatorname{Spin}(5)$ and $E_8$ comes from $SU(2)_{28} \subset G_2$ \cite{MR1617550}.

The planar algebras for the ADE subfactors are constructed in \cite{MR1929335}.  A purely generators-and-relations description for the $D_{2n}$ planar algebras is given in \cite{MR2559686}, and Bigelow \cite{MR2577673} gave generators and relations for $E_6$ and $E_8$, and gave explicit bases for all box spaces for all subfactors of index less than $4$.

\subsection{Subfactors of index exactly 4}
\label{sec:index-4}

The classification of subfactor planar algebras with index exactly 4 is outlined in \cite[p. 231]{MR1278111}, with parts of the proof in \cite{MR999799, MR1213139,MR1308617,MR1340721,MR1054961}.  On the one hand, this classification is simpler because all the examples are related to ordinary groups.  On the other hand, it is more technically delicate because there are infinite graphs with norm $2$.

At index exactly 4,  the possible principal graphs are all affine simply-laced Dynkin diagrams.  Furthermore, it is not difficult to work out for each graph the possible location for the vertex $\star$, up to symmetry (see Figure \ref{fig:affine}).   So we need only determine how many realizations there are for each of these starred affine Dynkin diagram.  Note that the $A_{2n}^{(1)}$ are not bipartite and so cannot come from subfactors.

\begin{theorem}
The number of subfactor planar algebras realizing each of the starred affine Dynkin diagrams is given by the following table:

\begin{center}\begin{tabular}{c|c|c|c|c|c|c|c|c}
Principal graph & $A_{2n}^{(1)}$ & $A_{2n-1}^{(1)}$ & $D_n^{(1)}$ & $E_6^{(1)}$ & $E_7^{(1)}$ & $E_8^{(1)}$ & $A_\infty$ & $D_\infty$ \\ \hline
Realizations & $0$ & $n$ & $n-2$ & $1$ & $1$ & $1$ & $1$ & $1$
\end{tabular} \end{center}
\end{theorem}

Here $A_\infty$ just comes from Temperley-Lieb.  In fact, at this value of the parameter, it can be thought of as coming from the compact group $\mathrm{SU}(2)$.  That is to say, the elements of Temperley-Lieb can be interpreted as maps between tensor powers of the $2$-dimensional representation, with the cap being the determinant map.  Any subgroup of $\mathrm{SU}(2)$ containing the center also gives a subfactor, and the principal graph corresponding to this subgroup is the graph for tensoring with the $2$-dimensional representation as in the McKay correspondence \cite{MR604577}. The grading into odd and even parts is given by the central character.  This construction can be modified slightly, by twisting the even part by an element of the third group cohomology satisfying certain conditions.

In fact, all subfactors of index $4$ come from such a subgroup $G \subset \mathrm{SU}(2)$ together with a cohomological twist.  In particular, the $A_{2n-1}^{(1)}$ come from the binary cyclic groups, the $D_n^{(1)}$ come from the binary dihedral groups, the exceptional E's come from the binary tetrahedral, octahedral, and icosahedral groups, $A_\infty$ comes from $\mathrm{SU}(2)$ itself, $A_\infty^{(1)}$ comes from the infinite binary cyclic group, and $D_\infty$ comes from the infinite binary dihedral group.  The $n$ different subfactor planar algebras with principal graph $A_{2n-1}^{(1)}$ are classified by $H^3(\Integer/ n \Integer, \mathbb{C}^\times)$ and the $n-2$ examples with principal graph $D_n^{(1)}$ are classified by certain elements of $H^3(D_{2(n-2)}, \mathbb{C}^\times)$.   The dual data is given by taking the dual of the corresponding representation of a subgroup of $SU(2)$.

There are some interesting phenomena at index $4$ which did not occur at smaller indices. Note that $A_n^{(1)}$ and $A_\infty^{(1)}$ are reducible, because the starred vertex has two neighbors (so ${}_N M_M$ has two irreducible summands).  Further note that there are several infinite depth examples.  Fortunately these infinite depth examples satisfy a property called ``strong amenability'', and Popa showed that in the strongly amenable case the standard invariant still determines the subfactor of the hyperfinite.  Finally, note that $4 = 2 \cdot 2$ is a product of allowed indices.  This means we can create new subfactors by combining smaller ones.  The universal such combination is called the Fuss-Catalan subfactor \cite{MR1437496}, and in this case gives $D_\infty$.  In some sense the $D_n^{(1)}$ also come from combining index $2$ subfactors.

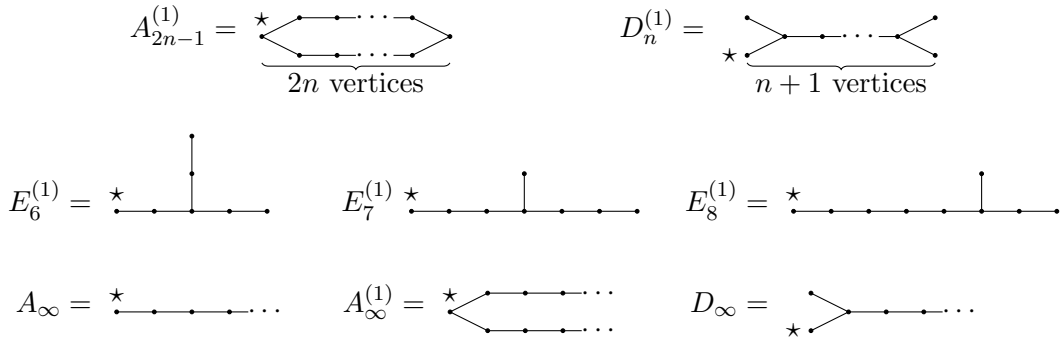
\begin{figure}
\begin{align*}
{A}_{2n-1}^{(1)} & = \begin{tikzpicture}[baseline=0, scale=.5]
\filldraw (0,0) circle (.5mm) node [above] {$\star$};
\filldraw (1,.5) circle (.5mm);
\filldraw (2,.5) circle (.5mm);
\node at (3,.5) {$\cdots$};
\filldraw (1,-.5) circle (.5mm);
\filldraw (2,-.5) circle (.5mm);
\node at (3,-.5) {$\cdots$};
\filldraw (4,.5) circle (.5mm);
\filldraw (4,-.5) circle (.5mm);
\filldraw (5,0) circle (.5mm);
\draw (0,0)--(1,.5)--(2.5,.5);
\draw (0,0)--(1,-.5)--(2.5,-.5);
\draw (3.5,.5)--(4,.5)--(5,0)--(4,-.5)--(3.5,-.5);
\draw[decorate,decoration={brace, mirror}] (0,-.7)--(5,-.7) node[midway, below] {$2n$ vertices};
\end{tikzpicture}
& 
{D}_{n}^{(1)} & = \begin{tikzpicture}[baseline=0, scale=.5]
\filldraw (0,-.5) circle (.5mm) node [left] {$\star$};
\filldraw (1,0) circle (.5mm);
\filldraw (0,.5) circle (.5mm);
\filldraw (2,0) circle (.5mm);
\node at (3,0) {$\cdots$};
\filldraw (4,0) circle (.5mm);
\filldraw (5,.5) circle (.5mm);
\filldraw (5,-.5) circle (.5mm);
\draw (0,.5)--(1,0)--(2.5,0);
\draw (1,0)--(0,-.5);
\draw (3.5,0)--(4,0)--(5,.5);
\draw (4,0)--(5,-.5);
\draw[decorate,decoration={brace, mirror}] (0,-.7)--(5,-.7) node[midway, below] {$n+1$ vertices};
\end{tikzpicture}
\end{align*}

\begin{align*}
{E}_6^{(1)} & = \begin{tikzpicture}[baseline=0, scale=.5]
\filldraw (0,0) circle (.5mm) node [above] {$\star$};
\filldraw (1,0) circle (.5mm);
\filldraw (2,0) circle (.5mm);
\filldraw (3,0) circle (.5mm);
\filldraw (4,0) circle (.5mm);
\filldraw (2,1) circle (.5mm);
\filldraw (2,2) circle (.5mm);
\draw (0,0)--(4,0);
\draw (2,0)--(2,2);
\end{tikzpicture}& {E}_7^{(1)} & \begin{tikzpicture}[baseline=0, scale=.5]
\filldraw (-1,0) circle (.5mm) node [above] {$\star$};
\filldraw (0,0) circle (.5mm);
\filldraw (1,0) circle (.5mm);
\filldraw (2,0) circle (.5mm);
\filldraw (3,0) circle (.5mm);
\filldraw (4,0) circle (.5mm);
\filldraw (5,0) circle (.5mm);
\filldraw (2,1) circle (.5mm);
\draw (-1,0)--(5,0);
\draw (2,0)--(2,1);
\end{tikzpicture}&{E}_8^{(1)}& = \begin{tikzpicture}[baseline=0, scale=.5]
\filldraw (-3,0) circle (.5mm) node [above] {$\star$};
\filldraw (-2,0) circle (.5mm);
\filldraw (-1,0) circle (.5mm);
\filldraw (0,0) circle (.5mm);
\filldraw (1,0) circle (.5mm);
\filldraw (2,0) circle (.5mm);
\filldraw (3,0) circle (.5mm);
\filldraw (4,0) circle (.5mm);
\filldraw (2,1) circle (.5mm);
\draw (-3,0)--(4,0);
\draw (2,0)--(2,1);
\end{tikzpicture}
\\\\
{A}_\infty &= \begin{tikzpicture}[baseline=0, scale=.5]
\filldraw (0,0) circle (.5mm) node [above] {$\star$};
\filldraw (1,0) circle (.5mm);
\filldraw (2,0) circle (.5mm);
\filldraw (3,0) circle (.5mm);
\node at (4,0) {$\cdots$};
\draw (0,0)--(3.5,0);
\end{tikzpicture}&
{A}_{\infty}^{(1)} &= \begin{tikzpicture}[baseline=0, scale=.5]
\filldraw (0,0) circle (.5mm) node [above] {$\star$};
\filldraw (1,.5) circle (.5mm);
\filldraw (2,.5) circle (.5mm);
\filldraw (3,.5) circle (.5mm);
\node at (4,.5) {$\cdots$};
\filldraw (1,-.5) circle (.5mm);
\filldraw (2,-.5) circle (.5mm);
\filldraw (3,-.5) circle (.5mm);
\node at (4,-.5) {$\cdots$};
\draw (0,0)--(1,.5)--(3.5,.5);
\draw (0,0)--(1,-.5)--(3.5,-.5);
\end{tikzpicture}
&
{D}_\infty &= \begin{tikzpicture}[baseline=0, scale=.5]
\filldraw (0,-.5) circle (.5mm) node [left] {$\star$};
\filldraw (1,0) circle (.5mm);
\filldraw (0,.5) circle (.5mm);
\filldraw (2,0) circle (.5mm);
\filldraw (3,0) circle (.5mm);
\node at (4,0) {$\cdots$};
\draw (0,.5)--(1,0)--(3.5,0);
\draw (1,0)--(0,-.5);
\end{tikzpicture}
\end{align*}
\caption{The affine Dynkin diagrams together with the only possible location for $\star$ up to symmetry.}
\label{fig:affine}
\end{figure}

\section{Index from 4 to 5}

Over the last decade, the classification of small index subfactors has gradually been pushed up to index 5. In this section, we begin in \S \ref{sec:statement} with the statement of the classification theorem. After that, we briefly discuss in \S \ref{sec:examples} the actual examples occurring in this range, along with the historical development  of the classification. In \S \ref{sec:proving} we explain the major techniques involved in proving the theorem.

\subsection{Statement of the theorem}
\label{sec:statement}

The following theorem represents the culmination of work of many authors. 
It is proved in the papers \cite{1007.1730, index5-part2, index5-part3, index5-part4}, which in turn rely on constructions given in \cite{MR1686551, 0909.4099, MR999799, MR1832764}, 
and a uniqueness result from \cite{1102.2052}. 
The combinatorial enumeration of possible principal graphs in \cite{1007.1730} builds on the earlier work of \cite{MR1317352}. 
The classification relies on the powerful number theoretic obstructions proved in \cite{1004.0665} which generalize earlier work of \cite{MR2472028}.  
This theorem supersedes an earlier classification below index $3+\sqrt{3}$ from \cite{MR1317352, MR1625762, MR1686551, MR2472028, 0909.4099} which we discuss in more detail in \S \ref{sec:proving}.

\begin{theorem}
\label{main-theorem}
If an extremal subfactor has index between $4$ and $5$ its standard invariant is Temperley-Lieb (with principal graph $A_\infty$) or it is one of the following ten planar algebras.

\begin{itemize}
\item The Haagerup planar algebra \cite{MR1686551}, with index $\frac{5+\sqrt{13}}{2}$ and principal bigraph pair $$\left(\bigraph{bwd1v1v1v1p1v1x0p0x1v1x0p0x1duals1v1v1x2v2x1} \bigraph{bwd1v1v1v1p1v1x0p1x0duals1v1v1x2}\right)$$ and its dual.
\item The extended Haagerup planar algebra \cite{0909.4099}, with index $\frac{8}{3}+\frac{2}{3} \operatorname{Re} \sqrt[3]{\frac{13}{2} \left(-5-3 i \sqrt{3}\right)}$ and principal bigraph pair $$\left(\bigraph{bwd1v1v1v1v1v1v1v1p1v1x0p0x1v1x0p0x1duals1v1v1v1v1x2v2x1}, \bigraph{bwd1v1v1v1v1v1v1v1p1v1x0p1x0duals1v1v1v1v1x2}\right)$$ and its dual.
\item The Asaeda-Haagerup planar algebra \cite{MR1686551}, with index $\frac{5+\sqrt{17}}{2}$ and principal bigraph pair $$\left(\bigraph{bwd1v1v1v1v1v1p1v1x0p0x1v1x0p0x1p0x1v1x0x0v1duals1v1v1v1x2v2x1x3v1}, \bigraph{bwd1v1v1v1v1v1p1v0x1p0x1v0x1v1duals1v1v1v1x2v1}\right)$$ and its dual.
\item The 3311 Goodman-de la Harpe-Jones planar algebra \cite{MR999799}, with index $3+\sqrt{3}$ and principal bigraph pair $$\left(\bigraph{bwd1v1v1v1p1p1v1x0x0v1duals1v1v1x2x3v1}, \bigraph{bwd1v1v1v1p1p1v1x0x0v1duals1v1v1x2x3v1}\right)$$ and its dual (since it is not self-dual despite having the same principal and dual principal graphs \cite{MR1355948}).
\item Izumi's self-dual 2221 planar algebra \cite{MR1832764} and its complex conjugate, with index $\frac{5+\sqrt{21}}{2}$ and principal bigraph pair $$\left(\bigraph{bwd1v1v1p1p1v1x0x0p0x1x0duals1v1v2x1}, \bigraph{bwd1v1v1p1p1v1x0x0p0x1x0duals1v1v2x1}\right).$$
\end{itemize}

\end{theorem}

Except for $A_\infty$ all of these graphs are finite, so by Popa's reconstruction theorem this gives a classification of non-$A_\infty$ subfactors of the hyperfinite $\II_1$ factor of index less than $5$.  On the other hand, $A_\infty$ is not amenable, and no classification of $A_\infty$ subfactors of the hyperfinite $\II_1$ factor with index less than $5$ is known (see \S \ref{sec:TLsubfactors}).

Furthermore, it is not too difficult to see that all non-extremal subfactor planar algebras of index between $4$ and $5$ are perturbations \cite{1111.1362,1009.0186, 1011.1808} of the $A_\infty^{(1)}$ planar algebra at index $4$.  There is exactly one such perturbation for each index value above $4$.  Since the index $4$ subfactor with principal graph $A_\infty^{(1)}$ is strongly amenable, these perturbed planar algebras can be realized uniquely as subfactors of the hyperfinite $\II_1$ \cite{MR1278111}.  These subfactors of the hyperfinite $\II_1$ are exactly the reducible subfactors of index greater than $4$ constructed in \cite{MR0696688}.

\subsection{Examples}
\label{sec:examples}

\subsubsection{The Haagerup subfactor} \label{Haagerup}
By the mid '90's the classification of index $\leq 4$ subfactors was complete, and it was also known that there are no principal graphs whose norm is between $2$ and $2.00659$ (the graph norm of $E_{10}$).  At the time the smallest known finite index subfactor above $4$ was the GHJ subfactor with index $3+\sqrt{3}$.  Then in
a tour de force in the mid '90's Haagerup showed \cite{MR1317352} that in fact there are no finite depth
subfactors between $4$ and $\frac{5+\sqrt {13}}{2}$ and announced an example
with that index.  The details of this construction appeared in joint work with Asaeda \cite{MR1686551}.  Their construction uses connections, and the main technique is to exploit the existence of nontrivial invertible objects.  It is easy to determine the connection attached to the generating object and to the invertible objects, and then a difficult calculation of a certain intertwiner shows that together these generate a finite tensor category.

There is a feeling that the first three subfactors in this classification are ``exotic'' (not strongly related to previously well understood subfactors) or ``sporadic'' (not belonging to any family).  See \cite{MR2468378, 1002.0168} for some evidence in this direction, and see \cite{MR2837122} for some interesting suggestions concerning how the Haagerup subfactor may not be sporadic after all.  In particular, the evidence increasingly suggests that the Haagerup subfactor does lie in an infinite (discrete) family of subfactors as explained below.  

In either case it is highly
desirable to have alternative constructions of these subfactors. Izumi \cite{MR1782145,MR1832764} 
developed a method based on endomorphisms and the Cuntz algebras, which  allowed a second construction of the Haagerup subfactor.   We will discuss Izumi's powerful techniques in more detail when we look at the $2221$ subfactor below.
Using this description, Izumi also computed the Drinfel'd center \cite{MR1151906, MR1107651} of the Haagerup fusion categories, and he constructed a subfactor with principal graph that looks rather like Haagerup's, but with the $3$-fold symmetry replaced by a $5$-fold symmetry:
$$\left(\bigraph{bwd1v1v1v1p1p1p1v1x0x0x0p0x1x0x0p0x0x1x0p0x0x0x1v1x0x0x0p0x1x0x0p0x0x1x0p0x0x0x1duals1v1v1x2x3x4v2x1x4x3}, \bigraph{bwd1v1v1v1p1p1p1v0x0x1x0p0x0x1x0p0x0x0x1p0x0x0x1duals1v1v1x2x3x4}\right)$$

A very interesting question is finding out when there are Izumi subfactors with any given finite abelian symmetry group.  Izumi gives explicit equations whose solutions each yield such a subfactor, but solving these equations remains difficult.  Evans and Gannon \cite{MR2837122} constructed an Izumi subfactor with symmetry group $\mathbb{Z}/7$, two subfactors with symmetry group $\mathbb{Z}/9$, and showed that there was none with symmetry group $\mathbb{Z}/3 \times \mathbb{Z}/3$.  Izumi has also constructed such subfactors with symmetry group $\mathbb{Z}/2 \times \mathbb{Z}/2$ and $\mathbb{Z}/4$.

Peters \cite{MR2679382} gave a third construction of the Haagerup subfactor which used planar algebra techniques to give another construction via ``generators and relations''.   These techniques were modified to give a construction of the extended Haagerup subfactor, so we will discuss Peters's techniques in more detail in a later section. 

The Haagerup subfactor has a rich structure which is only beginning to be fully understood.  For example, the Haagerup subfactor provides a negative answer to the number theoretic question of Etingof-Nikshych-Ostrik \cite[\S 2]{MR2183279}, asking whether every fusion category can be defined over a cyclotomic field, see \cite{1002.0168}.  Izumi and Grossman showed that there is an interesting quadrilateral of intermediate subfactors related to Haagerup \cite{MR2418197} (see \cite{1102.2631} for a ``trivial'' construction of this quadrilateral).  That is to say, there are two intermediate subfactors $N \subset P\subset M$ and $N \subset Q \subset M$ such that the inclusions $N \subset P$ and $N \subset Q$ give the Haagerup subfactor.  Grossman and Snyder \cite{1102.2631} showed that there is an additional fusion category Morita equivalent to the Haagerup fusion categories, and classified all subfactors coming from the Haagerup subfactor (several of which were previously unknown).

Subfactors coming from groups and quantum groups have a richer structure: they come from conformal field theories.  It would be very interesting to know whether the Haagerup subfactor also has this additional structure.  Evans and Gannon have given some evidence that there is a Haagerup conformal field theory and given some suggestions towards a construction \cite{MR2837122}.

\subsubsection{The Asaeda-Haagerup subfactor}

The second subfactor on this list to be constructed was the Asaeda-Haagerup subfactor of index $\frac{5+\sqrt{17}}{2}$.  This was constructed by Asaeda and Haagerup using a similar technique to their construction of the Haagerup subfactor, but with a more involved calculation.  This technique works again because there is a nontrivial invertible object.  The Asaeda-Haagerup subfactor is generally poorly understood.  It cannot be constructed directly using Izumi's techniques, since there are too few invertible objects. We do not yet have a planar algebraic description of Asaeda-Haagerup, and the Drinfel'd center of the Asaeda-Haagerup subfactor is unknown.

On the other hand, like the Haagerup subfactor, the Asaeda-Haagerup subfactor lies in an interesting quadrilateral of intermediate subfactors \cite{MR2812458}.   This was conjectured by Grossman and Izumi \cite{MR2418197}, and constructed by Asaeda and Grossman \cite{MR2812458}.  The other subfactor appearing in this quadrilateral has index $\frac{7+\sqrt{17}}{2}$ and it lies in yet another quadrilateral with another subfactor of index $\frac{9+\sqrt{17}}{2}$.  This suggests that the Asaeda-Haagerup subfactor has a very rich structure.  Subsequently Grossman and Snyder \cite{1202.4396} showed that the Brauer-Picard group of Asaeda-Haagerup is $\mathbb{Z}/2 \times \mathbb{Z}/2$, unlike the Haagerup case where the Brauer-Picard group is trivial.  This same calculation produced over a hundred new subfactors whose even parts are Morita equivalent to the even parts of the Asaeda-Haagerup subfactor.

\subsubsection{The extended Haagerup subfactor}
\label{sec:EH}
In 2009 Bigelow, Morrison, Peters, and Snyder \cite{0909.4099} gave a construction of a long suspected \cite{MR1633929} missing case: the extended Haagerup subfactor.   This construction roughly follows the outline of Peters's construction of the Haagerup subfactor \cite{MR2679382}, but with the addition of a key new idea.  Both constructions are based on foundational work from \cite{math/1007.1158}.  We will quickly sketch this construction.  The basic idea is to exhibit this planar algebra as the planar subalgebra of the graph planar algebra of its principal graph generated by a specific element satisfying certain concrete relations.  On the one hand, since it is a planar subalgebra of the graph planar algebra it is automatically unitary and non-trivial, while on the other hand using the generator and relations it is possible to prove that the $0$-box space is $1$-dimensional.  The latter argument requires an evaluation algorithm (simplifying diagrams into a canonical form) called the jellyfish algorithm.

Given a pair of candidate principal graphs, it is often possible to identify certain relations that must hold for certain elements of the corresponding planar algebra. In particular, following \cite{math/1007.1158}, if the principal graphs are $n$-supertransitive, and then have triple points, there must be a lowest weight rotational eigenvector in the $P_{n+1,+}$ space which satisfies
\begin{equation}
\label{eq:S2}
S^2  = (1-r) S + r f^{(n)}.
\end{equation}
Here $r$ is the ratio (greater than 1) of the two dimensions immediately past the branch point and $f^{(n)}$ is the $n$-strand Jones-Wenzl idempotent.
This is an easy calculation; the idempotents past the branch point must be linear combinations of $f^{(n+1)}$ and $S$, with known traces, summing to $f^{(n+1)}$. As we know that the planar algebra embeds in the graph planar algebra of the principal graph, we can then attempt to solve such equations in the graph planar algebra. (Often this is extremely difficult, requiring a combination of exact techniques and numerical methods. Since the solutions are often discrete, high precision numerical solutions can be approximated by algebraic numbers and then checked exactly.)

After finding such an element, one wants to show that the subalgebra of the graph planar algebra it generates is in fact a subfactor planar algebra. The subalgebra inherits positivity from the graph planar algebra, and the principal challenge is to show that the $0$-box space of the subalgebra is one dimensional. Equivalently, we need to find some relations amongst certain planar combinations of the generator, such that these relations suffice to evaluate any closed diagram built out of copies of the generator as a scalar multiple of the empty diagram.  Up to this point, the argument has very closely followed that developed by Peters in her thesis \cite{MR2679382} on the Haagerup subfactor planar algebra. However the technique used to evaluate closed diagrams there cannot be applied in the extended Haagerup case.

The main new technique introduced in \cite{0909.4099} is called the \emph{jellyfish algorithm}.
Typically, algorithms to simplify planar diagrams use some notion of the complexity of the diagram, and show that it is always possible to use some relation locally in the diagram to reduce the complexity. As an example, the relation for the Kauffman bracket in \S \ref{sec:TL} allows one to evaluate the Kauffman bracket of a link simply by monotically reducing the number of crossings. Kuperberg's analysis \cite{MR1403861} of the planar generators and relations for rank 2 Lie algebras similarly makes use of the fact that any closed planar trivalent graph has a face smaller than a hexagon, and he has relations that allow removing these in a way that reduces the total number of vertices.
The jellyfish algorithm is somewhat unusual in that its complexity function is nonlocal---it simplifies diagrams by reducing the distance from each generator to the outside of the diagram. (The name intends to evoke jellyfish floating to the surface of the ocean.) The ideas behind the jellyfish algorithm grew out of work on skein theories for the $D_{2n}$ and $E_n$ planar algebras below index $4$ \cite{MR2559686, MR2577673}.  These ideas have been further developed in \cite{1208.1564}, where it is noted that there is a deep connection between the jellyfish algorithm and a much older argument of Popa's \cite{MR1334479}.

In the extended Haagerup planar algebra, one finds the relations appearing in Figure \ref{fig:box-jellyfish}, where a box labelled $f^{(k)}$ denotes the $k$-strand Jones-Wenzl idempotent (c.f. \S \ref{sec:TL}).  

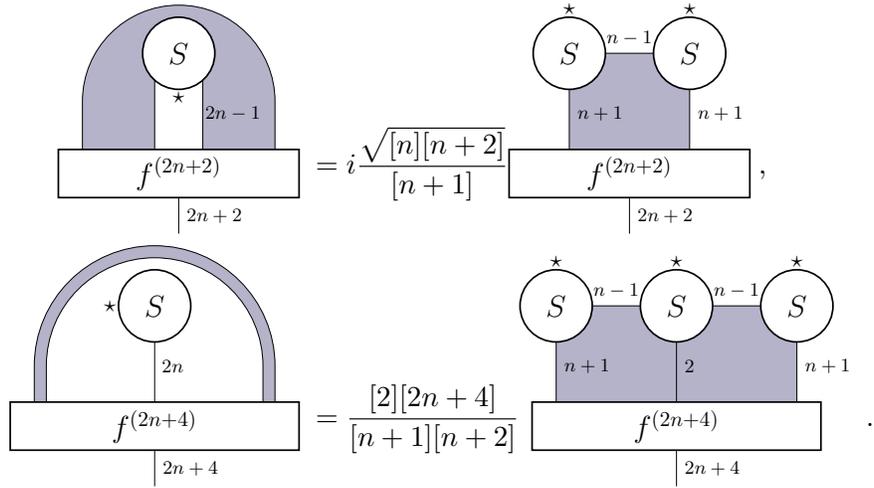
\begin{figure}[ht]
         \newcommand{\JW}[1]{f^{(#1)}}
	\begin{align*}
  	\scalebox{0.8}{%
%\beginpgfgraphicnamed{diagrams/tikz/#1-external}%
\begin{tikzpicture}[STrain]
	\RainbowOne;
        \draw (0,0)--(0,-0.5);
        \node[anchor=west] at (0,-0.35) {\footnotesize$2n+2$};
	\JWPlusTwo;
\end{tikzpicture}%
%\endpgfgraphicnamed
}
	        & = i \frac{\sqrt{[n][n+2]}}{[n+1]}
	\scalebox{0.8}{%
%\beginpgfgraphicnamed{diagrams/tikz/#1-external}%
\begin{tikzpicture}[STrain]
	\STrainStrings{$n+1$}{$n+1$} \STrainOne
        \draw (0,0)--(0,-0.5);
        \node[anchor=west] at (0,-0.35) {\footnotesize$2n+2$};
	\JWPlusTwo
\end{tikzpicture}%
%\endpgfgraphicnamed
}, \\
	\scalebox{0.8}{%
%\beginpgfgraphicnamed{diagrams/tikz/#1-external}%
\begin{tikzpicture}[STrain]
	\RainbowTwo
        \draw (0,0)--(0,-0.5);
        \node[anchor=west] at (0,-0.35) {\footnotesize$2n+4$};
	\JWPlusFour
\end{tikzpicture}%
%\endpgfgraphicnamed
} 
	     &    = \frac{[2][2n+4]}{[n+1][n+2]}
        \scalebox{0.8}{%
%\beginpgfgraphicnamed{diagrams/tikz/#1-external}%
\begin{tikzpicture}[STrain]
	\STrainThreeStrings{$n+1$}{$2$}{$n+1$} \STrainOneOne
        \draw (0,0)--(0,-0.5);
        \node[anchor=west] at (0,-0.35) {\footnotesize$2n+4$};
	\JWPlusFour
\end{tikzpicture}%
%\endpgfgraphicnamed
}.
	\end{align*}
	\caption{The `box' jellyfish relations for the extended Haagerup planar algebra, from \cite{0909.4099}. (Here $n=8$ corresponds to the extended Haagerup subfactor, but everything here applies equally to the Haagerup factor if we use $n=4$.)
}
	\label{fig:box-jellyfish}
\end{figure}

Expanding out the Jones-Wenzl idempotents as a sum of individual diagrams, and moving all the non-identity terms to the right hand side, one sees that the first relation allows moving a generator `through a single string'. We can only move from the shaded side of the string to the unshaded side, and when we move through there may be more than one copy of the generator. Nevertheless, if the initial generator was at depth $2k+1$ from the surface, each of the new generators is at depth $2k$. The second relation allows moving a generator past a pair of strings (beginning and ending in unshaded regions), again at the expense of increasing the number of generators in the diagram. One can think of these relations as poor substitutes for an actual braiding, which would allow pulling generators to the outside while preserving the number of generators. Applying these moves many times, we eventually reach a linear combination of diagrams where all generators are adjacent to the outside region. See Figure \ref{fig:jellyfish} for a schematic representation of this process.
\begin{figure}[!htb]
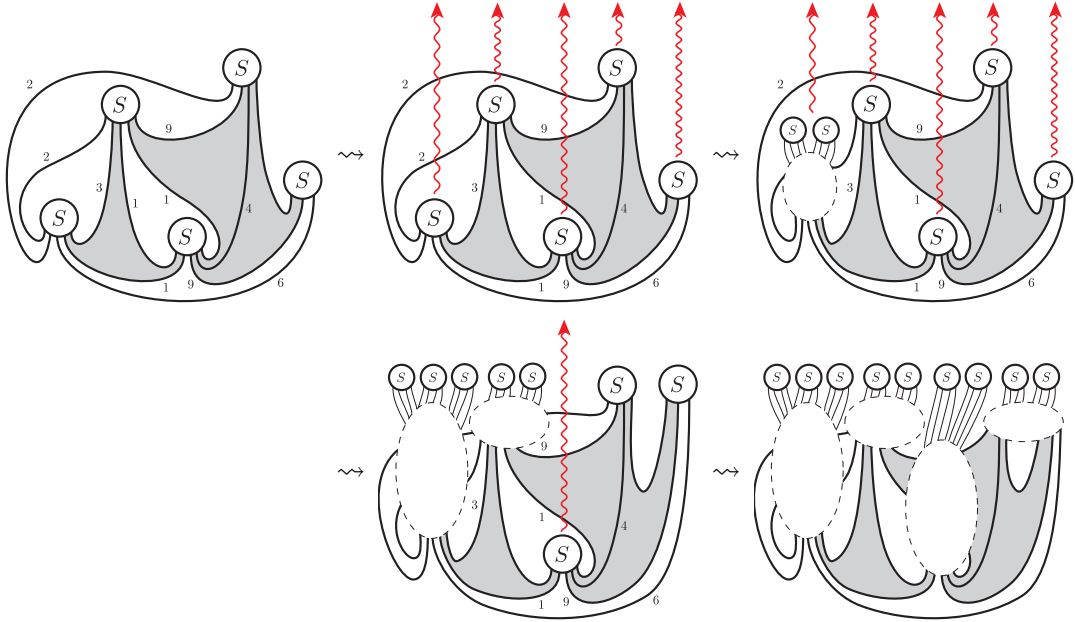

\begin{align*}
\mathfig{0.28}{jellyfish/network} & \rightsquigarrow
\mathfig{0.28}{jellyfish/network-paths} \rightsquigarrow
\mathfig{0.28}{jellyfish/network-paths-2} \\ & \rightsquigarrow
\mathfig{0.28}{jellyfish/network-paths-3} \rightsquigarrow
\mathfig{0.28}{jellyfish/network-paths-4}
\end{align*}
\caption{The initial steps of the jellyfish algorithm. The dotted ovals represent linear combinations of Temperley-Lieb diagrams. This is only a schematic illustration---to be precise, the result should be a linear combination of diagrams with various (sometimes large) numbers of copies of $S$.}
\label{fig:jellyfish}
\end{figure}

Once all the generators are in the outside region, it is easy to show that two must be connected by at least $n$ strings, and then Equation \eqref{eq:S2} allows us to replace that pair with a linear combination of diagrams with fewer generators, in such a way that the entire diagram is still in the prescribed form. Eventually we reach a diagram with no generators, which is just a collection of planar embedded circles, and evaluate this to a number.

Given that the relations described here have a parameter $n$ which can be specialized to give either the Haagerup or the extended Haagerup planar algebra, it seems interesting to understand what happens at other values of $n$. At present, we know only a number theoretic argument for why there are no other planar algebras in this family (higher values of $n$ give principal graphs with non-cyclotomic index, see \S \ref{sec:vines}). Presumably one can show that the relations are inconsistent (that is, the jellyfish algorithm can be applied in two different ways to some closed diagram, giving different answers) and so attempting to specify the planar algebra by generators and relations collapses to the zero algebra. At this point, the only proof of consistency of evaluation for $n=4$ and $n=8$ that we know is the existence of an explicit element in the appropriate graph planar algebra satisfying the relations.

\subsubsection{The Goodman-de la Harpe-Jones 3311 subfactor}
\label{sec:3311}

A 3311 subfactor was first constructed by Goodman, de la Harpe, and Jones in \cite{MR999799}.  Okamoto identified the principal graphs in \cite{MR1193932}.  Although the principal and dual principal graphs are the same, Kawahigashi proved that this subfactor is not self-dual \cite{MR1355948}.   He did this by showing that the GHJ subfactor and its dual yield different fusion rules on Bisch's list of possible $3311$ fusion rules \cite{MR1284945}. A 3311 subfactor was independently constructed via a formal embedding in \cite{MR1617550}, and identified as being the same as the GHJ subfactor in \cite[Proposition A.3]{MR1777347}.

There is a nice description of the GHJ construction from the module category perspective described in \S \ref{sec:module-categories}.  It turns out that for the subfactors of index less than $4$  the shading doesn't matter, so there are $ADE$ tensor categories as well as $ADE$ subfactors.  The $E_6$ fusion category can be thought of as module category over $A_{11}$. Choosing $X$ to be the trivial simple object in $E_6$, its internal endomorphisms  $\underline{\operatorname{Hom}}(X, X)$ gives an algebra object in $E_6$ and the corresponding subfactor is the 3311 subfactor.  From this perspective it is not difficult to see that the subfactor is not self-dual, as one of the even parts has the even part of $E_6$ as a subcategory while the other has the even part of $E_6$ as a quotient \cite[\S 5.1]{index5-part3}.

Proving that there is a unique (up to taking duals) subfactor with principal graph 3311 is somewhat difficult, as there is a $1$-parameter family of biunitary connections only two of which are flat. In unpublished work Rehren proved uniqueness directly by showing all but those two connections failed to be flat.  We gave a different argument in \cite{index5-part3} which requires less calculation to recover the same result.  We showed there that in a certain gauge choice, the eigenvalues for rotation can be read off from the connection.
This is a powerful condition necessary for flatness and computationally easy to check. Using this condition, it is easy to see that at most two points in the one-paramater family of connections can possibly be flat.  These two points yield the GHJ subfactor and its dual.

\subsubsection{The Izumi-Xu 2221 subfactor}
\label{sec:2221}

A 2221 subfactor was first constructed by Izumi in \cite{MR1832764}, using Cuntz algebras and type $\III$ factors.  Furthermore, in \cite{MR1832764} he proposed a method to construct a family of subfactors, whose even part fusion ring is a `near-group', with simple objects $G \cup \{X\}$ for some finite group $G$, with fusion rules $gX=X$, $X^2 = n' X + \sum_g g$. The first few cases, $G=\Integer/\Integer, \Integer/2\Integer$ and $\Integer/3\Integer$ with $n'=|G|$ correspond to the $A_4, E_6$ and $2221$ subfactors. Evans and Gannon in \cite{1208.1500} constructed many more such subfactors, and they and Izumi independently have proved that $n'=n-1$ or $n | n'$, where $n$ is the size of $G$.

The basic idea behind Izumi's approach is to realize the standard invariant as a collection of endomorphisms and isometries between these endomorphisms.  That is to say, look at the tensor category whose objects are endomorphisms of a fixed algebra $A$ and where a map from $\alpha$ to $\beta$ is an isometry $u$ satisfying $u \alpha(x) = \beta(x) u$.  This point of view is natural from the perspective of type $\III$ factors because any bimodule over a type $\III$ factor $M$ can be realized as coming from an endomorphism of $M$, and in this setting this tensor category is called the tensor category of sectors.  What Izumi noticed was that instead of looking at the whole factor $M$, you can instead look at a smaller subalgebra so long as it contains all the relevant isometries.  In particular, you can try to realize your standard invariant via endomorphisms of the Cuntz algebra, which is an algebra generated by certain isometries satisfying some simple relations.
This technique has proven quite powerful, especially for constructing subfactors where the even part has a relatively large group of invertible objects and where there's a single orbit of non-invertible objects.

\newcommand{\g}{\mathfrak{g}}
Independently, Xu found an alternative construction via a  conformal inclusion. Xu never published his construction, but essentially the same argument was given independently in a different language by Coquereaux--Rais--Tahri \cite{MR2742806}.  We will give a brief sketch of this argument in subfactor language.  There is a conformal inclusion of affine Lie algebras $(\hat \g_{G_2})_3\subset (\hat \g_{E_6})_1$ (see e.g. \cite{MR1424041}). Let $A$ denote the corresponding algebra object in the category of $(\g_{G_2})_3$ modules (the underlying object is the trivial plus the object of highest weight $(1,1)$). The subfactor built out of the $1-1$, $1-A$, $A-1$, and $A-A$ bimodule objects in $(\g_{G_2})_3$  has index $\frac{7+\sqrt{21}}{2}$ and principal graphs (``Haagerup with legs'')
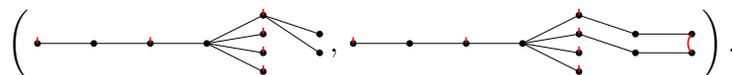
$$\left(\bigraph{bwd1v1v1v1p1p1p1v0x0x0x1p0x0x0x1duals1v1v1x2x3x4}, \bigraph{bwd1v1v1v1p1p1p1v0x0x1x0p0x0x0x1v1x0p0x1duals1v1v1x2x3x4v2x1}\right).$$
%(every vertex on the first graph is self-dual because all representations of $G_2$ are self-dual; duality on the second graph is easily determined).
Take the reduced subfactor construction with respect to one of the univalent vertices at depth 4 on the second graph, and obtain a subfactor with index $\frac{1}{2}(15+3\sqrt{21})$, having 2 invertible objects at depth 2. These ensure that there is an intermediate subfactor of index 3, and hence the other intermediate subfactor has index $\frac{1}{2}(5+\sqrt{21})$.  One can identify the principal graph of this intermediate subfactor as 2221. It is not immediately obvious that the subfactor constructed in this way is the same as Izumi's.

Later Ostrik \cite[Appendix]{1004.0665} modified Xu's construction to show that there is actually a $\Integer/2\Integer$-graded fusion category with principal graph 2221; this immediately gives a self-dual 2221 subfactor.   In addition to Xu's construction, this argument uses the characterization of fusion categories Morita equivalent to group extensions of a given fusion category, from \cite{0809.3031}. As shown in \cite[Theorem 1.0.1]{1004.0665}, the $2221$ fusion category yields the smallest dimension above $2$ of any object in a fusion category.

Finally, Han showed that all these constructions give the same subfactor, and indeed that there is a unique subfactor with principal graph 2221, in his Ph.D. thesis \cite{1102.2052}. He achieved this by identifying the only possible embedding of a 2221 planar algebra in the 2221 graph planar algebra. His approach introduced an interesting new tool for studying embeddings in graph planar algebras. If $S \in P_{n,\pm}$, then the image of $S$ in the graph planar algebra must have all its moments multiples of the empty diagram (simply because moments are a planar operation). However, in the graph planar algebra the zero box space has the same dimension as the number of even (or odd, for the shaded zero box space) vertices. The empty diagram is the uniform linear combination of the zero length loops at each vertex. Han thus set out to find all lowest weight 3-boxes $S$ in the graph planar algebra, which are rotational eigenvectors, such that $\tr{(S^n)}$ is constant on the vertices. He found a unique such 3-box, thus uniquely determining the image of $P$ in the graph planar algebra.

\subsubsection{Temperley-Lieb subfactors} \label{sec:TLsubfactors}
To say that the principal graph of a subfactor is $A_\infty$ is the same thing as saying its standard invariant is Temperley-Lieb.  But this standard invariant is non-amenable (since the norm of $A_\infty$ is $4$, but the index is larger than $4$) so we cannot  apply Popa's results to realize this standard invariant
with a hyperfinite subfactor. (But recall that Popa showed that such subfactors 
do occur in non-hyperfinite factors \cite{MR1198815}.)  Indeed, very little is known about $A_\infty$ subfactors of the hyperfinite $\II_1$; at this stage only a handful
of examples.

There is one key technique of Ocneanu's which does allow construction of Temperley-Lieb subfactors with certain indices.  Given \emph{any} biunitary connection on a pair of  bipartite graphs and a positive eigenvector for the adjacency matrices, you can still construct a subfactor called the ``flat part" \cite{ MR996454, MR1865095, MR1642584}.  This subfactor will not necessarily have the graphs you started with as principal graphs, but if the graphs are finite it will have index the square of the graph norm.  Unless this index corresponds to one of the above five examples, the subfactor constructed this way must be a Temperley-Lieb subfactor.  This should give a large supply of $A_\infty$ subfactors though because of the poverty of their standard invariants they have been little studied.

The smallest possible graph norm above $4$ comes from the graph $E_{10}$ with index $4.02642\ldots$, and it is known (by unpublished work of Haagerup, Ocneanu, and Schou \cite{schou}) that this graph has a (non-flat) connection and thus that there is a subfactor of the hyperfinite $\II_1$ with this  index.  This is the smallest known index above $4$ for an irreducible subfactor of the hyperfinite $\II_1$ factor. 

Finite depth subfactors always have indices which are algebraic integers, and it is natural to wonder whether all subfactors of the hyperfinite $\II_1$ have this property.  Bisch \cite{MR1293872} gave a negative answer to this question by constructing a subfactor of the hyperfinite $\II_1$ (coming from a non-flat connection on an infinite graph) whose index is $9/2$.

It would be of great interest to have stronger results about the classification of small index Temperley-Lieb subfactors. Any such results would have a very different flavor from the classification techniques used in the rest of the paper.  Our techniques are algebraic and look only at the structure of the standard invariant, while classifying Temperley-Lieb subfactors is an analytic question concerning the detailed structure of the hyperfinite $\II_1$ factor.

There is one question one might ask on which progress should be possible.
The norms of graphs are a dense subset of the reals in the interval
$[\sqrt{2+\sqrt{5}},\infty)$ \cite{MR0986863}. From our experience it is relatively easy to construct 
connections on pairs of graphs so one may ask the question:
\begin{question}
Is the set of indices of irreducible subfactors of the hyperfinite
II$_1$ factors dense in any interval in $\mathbb R$?
\end{question}

A probably much more difficult question is the following:
\begin{question}
Is the set of indices of Temperley-Lieb subfactors of the hyperfinite
II$_1$ factors dense in any interval in $\mathbb R$?
\end{question}

There are no published arguments ruling out any specific number above $4$ as an index value of a irreducible hyperfinite subfactor, though Popa
announced in 1990 \cite{MR1159284} that there are no such indices between $4$ and
$4.02642$ (the norm of $E_{10}$, the largest root of $x^5-9 x^4+27 x^3-31 x^2+12 x-1$).

At this stage we do not know of any irreducible hyperfinite subfactor
whose index is transcendental.

\section{Proving classification results}
\label{sec:proving}

The general approach to proving classification results is to start with some combinatorial constraints of principal graphs, describe all graphs satisfying those constraints below the given index, and then for each graph either eliminate it as a possible principal graph using some stronger constraints or construct and classify subfactors with the given principal graph.  For example, the $ADE$ classification of subfactors of index less than $4$ follows this technique.  Since the graph norm is smaller than $2$ the full list of candidates are the $ADE$ Dynkin diagrams.  One then eliminates the $D_{odd}$ and $E_7$ graphs by showing they fail some test (for example as in \ref{sec:index-leq-4}), and then explicitly constructs the other subfactors, at the same time proving that they're unique (up to complex conjugation).  Since we have already discussed constructions in the previous section, in this section we will only discuss how to enumerate and eliminate candidate graphs.

Another example of this general outline is the classification of subfactors of index less than $3+\sqrt{3}$ which was initiated by Haagerup \cite{MR1317352}.  Using Ocneanu's ``triple point obstruction'' (c.f. \S \ref{sec:ocneanu-triple-point} below), Haagerup was able to limit the combinatorial possibilities for the principal graph near the first branch point.  This gives the following classification of graph pairs which satisfy an associativity condition (c.f. \S \ref{sec:associativity} below), pass the triple point test, and have index less than $3+\sqrt{3}$:
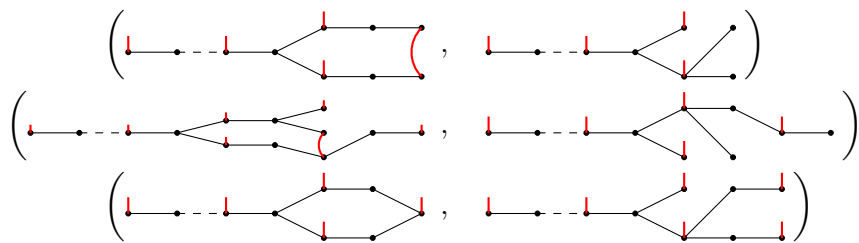
\begin{align*}
\left(\begin{tikzpicture}[baseline,scale=0.65]
\draw[fill] (0,0) circle (0.05);
\draw (0.,0.) -- (1.,0.);
\draw[fill] (1.,0.) circle (0.05);
\draw[dashed] (1.,0.) -- (2.,0.);
\draw[fill] (2.,0.) circle (0.05);
\draw (2.,0.) -- (3.,0.);
\draw[fill] (3.,0.) circle (0.05);
\draw (3.,0.) -- (4.,-0.5);
\draw (3.,0.) -- (4.,0.5);
\draw[fill] (4.,-0.5) circle (0.05);
\draw[fill] (4.,0.5) circle (0.05);
\draw (4.,-0.5) -- (5.,-0.5);
\draw (4.,0.5) -- (5.,0.5);
\draw[fill] (5.,-0.5) circle (0.05);
\draw[fill] (5.,0.5) circle (0.05);
\draw (5.,-0.5) -- (6.,-0.5);
\draw (5.,0.5) -- (6.,0.5);
\draw[fill] (6.,-0.5) circle (0.05);
\draw[fill] (6.,0.5) circle (0.05);
\draw[red, thick] (0.,0.) -- +(0,0.333333) ;
\draw[red, thick] (2.,0.) -- +(0,0.333333) ;
\draw[red, thick] (4.,-0.5) -- +(0,0.333333) ;
\draw[red, thick] (4.,0.5) -- +(0,0.333333) ;
\draw[red, thick] (6.,-0.5) to[out=135,in=-135] (6.,0.5);
\end{tikzpicture}\;\right.,&\quad\left.
\begin{tikzpicture}[baseline,scale=0.65]
\draw[fill] (0,0) circle (0.05);
\draw (0.,0.) -- (1.,0.);
\draw[fill] (1.,0.) circle (0.05);
\draw[dashed] (1.,0.) -- (2.,0.);
\draw[fill] (2.,0.) circle (0.05);
\draw (2.,0.) -- (3.,0.);
\draw[fill] (3.,0.) circle (0.05);
\draw (3.,0.) -- (4.,-0.5);
\draw (3.,0.) -- (4.,0.5);
\draw[fill] (4.,-0.5) circle (0.05);
\draw[fill] (4.,0.5) circle (0.05);
\draw (4.,-0.5) -- (5.,-0.5);
\draw (4.,-0.5) -- (5.,0.5);
\draw[fill] (5.,-0.5) circle (0.05);
\draw[fill] (5.,0.5) circle (0.05);
\draw[red, thick] (0.,0.) -- +(0,0.333333) ;
\draw[red, thick] (2.,0.) -- +(0,0.333333) ;
\draw[red, thick] (4.,-0.5) -- +(0,0.333333) ;
\draw[red, thick] (4.,0.5) -- +(0,0.333333) ;
\end{tikzpicture}\;
\right)
\\
\left(\begin{tikzpicture}[baseline,scale=0.65]
\draw[fill] (2.,0.) circle (0.05);
\draw (2.,0.) -- (3.,0.);
\draw[fill] (3.,0.) circle (0.05);
\draw[dashed] (3.,0.) -- (4.,0.);
\draw[fill] (4.,0.) circle (0.05);
\draw (4.,0.) -- (5.,0.);
\draw[fill] (5.,0.) circle (0.05);
\draw (5.,0.) -- (6.,-0.25);
\draw (5.,0.) -- (6.,0.25);
\draw[fill] (6.,-0.25) circle (0.05);
\draw[fill] (6.,0.25) circle (0.05);
\draw (6.,-0.25) -- (7.,-0.25);
\draw (6.,0.25) -- (7.,0.25);
\draw[fill] (7.,-0.25) circle (0.05);
\draw[fill] (7.,0.25) circle (0.05);
\draw (7.,-0.25) -- (8.,-0.5);
\draw (7.,0.25) -- (8.,0.);
\draw (7.,0.25) -- (8.,0.5);
\draw[fill] (8.,-0.5) circle (0.05);
\draw[fill] (8.,0.) circle (0.05);
\draw[fill] (8.,0.5) circle (0.05);
\draw (8.,-0.5) -- (9.,0.);
\draw[fill] (9.,0.) circle (0.05);
\draw (9.,0.) -- (10.,0.);
\draw[fill] (10.,0.) circle (0.05);
\draw[red, thick] (2.,0.) -- +(0,0.166667) ;
\draw[red, thick] (4.,0.) -- +(0,0.166667) ;
\draw[red, thick] (6.,-0.25) -- +(0,0.166667) ;
\draw[red, thick] (6.,0.25) -- +(0,0.166667) ;
\draw[red, thick] (8.,-0.5) to[out=135,in=-135] (8.,0.);
\draw[red, thick] (8.,0.5) -- +(0,0.166667) ;
\draw[red, thick] (10.,0.) -- +(0,0.166667) ;
\end{tikzpicture}\;\right.,&\quad\left.
\begin{tikzpicture}[baseline,scale=0.65]
\draw[fill] (2.,0.) circle (0.05);
\draw (2.,0.) -- (3.,0.);
\draw[fill] (3.,0.) circle (0.05);
\draw[dashed] (3.,0.) -- (4.,0.);
\draw[fill] (4.,0.) circle (0.05);
\draw (4.,0.) -- (5.,0.);
\draw[fill] (5.,0.) circle (0.05);
\draw (5.,0.) -- (6.,-0.5);
\draw (5.,0.) -- (6.,0.5);
\draw[fill] (6.,-0.5) circle (0.05);
\draw[fill] (6.,0.5) circle (0.05);
\draw (6.,0.5) -- (7.,-0.5);
\draw (6.,0.5) -- (7.,0.5);
\draw[fill] (7.,-0.5) circle (0.05);
\draw[fill] (7.,0.5) circle (0.05);
\draw (7.,0.5) -- (8.,0.);
\draw[fill] (8.,0.) circle (0.05);
\draw (8.,0.) -- (9.,0.);
\draw[fill] (9.,0.) circle (0.05);
\draw[red, thick] (2.,0.) -- +(0,0.333333) ;
\draw[red, thick] (4.,0.) -- +(0,0.333333) ;
\draw[red, thick] (6.,-0.5) -- +(0,0.333333) ;
\draw[red, thick] (6.,0.5) -- +(0,0.333333) ;
\draw[red, thick] (8.,0.) -- +(0,0.333333) ;
\end{tikzpicture}\;
\right)
\\
\left(\begin{tikzpicture}[baseline,scale=0.65]
\draw[fill] (0,0) circle (0.05);
\draw (0.,0.) -- (1.,0.);
\draw[fill] (1.,0.) circle (0.05);
\draw[dashed] (1.,0.) -- (2.,0.);
\draw[fill] (2.,0.) circle (0.05);
\draw (2.,0.) -- (3.,0.);
\draw[fill] (3.,0.) circle (0.05);
\draw (3.,0.) -- (4.,-0.5);
\draw (3.,0.) -- (4.,0.5);
\draw[fill] (4.,-0.5) circle (0.05);
\draw[fill] (4.,0.5) circle (0.05);
\draw (4.,-0.5) -- (5.,-0.5);
\draw (4.,0.5) -- (5.,0.5);
\draw[fill] (5.,-0.5) circle (0.05);
\draw[fill] (5.,0.5) circle (0.05);
\draw (5.,-0.5) -- (6.,0.);
\draw (5.,0.5) -- (6.,0.);
\draw[fill] (6.,0.) circle (0.05);
\draw[red, thick] (0.,0.) -- +(0,0.333333) ;
\draw[red, thick] (2.,0.) -- +(0,0.333333) ;
\draw[red, thick] (4.,-0.5) -- +(0,0.333333) ;
\draw[red, thick] (4.,0.5) -- +(0,0.333333) ;
\draw[red, thick] (6.,0.) -- +(0,0.333333) ;
\end{tikzpicture}\;\right.,&\quad\left.
\begin{tikzpicture}[baseline,scale=0.65]
\draw[fill] (0,0) circle (0.05);
\draw (0.,0.) -- (1.,0.);
\draw[fill] (1.,0.) circle (0.05);
\draw[dashed] (1.,0.) -- (2.,0.);
\draw[fill] (2.,0.) circle (0.05);
\draw (2.,0.) -- (3.,0.);
\draw[fill] (3.,0.) circle (0.05);
\draw (3.,0.) -- (4.,-0.5);
\draw (3.,0.) -- (4.,0.5);
\draw[fill] (4.,-0.5) circle (0.05);
\draw[fill] (4.,0.5) circle (0.05);
\draw (4.,-0.5) -- (5.,-0.5);
\draw (4.,-0.5) -- (5.,0.5);
\draw[fill] (5.,-0.5) circle (0.05);
\draw[fill] (5.,0.5) circle (0.05);
\draw (5.,-0.5) -- (6.,-0.5);
\draw (5.,0.5) -- (6.,0.5);
\draw[fill] (6.,-0.5) circle (0.05);
\draw[fill] (6.,0.5) circle (0.05);
\draw[red, thick] (0.,0.) -- +(0,0.333333) ;
\draw[red, thick] (2.,0.) -- +(0,0.333333) ;
\draw[red, thick] (4.,-0.5) -- +(0,0.333333) ;
\draw[red, thick] (4.,0.5) -- +(0,0.333333) ;
\draw[red, thick] (6.,-0.5) -- +(0,0.333333) ;
\draw[red, thick] (6.,0.5) -- +(0,0.333333) ;
\end{tikzpicture}\;
\right)
\end{align*}

Haagerup's actual statement is different in minor details: he doesn't explicitly show the dual data, but says that the supertransitivities of the first and third families must be 3 mod 4, and that the supertransitivity of the second family must be 5 mod 6. In fact, his paper only gave the detailed proof of the corresponding statement up to index $3+\sqrt{2}$, below which only the first graph in the first family above appears. Our results confirmed his stronger claim, of course.

Subsequently, each of these families was cut down to finitely many cases.  Haagerup immediately eliminated all but the Asaeda-Haagerup graph from the second family, using an unpublished technique (see \cite{index5-part2} for a reconstruction).  Bisch \cite{MR1625762} eliminated the third series entirely by showing that there was no consistent rules for fusion with the last vertex on the first graph. (Here's a terse reduction of his argument: Let $Y$ be the last vertex, and $X$ the generator. Working mod $2$, of course $\operatorname{dim} \operatorname{Hom}(2Y \otimes Y, Y) \equiv 0 \pmod 2$. But $2Y$ can be written as a polynomial in $X$, with constant term $1$, and $\operatorname{dim} \operatorname{Hom}(X^k \otimes Y, Y) \equiv 0 \pmod 2$ for $k>1$, by the symmetry of the graph.) Haagerup had already shown that the supertransitivity in the first family must be $3 \pmod 4$, and Asaeda-Yasuda \cite{MR2472028} then eliminated everything except Haagerup (with supertransitivity 3) and extended Haagerup (with supertransitivity 7) using a number theoretic argument which we will discuss further in \S \ref{sec:vines} below.

Our classification follows a similar outline. Via combinatorial arguments we restrict possible principal graphs to certain infinite families. These families come in two classes, called weeds and vines, described below. In Haagerup's classification, only vines appeared. These families must then be cut down to finitely many cases. This requires a specialized argument for each weed, but we introduce a general approach for any vine, subsuming the individual arguments described above for the three vines in Haagerup's classification.

\subsection{Enumerating potential principal graphs}
\label{sec:odometer}
In order to talk about classification theorems along the lines of Haagerup's theorem above, we introduce the following terminology.
We say in \cite{1007.1730} that a classification statement $\left(\Gamma_0, \Lambda, \cV, \cW\right)$ is the theorem:
\begin{quote}
Every subfactor whose principal bigraph pair is a translated extension of a fixed bigraph pair $\Gamma_0$ (and not $A_\infty$), and which has index strictly between $4$ and $\Lambda \in \Real$ has principal bigraph pair which is either
\begin{enumerate}
\item a translate (but not an extension!) of one of a certain set of bigraph pairs $\cV$, called ``vines'', or
\item a translate of an extension of one of a certain set of bigraph pairs $\cW$, called ``weeds''.
\end{enumerate}
\end{quote}

Here a translation of a bigraph pair means increasing the supertransitivity by an even amount. An extension of a bigraph pair adds new vertices and edges at strictly greater depths than the maximum depth of any vertex in the original pair. Haagerup's result above is the classification statement
\newcommand{\ao}{\left(\scalebox{0.4}{$\bigraph{bwd1duals1}$},\scalebox{0.4}{$\bigraph{bwd1duals1}$}\right)}
\newcommand{\at}{\left(\scalebox{0.4}{$\bigraph{bwd1v1duals1v1}$},\scalebox{0.4}{$\bigraph{bwd1v1duals1v1}$}\right)}
$$\left(\ao, 3+\sqrt{3}, \cV_H, \eset\right)$$
with
\begin{align*}
\cV_H & = \left \{ \left(\bigraph{bwd1v1v1v1p1v1x0p0x1v1x0p0x1duals1v1v1x2v2x1}, \bigraph{bwd1v1v1v1p1v1x0p1x0duals1v1v1x2}\right) \right . \\
	  & \qquad \left(\bigraph{bwd1v1v1v1p1v1x0p0x1v1x0p0x1p0x1v1x0x0v1duals1v1v1x2v2x1x3v1}, \bigraph{bwd1v1v1v1p1v0x1p0x1v0x1v1duals1v1v1x2v1}\right) \\
	  & \qquad \left. \left(\bigraph{bwd1v1v1v1p1v1x0p0x1v1x1duals1v1v1x2v1}, \bigraph{bwd1v1v1v1p1v1x0p1x0v1x0p0x1duals1v1v1x2v1x2}\right) \right\}
\end{align*}

Trivially, we always have classification statements of the form $\left(\Gamma_0, \Lambda, \emptyset, \{ \Gamma_0 \}\right)$, and in particular $\left(\ao, 5, \eset, \left\{ \ao \right\} \right)$. An intermediate subfactor argument \cite[\S 5]{1007.1730} \cite[Lemma 2.4]{1205.2742} shows that there are no 1-supertransitive subfactors with index between 4 and 5, so we further have the classification statement $$\left(\ao, 5, \eset, \left\{ \at \right\} \right).$$

Working combinatorially, we now proceed to refine classification statements in three ways.

\subsubsection{The odometer}
We can always take a bigraph pair out of a set of weeds, replacing it with all ways to extend by one depth, staying under the index limit. At the same time, we need to put that pair, along with any ways to extend just one graph of the pair, into the set of vines. See \cite[\S 2.3]{1007.1730} for a much more detailed explanation, including the combinatorics of enumerating all the depth one extensions. Repeatedly applying this process, we can arbitrarily increase the minimum depth of the pairs appearing in the set of weeds, at the expense of potentially greatly increasing the sizes of the sets of vines and weeds.

Furthermore, we can eliminate many weeds by applying some simple combinatorial tests which we explain below.  We call this process of repeatedly growing the list of weeds and vines, and then pruning using these combinatorial tests ``running the odometer".   In fortunate situations, eventually the odometer will eliminate all the weeds leaving only a list of vines.  This is exactly what occurred in Haagerup's classification up to index $3+\sqrt{3}$, but we are not so lucky going to index 5. Instead, careful decisions have to be made about when to stop running the odometer. We need to strike a balance between having few enough remaining weeds that they can be individually studied, but each rich enough that we can actually find obstructions for them.

\subsubsection{The associativity test} 
\label{sec:associativity}
The easiest combinatorial test which we can use to eliminate weeds just uses associativity of the fusion rules.  Recall that the principal graphs encode the fusion rules for tensoring on the right with ${}_N M_M$ and ${}_M M_N$. Combined with the dual data, we also know the fusion rules for tensoring on the left with these bimodules. Associativity then gives a condition because we can compare tensoring on the left and then right with tensoring on the right and then left.  This gives a  combinatorial obstruction for potential principal graph pairs.  In terms of the $4$-partite graph it is quite easy to understand; it says that the number of 2-step paths going half-way around the square between two specific vertices should be the same going either direction around the square.  This associativity condition is easy to see in any version of the standard invariant.  For example, in the language of paragroups, this associativity becomes the condition that biunitary matrices are square.

At first glance it appears that this associativity condition is relatively mild.  After all, it only uses properties of the principal graph, and if the two principal graphs are the same it often says nothing.  Nonetheless it is remarkably powerful as part of the odometer.  This is for three main reasons: it is easy to check, it can be checked locally if you know part of the principal graphs but not the whole thing, and it is remarkably effective in restricting graphs when used in tandem with Ocneanu's triple point obstruction below.

\subsubsection{Ocneanu's triple point obstruction}
\label{sec:ocneanu-triple-point}
Suppose that the principal graphs of a subfactor begin with an initial triple point.  The connection then contains a $3$-by-$3$ unitary matrix corresponding to this initial branch point, which satisfies certain normalization properties.  Remarkably, a simple linear algebra calculation shows that if the graphs are simple enough near the triple point then the combinatorial possibilities are extremely restricted.  This result was used to great effect by Haagerup who gives a proof in \cite{MR1317352} where he atributes it to Ocneanu.

The most general version of this triple point obstruction is a bit awkward to state (\cite[Lemmas 3.2 and 3.3]{1007.1730}).  However, it is easy to check and only requires knowing the princpal graphs two depths past the branch point.  To give the flavor of this obstruction we will state one special case.
\begin{proposition}
There are no subfactors whose principal graph and dual principal graph both begin like \scalebox{0.7}{$\smash{\bigraph{bwd1v1p1v1x0p1x0duals1v1x2}}$}.  

There are also no subfactors where both the principal and dual principal graphs begin like
 \scalebox{0.7}{$\smash{\bigraph{bwd1v1p1v1x0p0x1duals1v1x2}}$}.
\end{proposition}

This result is typical of the triple point obstruction in that it forces the principal graphs to start assymetrically.  See Lemma 6.4 of \cite{1007.1730} for several other examples of  this obstruction. Much of the power of the odometer comes from the tension between the triple point obstruction which forces principal graphs to start assymetrically and the associativity test which requires a symmetry in the number of paths. 

\subsection{Eliminating weeds}

Applying the odometer as above eventually yields a classification statement with a large number of vines and the five weeds in Figure \ref{fig:weeds}.  In this section we will explain how these remaining weeds can be eliminated on a case-by-case basis.  Three of these weeds start with triple points, while two start with quadruple points.  We will consider these groupings separately.

\begin{figure}
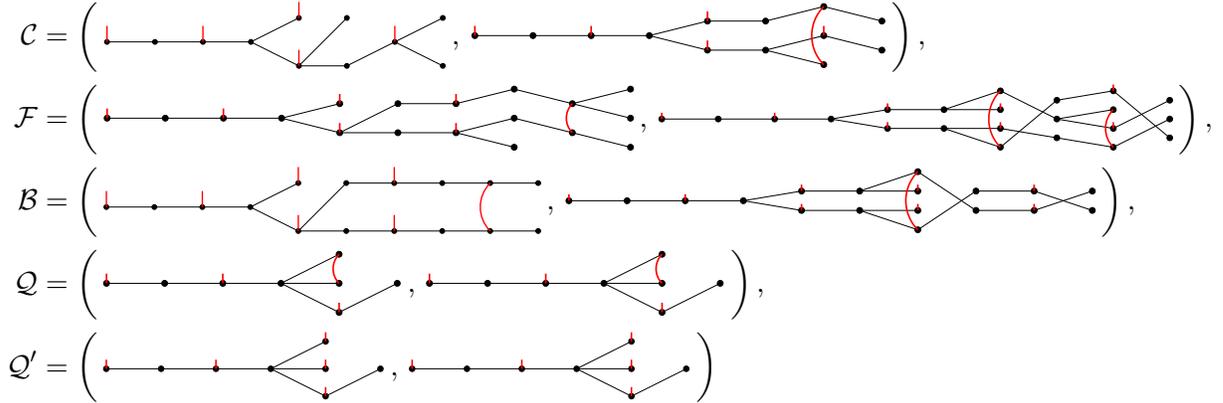

\begin{align*}
\cC &=  \left(\bigraph{bwd1v1v1v1p1v1x0p1x0v1x0v1p1duals1v1v1x2v1}, \bigraph{bwd1v1v1v1p1v1x0p0x1v1x0p1x0p0x1v0x1x0p0x0x1duals1v1v1x2v3x2x1}\right), \displaybreak[1]\\
     \cF &= \left(\bigraph{bwd1v1v1v1p1v1x0p1x0v1x0p0x1v1x0p1x0p0x1v0x1x0p0x0x1v1x0p0x1p0x1duals1v1v1x2v1x2v2x1}, \bigraph{bwd1v1v1v1p1v1x0p0x1v1x0p1x0p0x1p0x1v0x1x0x0p0x0x0x1p1x0x0x0v1x0x0p0x1x0p0x1x0p0x0x1v0x0x0x1p1x0x0x0p0x1x0x0duals1v1v1x2v4x2x3x1v3x2x1x4}\right), \displaybreak[1]\\
      \cB &= \left(\bigraph{bwd1v1v1v1p1v1x0p1x0v1x0p0x1v1x0p0x1v1x0p0x1v1x0p0x1duals1v1v1x2v1x2v2x1}, \bigraph{bwd1v1v1v1p1v1x0p0x1v1x0p1x0p0x1p0x1v0x0x0x1p1x0x0x0v1x0p0x1v0x1p1x0duals1v1v1x2v4x2x3x1v1x2}\right), \displaybreak[1]\\
          \cQ &=\left(\bigraph{bwd1v1v1v1p1p1v1x0x0duals1v1v1x3x2}, \bigraph{bwd1v1v1v1p1p1v1x0x0duals1v1v1x3x2}\right),  \displaybreak[1]\\
              \cQ' &= \left(\bigraph{bwd1v1v1v1p1p1v1x0x0duals1v1v1x2x3}, \bigraph{bwd1v1v1v1p1p1v1x0x0duals1v1v1x2x3}\right)
\end{align*}
\caption{The five weeds which must be eliminated by hand}
\label{fig:weeds}
\end{figure}

\subsubsection{Triple points}
\label{sec:triple-points}
The three weeds beginning with triple points, $\cC$, $\cF$, and $\cB$ were ruled out in \cite{index5-part2}.  Further developments have simplified the argument somewhat, and we sketch these improvements here.

First let us consider the weed $\cB$.  In \cite{index5-part2} this weed was eliminated via a somewhat tricky ad hoc calculation.  However, it turns out that this weed can be  ruled out using an earlier theorem of Popa's.  We would like to thank David Penneys for pointing this out.

\begin{theorem}\cite{MR1334479,1208.1564}
Say that a principal graph pair is \emph{stable at depth $k$} if each vertex at depth $k$ is connected to at most one vertex at depth $k+1$, and each vertex at depth $k+1$ is connected to exactly one vertex at depth $k$.

If an $n$-supertransitive subfactor with index greater than $4$ has principal graph pair which is stable at depth $k_0$ for some $k_0 > n$, then the principal graph is finite and stable at all depths $k \geq k_0$.
\end{theorem}

Applying this to the weed $\cB$ we see that any extension must end with $A_n$ tails. By Lemma 4.14 of \cite{index5-part2} these must in fact be $A_\infty$ tails, but this is impossible by results of Popa \cite{MR1356624} or Peters \cite[Theorem 6.5]{MR2679382}.

The remaining two weeds, $\cC$ and $\cF$, are ruled out using a triple point obstruction.  For convenience, we will state only the most general form of this triple point obstruction as given in \cite{1207.5090}, but for the argument it suffices to use earlier slightly weaker triple point obstructions \cite{math/1007.1158, index5-part2}.  (In fact, for the classification up to index $5$ it is only necessary to use the earliest of these results \cite{math/1007.1158}.)  In order to state this triple point obstruction we fix some notation.  Suppose that we have a subfactor planar algebra which is $n-1$ supertransitive with principal and dual principal graphs $\Gamma$ and $\Gamma'$, both beginning with a triple point.  Let $[k]$ denote the quantum number $(\nu^k - \nu^{-k})/(\nu-\nu^{-1})$ where $\nu$ is a number such that the index is $[2]^2$.  Let $r$ be the ratio of the two dimensions at depth $n$ in $\Gamma$.

\begin{theorem} \cite{1207.5090}
Suppose that one of the vertices at depth $n$ on $\Gamma'$ is $1$-valent, then
$$r+\frac{1}{r} = \frac{\lambda+\lambda^{-1}+2}{[n][n+2]}+2,$$
where $\lambda$ is an $n$-th root of unity.
\end{theorem}

The number $\lambda$ is an eigenvalue for the action of rotation in the planar algebra.  This theorem is proved using a mix of connections and planar algebras, following an argument in \cite{index5-part3} and relaxes some of the conditions from the earlier slightly weaker theorem in \cite{math/1007.1158}.

Since $\lambda$ is a root of unity, we know that $-2 \leq \lambda + \lambda^{-1} \leq 2$.  Hence this triple point obstruction implies a certain inequality in which $\lambda$ does not appear.  This inequality is quite useful on its own, and it can be proved directly using only connections as in \cite{index5-part2}.

This triple point obstruction has a somewhat different flavor than Ocneanu's.  It only applies in the specific situation where one of the vertices on $\Gamma'$ is $1$-valent, but when it does apply it gives a strong restriction.  Furthermore, applying this triple point obstruction is more subtle, because you need to know the supertransitivity, the index, and the value of $r$.  This makes it somewhat surprising that it can be applied to weeds at all, because weeds have arbitrary supertransitivity, and even if you fix the supertransitivity you typically can't work out either the index or the value of $r$.  However, in the cases of $\cC$ and $\cF$ we have a slight simplification: if you fix the index and the supertransitivity then you can work out $r$ explicitly.  This calculation of $r$ is somewhat delicate and requires using the dual data.  Once you work out $r$, the inequality rules out all but finitely many supertransitivities, and the full version rules out the remaining few examples.

\subsubsection{Quadruple points}
\label{sec:quadruple}
In \cite{index5-part3}, we showed that the only subfactors with principal graph represented by either of the weeds $\mathcal{Q}$ or $\mathcal{Q}'$ are the GHJ 3311 subfactor and its dual. There, we analyzed possible connections on a graph whose branch point is a quadruple point adjacent to two univalent vertices. We were able to prove that the dimensions of the two other vertices adjacent to this quadruple point must be equal, and as an easy corollary the index on an $n$-supertransitive extension of either $\mathcal{Q}$ or $\mathcal{Q}'$ is equal to the index of the $nn11$ graph. A slight elaboration of the techniques of \cite{1004.0665} (and see below) shows that there is a number theoretic obstruction eliminating all supertransitivies except  $n =2, 3$ or $4$. The even cases are not relevant, and a straightforward graph enumeration argument shows that the only $3$-supertransitive extension of $\mathcal{Q}$ or $\mathcal{Q}'$ with index $3+\sqrt{3}$ (the index of $3311$) is the $3311$ principal graph itself.  As explained above in \S \ref{sec:3311}, this subfactor is unique up to duality.

\subsection{Eliminating vines}
\label{sec:vines}

Now that we have eliminated all the weeds we need to consider the long list of vines remaining.  In Haagerup's classification, each of the vines was eliminated using a different technique.  Bisch used a fusion rules obstruction, Haagerup a triple point obstruction, and Asaeda-Yasuda number theory.  Rather than ad hoc techniques applied to each vine individually, it turns out that there is a number theoretic argument which eliminates all but finitely many graphs coming from a given vine in a uniform manner.  Indeed, number theory remains the only known technique for ruling out several of the vines (including the Haagerup vine).

The appearance of number theory is somewhat unexpected and mysterious.  The first application of number theory to subfactors is that all dimensions of bimodules occuring in finite depth subfactors must be algebraic integers (as they are eigenvalues of integer matrices).  This is already enough to eliminate certain graphs like $D_5$. Indeed  along with a result of Kronecker's \cite{an:053.1389cj} it gives a quick proof that finite depth subfactors of index less than $4$ must have index $4\cos^2\pi/n$.  Namely, if the index is $(q+q^{-1})^2$ then $q$ must be an algebraic integer all of whose Galois conjugates have size $1$, which by Kronecker's theorem forces $q$ to be a root of unity and the index to be $4\cos^2\pi/n$.  In addition to the condition that indices of finite depth subfactors are algebraic integers, several more sophisticated number theoretic techniques have been introduced.

\subsubsection{Cyclotomicity}
De Boer and Goeree \cite{MR1120140} first observed that certain invariants of conformal field theories always lie in a cyclotomic field.  Coste and Gannon \cite{MR1266785} clarified this result, showing that the entries of the $S$ matrix for any modular tensor category lie in a cyclotomic field.  These results are not immediately applicable to subfactors, because the tensor categories which appear in subfactor theory are typically not modular.  Nonetheless, from any spherical fusion category you can produce a modular category by taking the Drinfel'd center.  Etingof-Nikshych-Ostrik \cite{MR2183279} showed that the Coste-Gannon result for $Z(\cC)$ implies that the dimensions of objects in $\cC$ are themselves cyclotomic.  This immediately shows that the dimensions of all even depth objects in finite depth subfactors are cyclotomic integers, and in particular the index is too. Moreover, this is true not just for the Frobenius-Perron dimensions, but for any dimension function. Every multiplicity free eigenvector of the adjacency matrix of the principal graph (normalized to give 1 on the trivial bimodule) gives a dimension function. 

The key step in Coste and Gannon's argument is to prove that the entries of the $S$ matrix lie in an abelian extension of $\mathbb{Q}$, because the Galois group can be thought of  either as permuting the rows or as permuting the columns and thus must be commutative.  Then Kronecker-Weber \cite{an:053.1389cj,MR1554698}
 tells you that these numbers lie in a cyclotomic field.  This argument, although beautiful, is not very constructive.  Unpublished work of Siu-Hung Ng gives another proof that the dimensions of objects in a modular tensor category lie in a cyclotomic field, and his construction is more explicit.  The dimensions are traces of the action of rotation on a certain space, and thus must be sums of roots of unity.

Asaeda was the first to apply cyclotomicity to subfactors. She proved in \cite{MR2307421} that the $n$-supertransitive translates of the Haagerup and extended Haagerup principal graphs did not have cyclotomic index for $7<n\leq55$. Subsequently  
Asaeda and Yasuda \cite{MR2472028} showed that none of the graphs
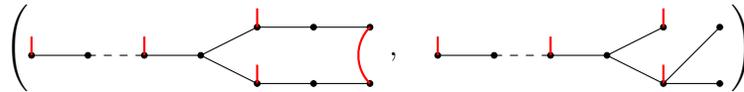
\begin{align*}
\left(\begin{tikzpicture}[baseline,scale=0.75]
\draw[fill] (0,0) circle (0.05);
\draw (0.,0.) -- (1.,0.);
\draw[fill] (1.,0.) circle (0.05);
\draw[dashed] (1.,0.) -- (2.,0.);
\draw[fill] (2.,0.) circle (0.05);
\draw (2.,0.) -- (3.,0.);
\draw[fill] (3.,0.) circle (0.05);
\draw (3.,0.) -- (4.,-0.5);
\draw (3.,0.) -- (4.,0.5);
\draw[fill] (4.,-0.5) circle (0.05);
\draw[fill] (4.,0.5) circle (0.05);
\draw (4.,-0.5) -- (5.,-0.5);
\draw (4.,0.5) -- (5.,0.5);
\draw[fill] (5.,-0.5) circle (0.05);
\draw[fill] (5.,0.5) circle (0.05);
\draw (5.,-0.5) -- (6.,-0.5);
\draw (5.,0.5) -- (6.,0.5);
\draw[fill] (6.,-0.5) circle (0.05);
\draw[fill] (6.,0.5) circle (0.05);
\draw[red, thick] (0.,0.) -- +(0,0.333333) ;
\draw[red, thick] (2.,0.) -- +(0,0.333333) ;
\draw[red, thick] (4.,-0.5) -- +(0,0.333333) ;
\draw[red, thick] (4.,0.5) -- +(0,0.333333) ;
\draw[red, thick] (6.,-0.5) to[out=135,in=-135] (6.,0.5);
\end{tikzpicture}\;\right.,&\quad\left.
\begin{tikzpicture}[baseline,scale=0.75]
\draw[fill] (0,0) circle (0.05);
\draw (0.,0.) -- (1.,0.);
\draw[fill] (1.,0.) circle (0.05);
\draw[dashed] (1.,0.) -- (2.,0.);
\draw[fill] (2.,0.) circle (0.05);
\draw (2.,0.) -- (3.,0.);
\draw[fill] (3.,0.) circle (0.05);
\draw (3.,0.) -- (4.,-0.5);
\draw (3.,0.) -- (4.,0.5);
\draw[fill] (4.,-0.5) circle (0.05);
\draw[fill] (4.,0.5) circle (0.05);
\draw (4.,-0.5) -- (5.,-0.5);
\draw (4.,-0.5) -- (5.,0.5);
\draw[fill] (5.,-0.5) circle (0.05);
\draw[fill] (5.,0.5) circle (0.05);
\draw[red, thick] (0.,0.) -- +(0,0.333333) ;
\draw[red, thick] (2.,0.) -- +(0,0.333333) ;
\draw[red, thick] (4.,-0.5) -- +(0,0.333333) ;
\draw[red, thick] (4.,0.5) -- +(0,0.333333) ;
\end{tikzpicture}\;
\right)
\end{align*}
have cyclotomic index except when the supertransitivity is $3$ or $7$, and hence these graphs cannot be the principal graph of a subfactor.

Subsequently, using rather different methods, Calegari, Morrison, and Snyder \cite{1004.0665} gave the following two results:
\begin{theorem}[Theorems 1.0.3 and 1.0.6 from \cite{1004.0665}] \label{theorem:graphs}
Let $\Gamma$ be a connected
graph with $|\Gamma|$ vertices. Fix a vertex $v$ of $\Gamma$, and
let $\Gamma_n$ denote the  sequence of graphs obtained by adding a $2$-valent tree
of length $n - |\Gamma|$ to $\Gamma$ at $v$ (see Figure \ref{fig:gamma-family}).
There exists an effective constant $N_1(\Gamma)$ such that for all $n \ge N_1(\Gamma)$, either:
\begin{enumerate}
\item  all the eigenvalues of the adjacency matrix $M_n$ are of the form
$\zeta + \zeta^{-1}$ for some root of unity $\zeta$, and the graphs $\Gamma_n$
are the  Dynkin diagrams $A_n$ or $D_n$, or
\item  the largest eigenvalue $\lambda$ of the adjacency matrix $M_n$ is greater than $2$,
and the field $\mathbb{Q}(\lambda^2)$ is not abelian. Thus $\Gamma_n$ is not the principal graph of a subfactor.
\end{enumerate}
\end{theorem}
\begin{figure}[!ht]
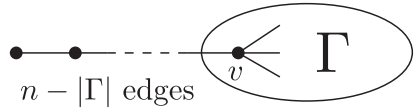
%
$$\mathfig{0.4}{Gamma_family}$$
\caption{The family of graphs $\Gamma_n$.}
\label{fig:gamma-family}
\end{figure}
The effective constant here is actually rather hard to compute, and tends to be very large!  As a result, for applications we use the following theorem where the effective constant is smaller and easier to compute.

\begin{theorem}[Theorem 10.0.1 from \cite{1004.0665}]
\label{theorem:graphs2}
With the same hypotheses and notation as in the previous theorem, there is an effective constant $N(\Gamma)$ such that for all $n \ge N$, either
\begin{enumerate}
\item the graphs $\Gamma_n$ are the Dynkin diagrams $A_n$ or $D_n$, or 
\item  there exists at least one eigenvalue
$\lambda$ of $M_n$ of multiplicity one such
that  $\mathbb{Q}(\lambda^2)$ is not abelian, and hence $\Gamma_n$ is not the principal graph of a subfactor.
\end{enumerate}
\end{theorem}

These arguments use results of Cassels and Loxton \cite{MR0246852, MR0309896}, who consider a height function $\mathcal{M}$ on cyclotomic integers and show that any cyclotomic integer with small height must be a sum of a small number of roots of unity.  The second theorem also uses a result of Gross, Hironaka, and McMullen \cite{MR2516970} to get an explicit handle on eigenvalues of the adjacency matrix of the form $\zeta+\zeta^{-1}$ where $\zeta$ is a root of unity.

In this second theorem, the constant $N(\Gamma)$ is relatively easily computable, and tends to be small. In \cite{index5-part4} Penneys and Tener gave algorithms to compute (upper bounds for) $N$, and used this theorem to reduce all the vines discussed in \S \ref{sec:odometer} to finitely many cases. The highest $N$ they have to consider is 333. Checking the finitely many cases below these bounds takes some more work. In fact, the vast majority still don't have cyclotomic index, suggesting that the bounds in Theorem \ref{theorem:graphs2} tend to be far from sharp. To show that individual indexes are not cyclotomic, it suffices to find the minimal polynomial, then identify a prime so that the factorization of the minimal polynomial module that prime has factors with different degrees. Computing the minimal polynomials eventually becomes computationally intensive, but they  found a shortcut: the ratio of the minimal polynomial to the characteristic polynomial of the square of the adjacency matrix has only a few prime factors, and this ratio tends to be periodic in $n$ with quite small period. This observation allows them to guess and then verify minimal polynomials for each graph in the finite family $\{\Gamma_n\}_{n < N(\Gamma)}$.

Finally, the authors of \cite{index5-part4} find the seven graphs coming from the Haagerup, extended Haagerup, Asaeda-Haagerup, and Izumi-Xu subfactors along with another 15 graphs which they eliminate using the obstruction of the next section. (The graphs coming from the Goodman-de la Harpe-Jones 3311 subfactor had arisen from a weed with a quadruple point, and were dealt with in \cite{index5-part3}, described in \S \ref{sec:quadruple} above, so didn't appear in their treatment of vines.)

Using number theory alone \cite{1004.0665} also gives an independent proof that any finite depth subfactor of index less than $4+\frac{10}{33}$ has index $3+2 \cos (\frac{2\pi}{7})$ or $\frac{5+\sqrt{13}}{2}$.  As we know from the main classification theorem, the assumption of finite depth is unnecessary and the former index is impossible, but nonetheless it's interesting to know that number theory alone gives such strong restrictions.  This proof uses Cassels-Loxton techniques, but requires considerable care (and some additional ideas based on work of A.J. Jones and Conway \cite{MR0224587, MR0422149}).  Furthermore, the same technique gives even stronger results restricting dimensions of objects in fusion categories.

\subsubsection{Ostrik's d-numbers}
Another number theoretic restriction on candidate subfactors was given by Ostrik in \cite{0810.3242}. A d-number is an algebraic integer which generates an ideal in the ring of algebraic integers invariant under the action of the absolute Galois group.
Ostrik showed that every formal codegree of a fusion category was a d-number. A formal codegree for a fusion category $\cC$ is the quantity $\sum_{V \in \operatorname{Irr}(\cC)} ||d(V)||^2$ for some simple representation $d$ of the fusion ring (the 1-dimensional representations are called dimension functions, but we can also talk about formal codegrees for other simple representations when the fusion ring is non-abelian). In particular, the formal codegree coming from the Frobenius-Perron dimension function is exactly the global dimension of the fusion category. In the subfactor literature, the global dimension of the fusion category of $N-N$ bimodules of a finite depth subfactor is often called the global index, because it is the index of the associated asymptotic inclusion \cite{MR996454}. 

Ostrik also gave a very explicit characterization: an algebraic integer $z$ with minimal polynomial $\sum_{i=0}^n a_{n-i} x^i$ and $a_0=1$ is a d-number if and only if $a_i^n$ is divisible by $a_n^i$ for each $i$. This immediately gives an obstruction to certain graphs being principal graphs of subfactors. Penneys and Tener used this to rule out the last 15 candidate principal graphs with index below $5$.

\section{Index 5 and beyond}
We begin with Figure \ref{fig:MapOfSubfactors} which summarizes the known small-index subfactors, the regions in which we have classification results, and indicates some candidate principal graphs which may be realized by currently unknown subfactors.

\begin{figure}[!ht]
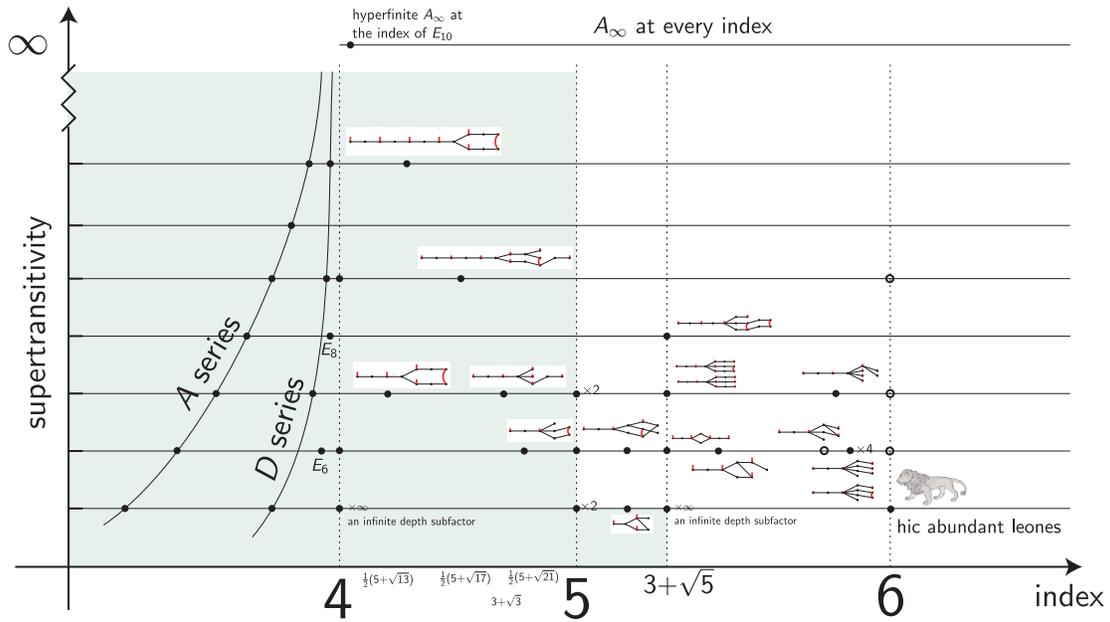

$$\mathfig{1}{MapOfSubfactors-monochrome}$$
\caption{The map of low index subfactors. In the shaded regions we have classification results. Filled dots show known subfactors. Open dots indicate candidate principal graphs, but are not exhaustive.}
\label{fig:MapOfSubfactors}
\end{figure}

\subsection{Index exactly 5}
At index exactly 5 there are seven subfactor planar algebras, all of which are group-subgroup subfactors or duals of group-subgroup subfactors.   (See \cite{MR1738515}, and \cite{MR1920326} for when two group-subgroup subfactors are equivalent.)  The group-subgroup subfactors at index $5$ are
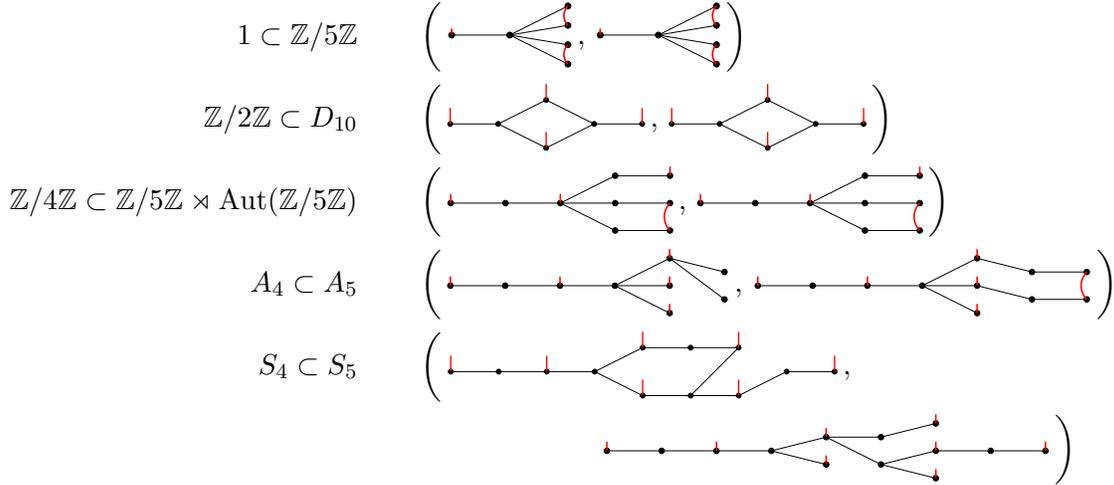
\begin{align*}
1 \subset \Integer/5\Integer & \qquad \left(\bigraph{bwd1v1p1p1p1duals1v2x1x4x3}, \bigraph{bwd1v1p1p1p1duals1v2x1x4x3}\right) \\
\Integer/2\Integer \subset D_{10} & \qquad \left(\bigraph{bwd1v1p1v1x1v1duals1v1x2v1}, \bigraph{bwd1v1p1v1x1v1duals1v1x2v1}\right) \\
\Integer/4\Integer \subset \Integer/5\Integer \rtimes \operatorname{Aut}(\Integer/5\Integer) & \qquad \left(\bigraph{bwd1v1v1p1p1v1x0x0p0x1x0p0x0x1duals1v1v2x1x3},\bigraph{bwd1v1v1p1p1v1x0x0p0x1x0p0x0x1duals1v1v2x1x3}\right) \\
A_4 \subset A_5 & \qquad \left(\bigraph{bwd1v1v1v1p1p1v0x0x1p0x0x1duals1v1v1x2x3}, \bigraph{bwd1v1v1v1p1p1v0x1x0p0x0x1v1x0p0x1duals1v1v1x2x3v2x1}\right) \\
S_4\subset S_5 & \qquad \left(\bigraph{bwd1v1v1v1p1v1x0p0x1v1x0p1x1v1x0v1duals1v1v1x2v1x2v1}, \right. \\
& \qquad\qquad\qquad\qquad
\left. \bigraph{bwd1v1v1v1p1v0x1p0x1v1x0p1x0p0x1v0x1x0v1duals1v1v1x2v1x2x3v1}\right)
\end{align*}
and the first three of these are self-dual and the last two are not.

This result has not yet appeared in the literature, but hopefully will soon. The combinatorial work done by the odometer cuts down the possibilities to the 5 cases that actually exist, and the following three graph pairs.
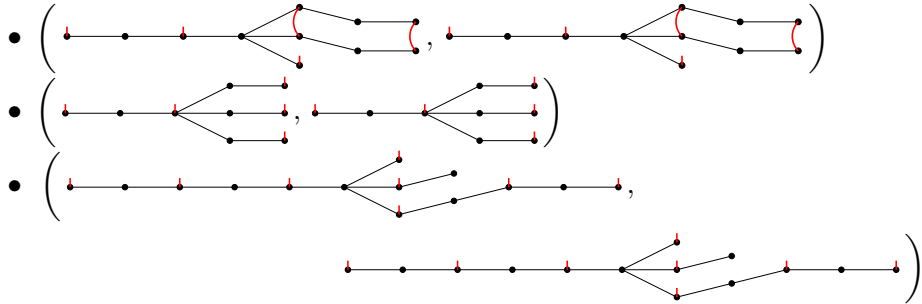
\begin{itemize}
\item $\left(\bigraph{bwd1v1v1v1p1p1v0x1x0p0x0x1v1x0p0x1duals1v1v1x3x2v2x1}, \bigraph{bwd1v1v1v1p1p1v0x1x0p0x0x1v1x0p0x1duals1v1v1x3x2v2x1}\right)$
\item $\left(\bigraph{bwd1v1v1p1p1v1x0x0p0x1x0p0x0x1duals1v1v1x2x3}, \bigraph{bwd1v1v1p1p1v1x0x0p0x1x0p0x0x1duals1v1v1x2x3}\right)$
\item $\displaystyle \begin{gathered}[t]
 \left(\bigraph{bwd1v1v1v1v1v1p1p1v1x0x0p0x1x0v1x0v1v1duals1v1v1v1x2x3v1v1}, \right. \qquad\qquad\qquad\qquad\qquad\\
\qquad\qquad\qquad\qquad\qquad \left.\bigraph{bwd1v1v1v1v1v1p1p1v1x0x0p0x1x0v1x0v1v1duals1v1v1v1x2x3v1v1}\right)
       \end{gathered}$
\end{itemize}
The first two pairs don't have biunitary connections, while the even part of the third has a multiplicity free eigenvalue, and hence dimension function, for which the formal codegree does not lie in the field generated by the Frobenius-Perron dimensions, contrary to an (unpublished) result of Ostrik's.

For the first three group-subgroup subfactors, Izumi's Goldman type result from \cite{MR1491121} shows that there is a unique subfactor planar algebra with the given principal graph. One can show uniqueness for the $A_4 \subset A_5$ and $S_4 \subset S_5$ subfactors by demonstrating that there is a unique gauge equivalence class of biunitary connections.

\subsection{Indexes between 5 and \texorpdfstring{$3+\sqrt{5}$}{3+sqrt(5)}}
Above index $4$, the first composite index is $3+\sqrt{5}$. Thus $3+\sqrt{5}$ is a natural upper bound to consider.  In this range there are two known subfactors which arise from quantum groups.  These two examples both have index the largest root of $x^3-6x^2+5x-1$, which is approximately $5.04892$, with principal graphs
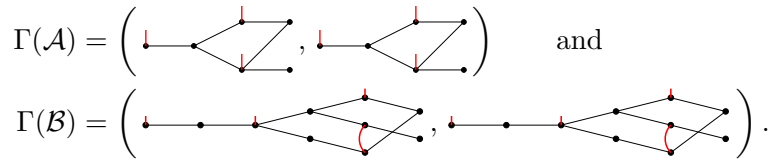
\begin{align*}
\Gamma(\cA) & = \left(\bigraph{bwd1v1p1v1x0p1x1duals1v1x2}, \bigraph{bwd1v1p1v1x0p1x1duals1v1x2}\right) \qquad \text{and} \\
\Gamma(\cB) & = \left(\bigraph{bwd1v1v1p1v1x0p0x1p0x1v0x1x0p1x0x1duals1v1v2x1x3}, \bigraph{bwd1v1v1p1v1x0p0x1p0x1v0x1x0p1x0x1duals1v1v2x1x3}\right).
\end{align*}
These come from the 3-dimensional representations of $SU(2)$ and $SU(3)$ at a $14$-th root of unity \cite{MR1470857}. In \cite{1205.2742}, Morrison and Peters have proved that these quantum group subfactors are the only subfactors with these principal graphs. Moreover, they prove that $\Gamma(\cA)$ is the only possible 1-supertransitive principal graph in the interval $(5, 3+\sqrt{5})$.

In this interval, there is computer evidence that there are no other finite depth subfactors. We can show that the principal graph pair of any other subfactor has at least 38 vertices. We also know that there is a (tiny!) gap, with no subfactors besides $A_\infty$ with index in the interval $(5,5.004)$.
 
 \begin{conj}
 The only two irreducible subfactor planar algebras (except those with principal graph $A_\infty$) with index in $(5, 3+\sqrt{5})$ are the quantum group subfactors described above.
 \end{conj}

\subsection{Index \texorpdfstring{$3+\sqrt{5}$}{3+sqrt(5)}}
The index $3+\sqrt{5}$ is a product of smaller allowed index values $2$ and $4 \cos^2(\pi/5)$, and as a result there are many interesting examples of subfactors at this index. The full list is as yet unknown. Perhaps most interesting is the infinite depth Fuss-Catalan subfactor, obtained as the free product $A_3 \ast A_4$ of subfactors at index $2$ and $4 \cos^2(\pi/5)$.  Its principal graph begins as
$$ \left( \mathfig{0.45}{fc1}, \mathfig{0.45}{fc2} \right) $$
and continues periodically.
(This is not actually the first instance of this phenomenon; the $D_\infty^{(1)}$ subfactor at index 4 is the free product $A_3 \ast A_3$, and many features of the index $3+\sqrt{5}$ Fuss-Catalan subfactor are analogous.)

There is also of course the tensor product $A_3 \tensor A_4$, with principal graph $$\left(\bigraph{bwd1v1p1p1v1x1x0duals1v1x2x3},\bigraph{bwd1v1p1p1v1x1x0duals1v1x2x3}\right).$$ Next, there are tantalizing hints of a sequence of subfactors interpolating between the tensor product and the free product; Bisch and Haagerup constructed (in unpublished notes on compositions of $A_3$ with $A_4$) a subfactor with principal graph
$$\left(\bigraph{bwd1v1p1v1x0p0x1v1x0p1x0p1x0p0x1v0x1x0x1v1duals1v1x2v1x2x4x3v1}, \bigraph{bwd1v1p1v1x0p1x0v1x0p1x0v1x1duals1v1x2v1x2}\right),$$ and recent unpublished work of Izumi has shown the existence of a subfactor with principal graph 
$$\left(\bigraph{bwd1v1p1v1x0p0x1v1x0p1x0p0x1v1x0x0p0x0x1v1x0p1x0p1x0p0x1p0x1v0x1x0x1x0v1duals1v1x2v1x3x2v1x2x4x3x5v1}, \bigraph{bwd1v1p1v1x0p1x0v1x0v1p1v1x0p1x0v1x1duals1v1x2v1v1x2}\right).$$  We note that the finite depth $D_{n}^{(1)}$ subfactors at index 4 should be considered analogues, interpolating between $A_3 \tensor A_3 = D_4^{(1)}$ and $A_3 \ast A_3 = D_\infty^{(1)}$.

Two other extremely interesting subfactors at index $3+\sqrt{5}$ have principal graph $3333$, with the group of invertible bimodules being $\Integer/4$ or $\Integer/2 \times \Integer/2$ (we will refer to these as $3^G$ where $G$ is the group of invertibles). These subfactors have been constructed by Izumi (unpublished), as special cases of a construction which potentially gives a $3^G$ subfactor for many abelian groups $G$ (this was also discussed in \S \ref{Haagerup}). Morrison and Penneys subsequently constructed a $4442$ subfactor \cite{1208.3637}, and gave an alternative construction of the $3^{\Integer/2 \times \Integer/2}$ subfactor, using a combination of the jellyfish algorithm and quadratic tangles techniques. Izumi then identified this 4442 subfactor as an equivariatization of the $3^{\Integer/2 \times \Integer/2}$ subfactor.

As we will explain below there are several ``wild'' phenomenon at index $6$ which show that there is no simple classification of subfactors at that index.  Since $6$ is also a composite index, it seems plausible that similar wild behavior takes place at index $3+\sqrt{5}$.  So far, however, there are no results in this direction.

\subsection{The onset of wildness at index 6} \label{Index6}
At index $6$ a full classification is not a reasonable goal. 
There are three ways in which we might fail to achieve a satisfactory classification. First, there could
be such a plethora of finite depth planar algebras that one could not hope to organise
them. Second, the same could be true of infinite depth planar algebras. And third,
there could be a huge family of distinct subfactors sharing the same planar algebra. 
All three of these ``wild" phenomena occur at index 6. The reason for all of these failures is 
that $6=2\times 3$ and the free product $(\mathbb Z/{2\mathbb Z})*({\mathbb Z}/{3\mathbb Z})$ is essentially a free group.

Unlike amenable groups, free groups can be made to act in a huge variety of
ways by outer automorphisms on II$_1$ factors. The construction of Bisch-Haagerup \cite{MR1386923}
takes an action on $M$ of the free product of finite groups $G$ and $H$ and considers the
subfactor $M^G\subset M\rtimes H$ (of index $|G||H|)$. (There is a planar algebraic description of this construction in  \cite{MR2523334}.) The outer conjugacy class of
the action of $G*H$ is an invariant of this subfactor, so that we obtain a different subfactor of index 6, with different
planar algebra, \emph{for every group which is a quotient of $(\mathbb Z/{2\mathbb Z})*({\mathbb Z}/{3\mathbb Z})$}!

Even worse is the failure of planar algebras to classify subfactors.
The tool for understanding this is the ``fundamental group" of a subfactor (nothing to do with any $\pi_1$):
given a subfactor $N\subset M$ and a projection $p$ in $N$ one forms the reduced subfactor
$pNp\subset pMp$. It is routine to show that the planar algebra does not change under this reduction process.
(After further tensoring with matrix algebras, the set of all traces of such projections forms a multiplicative subgroup
of $\mathbb R^+$ called the fundamental group of the subfactor.)
But the isomorphism type of the subfactor may change under reduction, even when $M$ (and hence all other factors involved, by \cite{MR0454659}) is finite.
It was Connes who first used Kazhdan's property T \cite{MR0209390} to show that the fundamental group of a $II_1$ factor may be only countable \cite{MR587372}
and Popa's techniques yielded total control in some cases over the fundamental group. 
Thus in  \cite{MR2314611}, by using these property T techniques applied to quotients of $(\mathbb Z/{2\mathbb Z})*({\mathbb Z}/{3\mathbb Z})$,
Bisch, Nicoara, and Popa constructed subfactors with index 6 with countable and trivial fundamental groups, and hence uncountable families of subfactors with index 6.

% ----------------------------------------------------------------
\newcommand{\urlprefix}{}
\bibliographystyle{alpha}
\bibliography{../../bibliography/bibliography}
% ----------------------------------------------------------------

\end{document}